\documentclass[10pt,reqno]{amsart}

\usepackage{color}
\usepackage{amsmath, amsthm, amssymb}
\usepackage{amsfonts}
\usepackage[ansinew]{inputenc}
\usepackage[dvips]{epsfig}
\usepackage{graphicx}
\usepackage[english]{babel}
\usepackage{hyperref}

\usepackage{cite}
\usepackage{graphicx}
\usepackage{amscd}
\usepackage{color}
\usepackage{bm}
\usepackage{enumerate}

\usepackage{verbatim}
\usepackage{hyperref}
\usepackage{amstext}
\usepackage{latexsym}
\theoremstyle{plain}
\newtheorem{theorem}{Theorem}[section]

\newtheorem{corollary}[theorem]{Corollary}
\newtheorem{proposition}[theorem]{Proposition}
\newtheorem{lemma}[theorem]{Lemma}
\newtheorem{remark}[theorem]{Remark}

\numberwithin{theorem}{section}
\numberwithin{equation}{section}

\newcommand{\average}{{\mathchoice {\kern1ex\vcenter{\hrule height.4pt
width 6pt depth0pt} \kern-9.7pt} {\kern1ex\vcenter{\hrule
height.4pt width 4.3pt depth0pt} \kern-7pt} {} {} }}

\def\R{\mathbb{R}}
\def\loc{\text{loc}}

\renewcommand{\a }{\alpha }
\renewcommand{\b }{\beta }
\renewcommand{\d}{\delta }

\newcommand{\D }{\Delta }

\newcommand{\e }{\varepsilon }
\newcommand{\g }{\gamma}

\newcommand{\G }{\Gamma}
\renewcommand{\l }{\lambda }
\renewcommand{\L }{\Lambda }
\newcommand{\n }{\nabla }
\newcommand{\vp }{\varphi }
\newcommand{\rh }{\rho }

\newcommand{\s }{\sigma }

\renewcommand{\t }{\tau }

\renewcommand{\th }{\theta }

\renewcommand{\o }{\omega }
\renewcommand{\O }{\Omega }

\newcommand{\ov}{\overline}

\newcommand{\be}{\begin{equation}}
\newcommand{\ee}{\end{equation}}

\newcommand{\de}{\partial}

\newcommand{\ti}{\widetilde}

\newcommand{\ra}{{\rangle}}
\newcommand{\la}{{\langle}}

\renewcommand{\k}{\kappa}

 \usepackage{mathrsfs}

\newcommand{\scrK }{\mathscr{K}}

\newcommand{\N}{\mathbb{N}}

\newcommand{\cH}{{\mathcal H}}

\newcommand{\cK}{{\mathcal K}}
\newcommand{\cL}{{\mathcal L}}
\newcommand{\cM}{{\mathcal M}}

\newcommand{\cS}{{\mathcal S}}

\newcommand{\cV}{{\mathcal V}}

\newcommand{\M}{\cM}

\renewcommand{\epsilon}{\varepsilon}

\newcommand{\Ds}{ (-\D)^s}

\begin{document}

\title 
{Regularity estimates  for nonlocal  Schr\"odinger equations}

\author[Mouhamed  Moustapha Fall]
{Mouhamed  Moustapha Fall}
\address{M.M.F.: African Institute for Mathematical Sciences in Senegal, 
KM 2, Route de Joal, B.P. 14 18. Mbour, S\'en\'egal}
\email{mouhamed.m.fall@aims-senegal.org, mouhamed.m.fall@gmail.com}

\thanks{ 
 The  author's work is supported by the Alexander 
von Humboldt foundation. Part of this work was done while he was visiting
the Goethe University in Frankfurt am Main during   August-September 2017 and he thanks the Mathematics department for the kind hospitality. The author is  grateful to Xavier Ros-Oton, Tobias Weth and Enrico Valdinoci for  their availability and for the many useful discussions  during the preparation of this work.     }


 \begin{abstract}
   \noindent
We are concerned with    H\"older regularity estimates   for weak solutions $u$ to  nonlocal    Schr\"odinger equations subject to exterior Dirichlet conditions in an open set $\Omega\subset \mathbb{R}^N$. The class of nonlocal operators considered here  are  defined, via Dirichlet forms,  by  symmetric kernels $K(x,y)$ bounded from above and below by $|x-y|^{N+2s}$, with $s\in (0,1)$.   The entries in the equations are in some  Morrey spaces and the underline domain $\Omega$ satisfies some mild regularity assumptions. In the particular case of the fractional Laplacian, our results are new.  When $K$ defines a nonlocal operator with sufficiently regular coefficients,   we obtain H\"older estimates, up to the boundary of $ \Omega$,  for $u$ and  the ratio $u/d^s$, with $d(x)=\textrm{dist}(x,\mathbb{R}^N\setminus\Omega)$.  If the kernel $K$ defines a nonlocal operator with H\"older continuous coefficients and the entries are H\"older continuous, we obtain interior $C^{2s+\b}$ regularity estimates of  the weak solutions $u$.
Our argument is based on blow-up analysis  and compact Sobolev embedding. 
 \end{abstract}

\maketitle


%
 \section{Introduction}
We consider $s\in (0,1)$, $N\geq 1$ and   $\O$   an open set   in $\R^N$ of class $C^{1,\g}$, for some $\g>0$. We are interested in  interior and  boundary H\"older  regularity estimates  for  functions $u $ solution   to the equation 
\be  \label{eq:Main-problem}
\cL_K u+ Vu  = f \qquad\textrm{ in $ \O$}\qquad \textrm{ and } \qquad u= 0 \qquad\textrm{ in $  \O^c$}.
\ee 
where $\O^c:=\R^N\setminus \O$ and   $\cL_K$ is a nonolocal operator defined by a symmetric kernel $K\asymp |x-y|^{-N-2s}$. We refer to  Section \ref{ss:Non-loc-cont} below   for more details.  Our model operator is  $\cL_K= \Ds_a$, the so    called   anisotropic fractional Laplacian, up to a sign multiple. It  is    defined, for all  $\vp\in C^2_c(\R^N)$,  by
$$
\Ds_a\vp(x)=PV\int_{\R^N}\frac{\vp(x)-\vp(x-y)}{|y|^{N+2s}}a\left({y}/{|y|}\right) \, dy,
$$ 
with $a: S^{N-1}\to \R$ satisfying
\be\label{eq:def-a-anisotropi}
a(-\theta)=a(\theta) \qquad\textrm{ and } \qquad\L \leq a(\theta)\leq   \frac{1}{\L}\ \qquad\textrm{ for all $\theta\in S^{N-1}$},
\ee 
for some constant $\L>0$.  Here, the entries $V,f$ in \eqref{eq:Main-problem} belongs to some  Morrey spaces.\\
 In the recent years the study of nonlocal equations have attracted a lot of interest due to their manifestations in  the modeling of real-world phenomenon and their rich structures in the mathematical point of view. In this respect,  regularity theory remains central  questions.  
 Interior regularity and  Harncak inequality have been intensiveley  investigated in  last decades, see e.g. \cite{FK,Hoh,BL,Kassm,Barrios,CS1,CS2,CS3,BC,KM1,KM2,Jin-Xiong,Kriventsov,SS,Cozzi} and the references therein. On the other hand, boundary regularity and  Harnack inequalities was studied in  \cite{Bogdan-bdr,Chen-Song,BCI,BKK}.  
 
Results which are, in particular, most  relevent to the content of this paper   concern  those dealing  with  nonlocal operator in "divergence form" with measurable coefficient, i.e. $K$ is symmetric on $\R^N\times \R^N$ and    $K(x,y)\asymp |x-y|^{-N-2s}$, see \cite{KS}.    In this case the de Giorgi-Nash-Moser energy methods  were used to obtain interior  Harnack inequality and H\"older estimates, see \cite{Dicastro, KMS, KRS, Cozzi}.  We note that  the papers  \cite{KRS,KS} deal also with more general kernels  than those satisfying  $K(x,y)\asymp |x-y|^{-N-2s}$ on $\R^N\times \R^N$ only.  We also mention the work of  Kuusi, Mingione and Sire in \cite{KMS} who obtained local pointwise behaviour of solutions to   quasilinear nonlocal elliptic equations with measurable coefficients, provided the Wolff potential of the right hand side satisfies some  qualitative properties. 

 In this paper, we are concerned in both  interior and boundary regularity of nonlocal equations   with "continuous coefficient".   Let us recall that in the classical case of operators in divergence form with continuous coefficients, after scaling, the limit operator is given by the Laplace operator $\D$.  The meaning of "continuous coefficient" in the nonlocal framework is not immediate due to the singularity of the kernel $K$ at the diagonal points $x=y$.  However,  under nonrestrictive continuity assumptions, detailed in Section \ref{ss:Non-loc-cont} below, we find out that the limiting nonlocal operator   is  the anisotropic fractional Laplacian $-\Ds_a$ in many situations. \\
    
 %
 %
  Letting $d(x):=\textrm{dist}(x,  \O^c )$,  the boundary regularity we are interested in here is the    H\"older regularity estimates of $u/d^s$ for  the nonlocal operator $\cL_K$.   Such regularity results for solutions to \eqref{eq:Main-problem} has been studied long time ago  when $\cL_K=(-\D)^{1/2}$. They are  of  interest e.g. in fracture mechnics, see  \cite{Costabel} and the referenes therein. The general case for $\Ds$,  $s\in (0,1)$,    has been considered only in the recent  years and  it is by now merely  well understood when $u,V,f\in L^\infty(\R^N)$ and $\O$ a domain of class $C^{1,\g}$, for some $\g>0$. Indeed,   in the case of the fractional Lapalcian ($a\equiv 1$) and $\g=1$, the first   H\"older regularity estiamte of $u/d^s$ in $\ov \O$ was obtained by Ros-Oton and Serra in \cite{RS1}. They sharpened and generalized this result to translation invariant operators, even to fully nonlinear equations,     in their subsequent papers \cite{RS2,RS3,RS4}.  We refer the reader to the recent survey paper \cite{RS-survey} for a detailed list of existing results.  
 In the case where $V $,  $a$, $f$ and $\O$ are of class $C^\infty$, we quote the   works of Grubb, \cite{Grubb1,Grubb2,Grubb3}, where it is proved that $u/d^s$ is of class $C^\infty$, up to the closure of $\O$. Especially in   \cite{Grubb1},   the enteries $V,f$ are also allowed to belong to some  $  L^p$ spaces, for some large $p$. More precisely, when   $V\in C^\infty(\ov \O)$ and  $f\in L^p(\O)$, for some $p>N/s$, then provided $\O$ and $a$ are of class $C^\infty$, Grubb proved in \cite{Grubb1} that  $u/d^s$ is of class $C^{s-N/p}(\ov\O)$.  \\
 %
Here,  we prove sharp   H\"older regularity estimates of $u/d^s$, for  $\O$ an open set of class $C^{1,\g}$ and $V,f$ are in some   spaces  containing the Lebesgue space $L^p$ for $p>N/s$. Let us now recall the Morrey   space which  will be  considered in the following of this paper.   For $\b\in [0,2s)$,  we define the Morrey space $\cM_\b$ by    the set of functions $f \in L^1_{loc} (\R^N)\ $ such that 
$$
 \|f\|_{\cM_\b}:= \sup_{\stackrel{x\in \R^N}{r\in (0,1)}} r^{\b-N}\int_{  B_r(x)} |f(y)| \, dy<\infty,
$$
with  $\cM_0:=L^\infty(\R^N)$. Such spaces   introduced by Morrey in \cite{Morrey},  are suitable for getting H\"older regularity in the study of partial differential equations.

 Let us now explain in an abstract form the insight  in our  consideration  of  the Morrey space.  Indeed,  given a function $g\in L^1_{loc}(\R^N)$, we  put $g_{r,x_0}(x):=r^{2s} g(r x+ x_0)$, for $x_0\in\R^N$ and $r>0$.\\
  For $\b\geq0$,  we say that $g$ satisfies a Rescaled Translated Coercivity Property  (RTCP, for short) of order $\b$,    if there exists a constant  $C:=C(g,N,s,\b)>0$ such that for all  $x_0\in \R^N$ and $r\in (0,1)$, we have 
\be\label{eq:resc-coerciv-int}
C \int_{\R^N} | g_{r,x_0}(x)| v^2(x)\, dx \leq  r^{2s-\b} \|v\|^2_{H^s(\R^N)} \qquad\textrm{ for all $v\in H^s(\R^N)$.}
\ee
     Then what we will prove, for solutions  $u$ to \eqref{eq:Main-problem} when $\cL_K$ is  (up to a scaling)  close to $\Ds_a$,   are the following implications:
\be \label{eq:implic-first-int}
\begin{aligned}
f,V \textrm{ satisfy a   RTCP of order  $\b\in [0, 2s)$}   \qquad\Longrightarrow \qquad & u\in C^{\min (1,2s-\b)-\e}_{loc}(\O),\\
  f,V \textrm{ satisfy  a RTCP of order  $\b\in [0,2s)$}  \qquad\Longrightarrow \qquad& u\in C^{\min (s,2s-\b)-\e}_{loc} (\ov \O),
\end{aligned}
\ee
for every $\e>0$ and $\O$ an open set with   $C^1$ boundary.  For higher order   regularity, under some regularity assumptions on $K$ and $\O$, quantified by some parameter  $\b'\in [0, 2s)$,  we obtain
\be \label{eq:implic-int}
\begin{aligned}
f,V \textrm{ satisfy  a RTCP of order  $\b\in (0, 2s)$}  \qquad\Longrightarrow \qquad & u\in C^{2s-\max(\b,\b')}_{loc}(\O),\\
  f,V \textrm{ satisfy a RTCP of order  $\b\in (0,s)$   }  \qquad\Longrightarrow \qquad& u/d^s\in C^{s-\max(\b,\b')}_{loc} (\ov \O),
\end{aligned}
\ee
provided $2s-\max(\b,\b')\not=1 $.
   Here and in the following, it will be understood that   $C^{\nu}:=C^{1,\nu-1}$ if $\nu\in (1,2)$. 
It is not difficult to see that   functions $g$ satisfying a  RTCP of order $\b$ belongs to $\cM_\b$. On the other hand the converse, which is not trivial, also  holds true, and  in fact,  we will prove a more general inequality for the Kato class of functions which could be of independent interest, see Lemma \ref{lem:coerciv} below. \\
%

%

Since our results are  already new for $\Ds_a$, we   state first   simpler versions of our main results, and postponed the generalization  to  $\cL_K$ in  Section \ref{ss:Non-loc-cont} below. To do so, we need to recall the distributional domain of the operator $\Ds_a$. It is given by   $\cL^1_s$, the set  of functions $u\in L^1_{loc}(\R^N)$ such that   $\|u\|_{\cL^1_s}:=\int_{\R^N}\frac{|u(x)|}{1+|x|^{N+2s}}\,dx<\infty$.

\begin{theorem}\label{th:mainth}
Let $s\in (0,1)$, $\b\in (0,2s)$ and $a$ satisfy \eqref{eq:def-a-anisotropi}. Let 
$\O\subset\R^N$ be an open set of class  $ C^{1,\g}$,   $\g>0$, in a neighborhood of  $0\in \de\O$.   
Let $u\in H^s(B_1)\cap \cL^1_s $ and $f,V\in \cM_\b$  be such that
$$
\Ds_a u+ Vu  = f \qquad\textrm{ in $  \O$}.
$$
\begin{enumerate}
\item[(i)] Then  for every $\O_1\subset\subset \O\cap B_1$,   there  exists a     constant $C >0$ such that 
$$
\|u\|_{C^{2s-\b}(\O_1)} \leq C\left(\|u\|_{L^2(B_1)}+ \|u\|_{\cL^1_s} +  \|f\|_{\cM_\b}  \right).
 $$
 \item[(ii)] If   $\b\in(0,s)$,  and  $u\equiv 0$ in $  \O^c $,   then    there  exist  some constants $C,\varrho>0$ such that 
$$
 \|u/ d^s\|_{C^{\min(\g,s-\b)} ( B_{\varrho}\cap \ov\O)}\leq C\left(\|u\|_{L^2(B_1)}+ \|u\|_{\cL^1_s} +  \|f\|_{\cM_\b}  \right),
 $$
\end{enumerate}
where   $d(x)=\textrm{dist}(x, \O^c)$.  The constants $C $ and $\varrho$ above, only depend   on $s,N,\b,\g,\L,\O ,\O_1$ and $  \|V\|_{\cM_\b}$.
\end{theorem}
Provided $u,V,f\in L^\infty(\R^N)$, by letting $\b\searrow0$, we recover the boundary regularity in  \cite{RS3} for $C^{1,\g}$ domains  and partly the one in \cite{RS2} for $C^{1,1}$ domains. We mention that   in \cite{RS2}, a weaker   ellipticity assumption (second condition in \eqref{eq:def-a-anisotropi}) was considered.\\
Obviously if $f\in L^p(\R^N)$, with $p>1$, then $f\in \cM_{\frac{N}{p}} $. For the strict inclusion of Lebesgue spaces  in Morrey spaces, see e.g. \cite{Fazio}.   An immediate consequence of Theorem \ref{th:mainth} is therefore the following result.
\begin{corollary}\label{cor:maincor}
Let $s\in (0,1)$ and $a$ satisfy \eqref{eq:def-a-anisotropi}. Let  $\O\subset\R^N$ be an    open set    of class  $ C^{1,\g}$,   $\g>0$, in a neighborhood of   $0\in \de\O$.   
Let   $f, V\in L^p(B_1) $, for some $p>\frac{N}{s}$, and    $u\in H^s (B_1)\cap\cL^1_s$ satisfy 
$$
\Ds_a u + V u= f \qquad\textrm{ in $ \O$} \qquad\textrm{ and } \qquad u= 0 \qquad\textrm{ in $  \O^c $}.
$$
Then there  exist positive  constants $C,\varrho>0$ such that 
$$
   \|u/ d^s\|_{C^{\min(\g,s-N/p)} (B_{\varrho}  \cap \ov\O)}\leq C\left(\|u\|_{L^2(B_1)}+\|u\|_{\cL^1_s} +  \|f\|_{L^p(B_1)}   \right).
 $$
The constants $C$ and $\varrho  $  depend only on $s,N,\g,p,\L,\O$ and $  \|V\|_{L^p(  B_1)}$.
\end{corollary}
As mentioned earlier, we recall that the boundary regularity in Corollary \ref{cor:maincor}   was known only when $a\in C^\infty(S^{N-1})$ and $\O$ of class $C^\infty$, see \cite{Grubb1}.
In the classical case of the Laplace operator,   the corresponding result of   Corollary \ref{cor:maincor} is that $u$ is of class $ C^{1,\min(\g,1-N/p)}$ up to the boundary,  see \cite{GT}.
We note that interior and  boundary Harnack inequalities for the operator  $\Ds+V$, with   $V$ in the Kato class of potentials (larger than the Morrey space) and $\O$ a Lipschitz domain have been proven   in \cite{Bogdan,Song}. In \cite{Fall-Oton}, we shall   provide an explicit modulus of   continuity for  solutions to \eqref{eq:Main-problem}, when $V$ and $f$ belong to  the Kato class of potentials. \\
%

\subsection{Nonlocal operators with possibly continuous coefficients}\label{ss:Non-loc-cont}
In the following, for a function $b\in L^\infty(S^{N-1}) $, we define $\mu_b(x,y)=|x-y|^{-N-2s} b((x-y)/|x-y|))$ for every  $x\not=y\in\R^N$.

 Let $\k>0$ be a positive constant and   $\l:\R^{N}\times\R^N\to [0,\k^{-1}] $. We  consider the class of  kernels 
 $K: \R^N\times \R^N\to [0, +\infty]$  satisfying the following properties:  
 \be \label{eq:Kernel-satisf-cont}
 \begin{aligned}
(i)\,& K(x,y)=K(y,x) \qquad\textrm{ for all $x\not=y\in\R^N$,}\\
(ii)\,& \k \mu_1(x,y)\leq K(x,y)\leq \frac{1}{\k} \mu_1(x,y) \qquad\textrm{ for all $x\not=y\in\R^N$,} \\
 (iii) \,&\left| K( x,y)-\mu_b(x,y)  \right| \leq    \l(x,y)   \mu_1(x,y) \qquad\textrm{ for all $x\not=y\in B_2$.}\hspace{4cm}
 \end{aligned}
 \ee
The class of kernels satisfying  \eqref{eq:Kernel-satisf-cont} is denoted by $\scrK(\l,b,\k)$.\\  
Let   $\O\subset\R^N$ be an open set and let $f,V\in L^1_{loc}(\R^N)$.
For $K$ satisfying  \eqref{eq:Kernel-satisf-cont}(i)-(ii),    we say that $u\in H^s_{loc}(\O)\cap \cL^1_s $ is  a (weak) solution to 
$$ 
 \cL_{K} u+ Vu  = f \qquad\textrm{ in $  \O$} 
$$
if $uV\in L^1_{loc}(\O)$ and for all $\psi\in C^\infty_c( \O)$, we have 
\begin{align*}
\frac{1}{2}\int_{\R^{2N}}(u(x)-u(x+y))(\psi(x)-\psi(x+y))K(x,x+y)dxdy+\int_{\R^N} V(x)u(x)\psi(x)dx=\int_{\R^N}f(x)\psi(x)dx.
\end{align*}
%
%
The class of operator $\cL_K$ corresponding to the kernels $K$ satisfying  \eqref{eq:Kernel-satisf-cont}(i)-(ii)  can be seen as the nonlocal version of  operators in divergence form in the classical case. Here we obtain regularity estimates for $K\in\scrK(\l,a,\k)$  provided $\l$ is small and $a$ satisfies \ref{eq:def-a-anisotropi}. We thus include, in particular,  nonlocal operators with    "continuous" coefficients. The meaning of continuous coefficients for nonlocal operators   might be awkward, since one is dealing with kernels which are not finite at the points $x=y$. Using polar coordinates, we can depict an encoded limiting operator which is nothing but   the anisotropic fractional Laplacian. \\
In view of  Remark \ref{rem:Kernel-compact-supp} below, all  results stated below remains valid if we consider    kernels    $\ti K: \R^N\times\R^N\to [0,\infty]$ (with possibly   compact support)  satisfying
 \be \label{eq:locl-elliptic-int}
 \begin{aligned}
(i)\,&\ti  K(x,y)=\ti K(y,x) \qquad\textrm{ for all $x\not=y\in\R^N$,}\\
(ii')\,& \k \mu_1(x,y)\leq \ti K(x,y)  \qquad\textrm{ for all $x\not=y\in B_2$} \\
 (ii'') \,&    \ti K(x,y)  \leq \frac{1}{\k}    \mu_1(x,y) \qquad\textrm{ for all $x\not=y\in \R^N$.}   \hspace{6.5cm}
 \end{aligned}
 \ee
 This  is due to the fact  that the regularity theory of the   operators $\cL_{\ti K}$ is included in those   of the form $\cL_K+V$, with $K$ satisfying  \eqref{eq:Kernel-satisf}$(i)$-$(ii) $, for some potential $V$ of class $ C^\infty$.   \\
 
%
We introduce $\ti \scrK(\k)$,  the class of kernels    $K$, satisfying \eqref{eq:locl-elliptic-int}  and   such that    the map 
$$
\R^N\times (0,\infty)\times S^{N-1}\to \R, \qquad (x,r,\th)\mapsto r^{N+2s}K( x,x+r\th )
$$ 
has an extension $\ti\l_K: \R^N\times [0,\infty)\times S^{N-1}\to \R$ that is continuous in the variables $x,r$ on $B_2\times \{0\}$. That is, for every $x_0\in B_2$, we have  $\left| \ti\l_{K}(x,r,\th)-\ti\l_{K}(x_0,0,\th) \right |\to 0$     {as  $|x-x_0|+r\to 0$. }\\
   
 Now for $K\in  \ti \scrK(\k)$ and $x_0\in B_2$ let us suppose that 
$$
\begin{aligned}
\sup_{\th \in  S^{N-1}}  \left| \ti\l_{K}(x,r,\th)-\ti\l_{K}(x_0,0,\th) \right |\to 0   \qquad& \textrm{as  $|x-x_0|+r\to 0$ }    
%
%
 %
%
%
\end{aligned}
$$
and consider    the rescaled kernel around $x_0$, given by  $K_{\rho,x_0}(x,y):=\rho^{N+2s} K(\rho x+x_0,\rho y+ x_0)$. Then     we can show that, provided $\rho$ is small,  $K_{\rho,x_0}$ satisfies \eqref{eq:Kernel-satisf-cont}$(iii)$ with $b(\theta)=  \ti\l_{K}(x_0,0,\th)  $,  for some function function $\l_\rho$, satisfying   $\|\l_\rho\|_{L^\infty(B_1\times B_1)}\to 0$ as $\rho\to 0$. Obviously,   to expect    the limiting kernel to be  symmetric, we need to require that $ \ti\l_{K}(x_0,0,\th)= \ti\l_{K}(x_0,0,-\th) $, for all $\th\in S^{N-1}$. 
From this,  it is   natural to expect that $\cL_K$ inherits certain  regularity properties  of  $\Ds_a$ whenever $K\in\scrK(\l,a,\k)$, provided $\l$ is small, in the spirit  of  Caffarelli \cite{Caffarelli} and Caffarelli-Silvestre \cite{CS2}. This is the purpose of the next results, under mild regularity assumptions on $K$ and $\O$. \\

It is worth to  mention that the  kernels in $\ti\scrK(\k)$ appear for instance  in the study of nonlocal mean curvature operator about a smooth hypersurface, see e.g. \cite{Barrios,Fall-Reg-Appl}. More generally,  a typical example is when considering a  $C^{1}$-change of coordinates $\Phi$ e.g.  in the kernel $|x-y|^{-N-2s} $ (could be defined on the product of   hypersurfaces $\cM\times\cM$). The singular part of the new kernel is then  given by $K_{\Phi}(x,y)=|\Phi(x)-\Phi(y)|^{-N-2s}$, for some local  diffeomorphism $\Phi\in C^1(B_4;\R^N)$. In this case,     
$$
\ti\l_{K_\Phi}(x,r,\th)=\left|\displaystyle\int_0^1D\Phi(x+ r \t \th )\th  \,d\t \right|^{-N-2s}
$$
and thus $ K_{\Phi}\in \ti \scrK(\k) $, for some $\k>0$. 
 As a consequence,     $\ti\l_{K_\Phi}(x,0,\th)=\left|D\Phi(x  )\th   \right|^{-N-2s},$ which is even in $\th$. We refer to Section \ref{ss:proof-bdr-reg-change-var} below in a more general setting.   Thanks to the \`a priori estimates that we are about to state below,   we shall  prove in \cite{Fall-Reg-Appl} optimal  regularity results   paralleling the regularity theory for elliptic equations in divergence form, with regular coefficients,  and provide applications in the study of nonlocal geometric problems.  \\

To obtain interior regularity in \eqref{eq:implic-int}, we need to care on the kernels $K$ for which the action of $\cL_K$ on affine functions  can be quantified. In this respect,  some regularity on $K$ is required. More precisely, for  $K\in \scrK(\l,a,\k)$ and $x\in \R^N$, we define     
 \be\label{eq:def-small-joK}
 j_{o,K}(x):=(2s-1)_+ PV\int_{\R^N}y \left\{K(x,x+y)-K(x,x-y)  \right\}\,dy,
\ee
where $\ell_+:=\max(\ell,0)$, for $\ell\in \R$. Letting  
\be\label{eq:def-ti-lambda-o-K-intro}
  \ti\l_{o,K}(x,r,\th)= \frac{1}{2} \left\{ \ti\l_{K}(x,r,\th)- \ti\l_{K}(x,r,-\th)  \right\},
\ee 
we  see that 
$$
 j_{o,K}(x)=2(2s-1)_+\int_{S^{N-1}}\th  \left\{  PV\int_0^\infty r^{-2s}\ti\l_{o,K}(x,r,\th)\,dr  \right\}\, d\th.
$$
The   main    regularity assumption we make on $K$ is that $j_{o,K}$ is locally in $\cM_{\b'}$, in the sense that     $\vp_2 j_{o,K}\in  \cM_{\b'}$, for some $\b'=\b'(K)\in [0,2s)$.  Here and in the following,  $\vp_R\in C^\infty_c(B_{2R})$, with $\vp_R\equiv 1$ on $B_R$. Since, $\b'$ depends on $K$,   the main point  will be to obtain   regularity estimate  by constants independent in $\b'$.  We observe that, for example, for a kernel $K $  such that  $\ti \l_{o,K}(x,r,\th)\leq  r^{\a+(2s-1)_+}$, for $r\in (0,1)$ and for some $\a>0$, then  $\vp_2 j_{o,K}\in L^\infty(\R^N)=\cM_{0}$.  Of course if $s\in (0,1/2]$ such additional regularity assumption on $K$ is unnecessary. This is also the case if $\cL_K$ is a translation invariant, i.e. $K(x,y)=k(x-y)$, for some even function $k:\R^N\to \R$. 
  
Our main result for interior regularity reads as follows.
\begin{theorem}\label{th:int-reg-Morrey-intro}
Let    $s\in (0,1)$ and  $\b,\d\in (0,2s)$.  Let $a$ satisfy \eqref{eq:def-a-anisotropi}.  Let $K\in \scrK(\l,a,\k)$  and assume that $j_{o,K}$ defined in \eqref{eq:def-small-joK},  satisfies
$$
\| \vp_2 j_{o,K} \|_{\cM_{\b'}}\leq c_0 ,
$$
for some $c_0$ and   $\b'=\b'(K)\in [0,2s-\d)$.
  Let $f,V\in \cM_\b$ and  $u\in H^s(B_2)\cap \cL^1_s$   be such that
$$
\cL_{K} u+Vu= f \qquad\textrm{ in $B_2$}.
$$
Then, provided $2s-\max(\b,\b')\not=1$,   there exist $C,\e_0>0$, only   depending  only on $N,s,\b, \L,\k,c_0,\d$ and $\|V\|_{\cM_\b}$,   such that if $\|\l\|_{L^\infty(B_2\times B_2)}\leq \e_0$, we have 
  $$
  \|u\|_{C^{2s-\max(\b,\b')}(B_{1})}\leq C\left( \| u\|_{L^2(B_{2})}+\|u\|_{\cL^1_s}+\|f\|_{\cM_\b} \right).
  $$
  Moreover if $2s\leq1$, then we can let $\b'=0$.
\end{theorem}
As a consequence, we have the following result.
\begin{corollary}\label{cor:int-reg-Morrey-intro}
Let    $s\in (0,1)$,  $\b,\d\in (0,2s)$ and $\k>0$.     Let $K\in\ti\scrK(\k) $    satisfy:  for every $x_1,x_2\in B_2$, $r\in (0,2)$, $\th \in S^{N-1}$,
\begin{itemize}
\item   $   \left| \ti\l_{K}(x_1,r,\th)-\ti\l_{K}(x_2,0,\th) \right |\leq \t(|x_1-x_2|+r)  ; $
 \item  $ \ti\l_{K}(x_1,0,\th)= \ti\l_{K}(x_1,0,-\th), $
\end{itemize}
 for some modulus of continuity $\t\in L^\infty(\R_+)$, with  $\t(t)\to 0$ as $t\to 0$.  Assume  that $j_{o,K}$ defined in \eqref{eq:def-small-joK},  satisfies
$$
\| \vp_2 j_{o,K} \|_{\cM_{\b'}}\leq c_0 ,
$$
for some $c_0$ and   $\b'=\b'(K)\in [0,2s-\d)$.
  Let $f,V\in \cM_\b$ and  $u\in H^s(B_2)\cap \cL^1_s$   be such that
\be \label{eq:cLK-uVu-f}
\cL_{K} u+Vu= f \qquad\textrm{ in $B_2$}.
\ee
Then, provided $2s-\max(\b,\b')\not=1$,   there exists $C>0$, only   depending   on $N,s,\b, \d,\k,c_0,\t$ and $\|V\|_{\cM_\b}$,   such that 
  $$
  \|u\|_{C^{2s-\max(\b,\b')}(B_{1})}\leq C\left( \| u\|_{L^2(B_{2})}+\|u\|_{\cL^1_s}+\|f\|_{\cM_\b} \right).
  $$
   Moreover if $2s\leq1$, then we can let $\b'=0$.
\end{corollary}

It is natural to expect that under some H\"older regularity assumption on $\ti\l_K$ and on the entries,    solutions are in   fact  classical. 
Indeed we have.
\begin{theorem}\label{th:Small-Shauder}
Let    $s\in (0,1)$, $\k>0$ and      let $K\in\ti\scrK(\k) $    satisfy: 
\begin{itemize}
\item   for every $x_1,x_2\in \R^N$, $r_1,r_2\in [0,2)$, $\th \in S^{N-1}$,
 $$
  \left| \ti\l_{K}(x_1,r_1,\th)-\ti\l_{K}(x_2,r_2,\th) \right |\leq c_0 (|x_1-x_2|^\a+|r_1-r_2|^\a)  ;
 $$
 \item  for every $x_1,x_2\in B_2$, $r \in [0,2)$, $\th \in S^{N-1}$, we have $ \ti\l_{o, K}(x_1,0,\th)=0$ and 
 $$
  \left| \ti\l_{o, K}(x_1,r,\th)- \ti\l_{o,K}(x_2,r,\th) \right |\leq  c_0\min(|x_1-x_2|^{\a+(2s-1)_+} ,r^{\a+(2s-1)_+} ),
  $$
\end{itemize}
for some constants $\a,c_0>0$, where $  \ti\l_{o, K}$ is given by \eqref{eq:def-ti-lambda-o-K-intro}.   Let $f\in {C^\a(B_2)}$ and  $v\in H^s(B_2)\cap C^\a(\R^N)$   be such that
$$
\cL_{K} v= f \qquad\textrm{ in $B_2$}.
$$
 Then there exists $\ov \a>0 $ only depending on $s,N,c_0,\k$ and $\a$,  such that  for all $\b\in (0,\ov \a)$, with $2s+\b\not\in\N$,
\be \label{eq:Shauder-estim-intr}
  \|v\|_{C^{2s+\b  }(B_{1})}\leq C\left( \| v\|_{C^\b(\R^N)}+ \|f\|_{C^\b(B_2)} \right),
\ee
for some constant $C$ depending   only on  $s,N,c_0,\k,\a$ and $\b$.
\end{theorem}
We note that the $C^\b(\R^N)$-norm of $v$ in  \eqref{eq:Shauder-estim-intr} can be replaced with $ \|v\|_{L^2(B_{2})}+ \|v\|_{\cL^1_s}$, provided, we require H\"older regularity of $\ti\l_K$ in the variable $\th$ i.e.  $\|\ti\l_K\|_{C^\a(\R^N\times [0,2)\times S^{N-1})}\leq c_0$.    \\

We  now turn to  our boundary regularity estimates in \eqref{eq:implic-int}. In this case, it is important to consider those kernels $K$ for which $\cL_K d^s$ can be quantified. Here our assumption is that $\cL_K d^s$ is given by   a function in $\cM_{\b'}$, for some $\b'\in [0,s)$. To be more precise, we consider all kernel $K$ for which,   there exist $\b'=\b'(\O,K)\in [0,s)$  and  a   function $g_{\O,K}\in \cM_{\b'}$  such that 
\be\label{eq:def-Ds-ds-g-int}
 \cL_K   (\vp_2 d^s)=g_{\O,K} \qquad\textrm{  in the weak sense in $B_{ r_0}\cap \O$,}
\ee
where $r_0>0$, only depends on $\O$, is such that $\vp_2 d^s \in H^s (B_{r_0} )\cap \cL^1_s$.  We note that   $g_{\O,K}$ might  be  singular near  the  boundary, since we are considering only domains of class $C^{1,\g}$. In fact, see  \cite{RS3}, for  $\g\not=s$ then    $|\Ds_a  d^s(x)|\leq C d^{(\g-s)_+}(x)$ for every  $x\in B_{r_0} \cap \O$, for some $r_0$, only depending on $\O$. This, in particular,  shows  that there  exists a  $g_{\O, \mu_a}\in \cM_{(s-\g)_+}$ satisfying \eqref{eq:def-Ds-ds-g-int}.   We note that \eqref{eq:def-Ds-ds-g-int} encode both the regularity of $K$ and of $\O$.

 Our next main result is the following.
\begin{theorem}\label{th:mainth-gen}
Let   $s\in (0,1)$,    $\b,\d\in (0,s)$ and $\O$ an open set of class $C^{1,\g}$, $\g>0$, near $0\in \de\O$. Let $a$ satisfy \eqref{eq:def-a-anisotropi}, for some $\L>0$.  Let $K\in \scrK(\l,a,\k)$  satisfy   \eqref{eq:def-Ds-ds-g-int}, with    $$\|g_{\O,K} \|_{ \cM_{\b'}}\leq  c_0,$$   for some $\b'=\b'(\O, K)\in [0,s-\d)$ and $c_0>0$.  
Let  $f, V\in \cM_\b$,  and $u\in H^s(B_2) \cap \cL^1_s$ be such that
$$
\cL_K u+ V   u= f \qquad\textrm{ in $   \O$}\qquad\textrm{ and }\qquad u=0 \qquad\textrm{ in  $   \O^c$. }
$$
Then    there exist $C,\varrho>0$, only   depending  only on $N,s,\b, \L,\k,\O,c_0,\d$ and $\|V\|_{\cM_\b}$, such that if $\|\l\|_{L^\infty(B_2\times B_2)}\leq \e_0$, we have 
$$
 \|u/ d^s\|_{C^{s-\max(\b,\b')} ( B_{\varrho}\cap \ov\O)}\leq C\left(\|u\|_{L^2(B_{2})}+ \|u\|_{\cL^1_s}+ \|f\|_{\cM_\b}  \right).
 $$
\end{theorem}
In the case of uniformly continuous coefficient also, we have  the following boundary regularity estimates.
\begin{corollary}\label{cor:mainth-gen}
Let   $s\in (0,1)$,    $\b,\d\in (0,s)$ and $\O$ an open set of class $C^{1,\g}$, $\g>0$, near $0\in \de\O$.  For  $\k>0$,      let $K\in\ti\scrK(\k) $    satisfy: for every $x_1,x_2\in B_2$, $r\in (0,2)$, $\th \in S^{N-1}$,
\begin{itemize}
\item   $   \left| \ti\l_{K}(x_1,r,\th)-\ti\l_{K}(x_2,0,\th) \right |\leq \t(|x_1-x_2|+r)  ; $
 \item  $ \ti\l_{K}(x_1,0,\th)= \ti\l_{K}(x_1,0,-\th), $
\end{itemize}
 for some function $\t\in L^\infty(\R_+)$, with  $\t(t)\to 0$ as $t\to 0$. Suppose also that     $$\|g_{\O,K} \|_{ \cM_{\b'}}\leq  c_0,$$   for some $\b'=\b'(\O, K)\in [0,s-\d)$ and $c_0>0$.  
Let  $f, V\in \cM_\b$,  and $u\in H^s(B_2) \cap \cL^1_s$ be such that
$$
\cL_K u+ V   u= f \qquad\textrm{ in $   \O$}\qquad\textrm{ and }\qquad u=0 \qquad\textrm{ in  $   \O^c$. }
$$
Then    there exist $C,\varrho>0$, only   depending   on $N,s,\b, \t,\k,\O,c_0,\d$ and $\|V\|_{\cM_\b}$, such that  
$$
 \|u/ d^s\|_{C^{s-\max(\b,\b')} ( B_{\varrho}\cap \ov\O)}\leq C\left(\|u\|_{L^2(B_{2})}+ \|u\|_{\cL^1_s}+ \|f\|_{\cM_\b}  \right).
 $$
\end{corollary}
As an application of the above result together with a global diffeomorphism   that locally  flatten the boundary $\de\O$ near $0$, we get the following 
\begin{theorem}\label{th:bdr-reg-C11-domain}
Let   $s\in (0,1)$,    $\b\in (0,s)$ and $\O$ an open set of class $C^{1,1}$,  near $0\in \de\O$.  For  $\k>0$,      let $K\in\ti\scrK(\k) $    satisfy: 
\begin{itemize}
\item   $ \| \ti\l_{K}  \|_{C^{s+\d}\left(B_2\times[0,2)\times S^{N-1} \right)}\leq c_0  ;$
 \item $   \ti\l_{K}(x,0,\th)= \ti\l_{K}(x,0,-\th)   $,   for every $x\in B_2$, and  $\th \in S^{N-1}$,  
\end{itemize}
 for some $\d,c_0>0$.    
Let  $f, V\in \cM_\b$,  and $u\in H^s(B_2) \cap \cL^1_s$ be such that
$$
\cL_K u+ V   u= f \qquad\textrm{ in $   \O$}\qquad\textrm{ and }\qquad u=0 \qquad\textrm{ in  $   \O^c$. }
$$
Then    there exist $C,\varrho>0$, only   depending   on $N,s,\b, \t,\k,\O,c_0,\d$ and $\|V\|_{\cM_\b}$, such that  
$$
 \|u/ d^s\|_{C^{s-\b} ( B_{\varrho}\cap \ov\O)}\leq C\left(\|u\|_{L^2(B_{2})}+ \|u\|_{\cL^1_s}+ \|f\|_{\cM_\b}  \right).
 $$
\end{theorem}

 
  
The proof of Theorem \ref{th:mainth-gen} and Theorem \ref{th:int-reg-Morrey-intro}  are based on some  blow-up analysis argument, where   normalized, rescaled and translated sequence of a solution to a PDE satisfy certain growth control and converges to a solution on a symmetric space, so that Liouville-type results allow to calssify the limiting solutions. Here, we are  inpired by the work of   Serra in \cite{Serra-OK},  see also \cite{RS2,RS3,RS4,Serra} for boundary regularity estimates for translation invariant nonlocal operators. Note that in the aforementioned papers, since   entries  and solutions are in $L^\infty$, the use of barriers to get \`a priori pointwise estimates and   Arzel\`a-Ascoli compactness theorems were the main tools to carry out  their blow-up analysis.   In our situation, it is clear that there is no hope of using  such tools.
Our  argument will be  based on the  estimate of  the $L^2$-average mean oscillation  of $u$ to get \`a priori pointwise estimates.  Indeed, to prove  \eqref{eq:implic-first-int}, we show the growth estimates
\be \label{eq:estim-mean-oscilla-interior-intro}
 \sup_{z\in B_1}\|u-(u)_{B_r(z)}\|_{L^2(B_r(z))}\leq C r^{N/2+\min (1,2s-\b)-\e},
\ee
with $(u)_{B_r(z)}:=\frac{1}{|B_r|}\int_{B_r(z)}u(x)\,dx$,  and 
\be  \label{eq:estim-mean-oscilla-bdr-intro}
   \sup_{z\in B_1 \cap \de\O}\|u \|_{L^2(B_r(z))}\leq C r^{N/2+\min (s,2s-\b)-\e},
\ee
for interior regularity  and    boundary regularity,  respectively. 
 The use of  Caccioppoli-type estimates, the rescaled-translated-coercivity condition \eqref{eq:resc-coerciv-int}  and    Liouville-type theorems for semi-bounded nonlocal operators  are crucial to carry out the argument. Note that \eqref{eq:estim-mean-oscilla-interior-intro} always implies $C^{\min (1,2s-\b)-\e}$  estimates. On the other hand coupling \eqref{eq:estim-mean-oscilla-bdr-intro} with interior estimates yield $C^{\min (s,2s-\b)-\e}$ estimate up to the boundary. \\
  To prove   Theorem \ref{th:int-reg-Morrey-intro} we show   an  expansion    of the form 
\be\label{eq:estim-L-E2-introduc-int}
|u(x)-u(z)-(2s-1)_+T(z)\cdot (x-z)| \leq C |x-z|^{2s-\max(\b,\b')} \qquad\textrm{ for every $x,z\in B_{1} $,}
\ee
with $\|T\|_{L^\infty(B_1)}\leq C$, while for  Theorem  \ref{th:mainth-gen},  
\be\label{eq:estim-L-E2-introduc}
\|u-\psi(z) d^s\|_{L^2(B_r(z))} \leq C r^{N/2+2s-\max(\b,\b')} \qquad\textrm{ for every $z\in B_{1}\cap \de\O$ and $r\in (0,r_0)$},
\ee
with $\|\psi\|_{ L^\infty(B_{1}\cap\de\O)}\leq C$ and the constant $C$ does not depend on $\b'$.   Recall that  $\ell_+:=\max(\ell,0)$.
 Now using appropriate  interior   regularity estimates (\eqref{eq:implic-first-int} is enough), we  translate the $L^2$ estimates in \eqref{eq:estim-L-E2-introduc} to a pointwise  estimate which yields the conclusion of the theorem.    
The proof of  \eqref{eq:estim-L-E2-introduc} uses    blow-up argument that allows to estimate the growth, in $r>0$, of the difference between $u$ and its $L^2(B_r(z))$-projection on  $\R d^s$, the one-dimensional space generated by $d^s$. Similarly  the proof of \eqref{eq:estim-L-E2-introduc-int} is achieved by estimating  the growth, in $r>0$, of the difference between $u$ and its $L^2(B_r(z))$-projection on the finite dimensional space of affine functions $\left\{t+(2s-1)_+ T\cdot (x-z)\,:\,  \textrm{$t\in \R$ and $T\in \R^N$} \right\}$.    \\
We obtain 
 Theorem \ref{th:Small-Shauder} by freezing the radial variable $r$ at $r=0$ and by using the Shauder estimates for nontranslation invariant  nonlocal operators of Serra \cite{Serra}. For that, we use   our   lower order term estimates Corollary \ref{cor:int-reg-Morrey-intro} together with some approximation procedure and boundary regularity. \\ 
Related to  this   work is the one of Monneau  in \cite{Monneau} where     blow-up arguments were used  to estimate  the modulus of mean oscillation (in $L^p$ average) for  solutions to the Laplace equation with Dini-continuous right hand sides.   \\
Sharp boundary regularity   in $ C^{1,\g}$ domains and refined Harnack inequalities   in $C^1$ domains  are        useful tools to obtain sharp regularity of the free boundaries in the study of nonlocal obstacle problems, see e.g.   \cite{CRS}. We believe that our result and arguments  might be of interest in the study of obstacle problems with non smooth  obstacles and for parabolic problems.\\
For the organization of the paper, we put in   Section \ref{eq:s-Not-Prem}   some notations and preliminary  results related to Kato class of potentials. Section \ref{s:L2-growths} is devoted to interior and boundary $L^2$-growth estimates of solutions to  \eqref{eq:Main-problem} in $C^1$ domains. Statement  \eqref{eq:implic-first-int},  is proved in   Section \ref{s:int-and-ndr-reg}. 
  Higher order boundary and interior regularity are proved in Section \ref{s:high-int-reg} and Section \ref{s:Higher-reg-int}, respectively. The proof of the main results (in particular \eqref{eq:implic-int}) are gathered in Section \ref{s:proof-main-results}.   Finally,  we prove the   Liouville theorems in Appendix \ref{s:append-Liouville-them} and we put  some useful technical results in Appendix  \ref{s:appnd-2-techinacality}. 
\section{Notations and Preliminary results}\label{eq:s-Not-Prem}
In this paper, the ball centered at  $z\in\R^N$ with radius $r>0$ is denoted by $B(z,r)$ and $B_r:=B_r(0)$. 
Here and in the following, we let $\vp_1 \in C^\infty_c(B_2)$ such that  $\vp_1 \equiv 1$ on $B_{1}$ and $0\leq \vp_1\leq  1$ on $\R^N$.  We put  $\vp_R(x):=\vp(x/R)$.  For $b\in L^\infty(S^{N-1})$,   we define   $
\mu_b(x,y)= |x-y|^{-N-2s} b\left(\frac{x-y}{|x-y|} \right).$\\

Recall  that  (see e.g. \cite{FW}), if $b$ is even, there exists $C=C(N,s,\|b\|_{L^\infty(S^{N-1})})$, such that for all $\psi\in C^\infty_c(\R^N)$ and for every  $x\in \R^N$, we have 
\be \label{eq:Dsa-abs-estim-C-infty-c}
\left|PV \int_{\R^{N}}(\psi(x)-\psi(y) )\mu_b(x,y)\,dy \right|\leq C   \frac{ \|\psi\|_{C^2(\R^N)} }{1+|x|^{N+2s}},
\ee
where $PV$ means that the integral is understood in the principle value sense. 
Throughout this paper, for the   seminorm of  the fractional Sobolev spaces, we adopt the notation
$$
[u]_{H^s(\O)}:=\left(\int_{\O\times\O} {|u(x)-u(y)|^2}\mu_1(x,y)\, dxdy\right)^{1/2}
$$
and  for  the   H\"older seminorm, we write
$$
[u]_{C^\a(\O)}:=\sup_{x\not=y\in\O} \frac{|u(x)-u(y)|}{|x-y|^\a},
$$ 
for $\a\in (0,1)$.
Letting $u\in L^1_{loc}(\R^N)$, the mean value of $u$ in $ B_r(z)$ is denoted by 
$$
u_{B_r(z)}=(u)_{B_r(z)}:=\frac{1}{|B_r|}\int_{B_r(z)}u(x)\, dx. 
$$
 \subsection{The class of operators}
In the following, it will be crucial  to consider certain  class of operators which we describe next.
\subsubsection{Symmetric operators with bounded measurable coefficients}\label{eq:op-meas-coeff}
 Firstly  we will consider kernels 
 $K: \R^N\times \R^N\to (0,\infty]$  satisfying the following properties:
 \be \label{eq:Kernel-satisf}
 \begin{aligned}
(i)\,& K(x,y)=K(y,x) \qquad\textrm{ for all $x\not =y\in\R^N$, }\\
(ii)\,& \k  \mu_1(x,y)\leq K(x,y)\leq \k^{-1} \mu_1(x,y) \qquad\textrm{ for all $x\not = y\in\R^N$, for some constant $\k>0$.}
 %
 \end{aligned}
 \ee
%
%
%
Let   $\O\subset\R^N$ be an open set and let $f,V\in L^1_{loc}(\R^N)$.
For $K$ satisfying  \eqref{eq:Kernel-satisf},    we say that $u\in H^s_{loc}(\O)\cap \cL^1_s $ is  a (weak) solution to 
$$ 
 \cL_{K} u+ Vu  = f \qquad\textrm{ in $  \O$} 
$$
if $uV\in L^1_{loc}(\R^N)$ and for all $\psi\in C^\infty_c( \O)$, we have 
\be \label{eq:nonition-of-sol}
\frac{1}{2}\int_{\R^{2N}}(u(x)-u(y))(\psi(x)-\psi(y)K(x,y)\, dxdy+\int_{\R^N} V(x)u(x)\psi(x)\,dx=\int_{\R^N}f(x)\psi(x)\, dx.
\ee
Note, in fact, that  for  the first   term in \eqref{eq:nonition-of-sol} to be  finite, it is enough that  $K$ satisfies only the upper bound in   \eqref{eq:Kernel-satisf}$(ii)$.
\begin{remark}\label{rem:Kernel-compact-supp} [Kernels with possible compact support]
In many applications, it  is important to consider    kernels $K'$ with possible compact support. This allows to treat kernels which are only locally symmetric and locally elliptic (\eqref{eq:locl-elliptic} below).  As a matter of fact,  we note that the regularity theory of the   operators $\cL_{K'}$ is included in those   of the form $\cL_K+V$, with $K$ satisfying  \eqref{eq:Kernel-satisf}, for some potential $V$ of class $ C^\infty$. Indeed,  consider a   kernel $\ti K: \R^N\times \R^N\to [0, +\infty]$ satisfying
 \be \label{eq:locl-elliptic}
 \begin{aligned}
(i)\,& \ti K(x,y)=\ti K(y,x) \qquad\textrm{ for all $x\not=y\in\R^N$,}\\
(ii')\,& \k \mu_1(x,y)\leq \ti K(x,y)  \qquad\textrm{ for all $x\not=y\in B_2$} \\
 (ii'') \,&    \ti K(x,y)  \leq \frac{1}{\k}    \mu_1(x,y) \qquad\textrm{ for all $x\not=y\in \R^N$.}   \hspace{6.5cm}
 \end{aligned}
 \ee
We define $\eta_1(x):=1-\vp_1(x)$ and $\eta(x,y)=\eta_1(x)+ \eta_1(y)$,  which satisfies
$$
\eta(x,y)\geq 
\begin{cases} 
1&  \textrm{ if $(x,y)\in \R^N\times\R^N\setminus (B_2\times B_2)$}\\ 
0&  \textrm{ if $(x,y)\in B_2\times B_2$.}
\end{cases}
$$  
Then letting $u\in H^s(B_2)\cap\cL^1_s$  be a weak solution  (in the sense of \eqref{eq:nonition-of-sol}) to the equation
$$
\cL_{\ti K} u+ \ti Vu  = \ti f \qquad\textrm{ in $   B_{2}$},
$$
we then  have  that
$$
\cL_{K} u+  Vu  =   f \qquad\textrm{ in $ B_{1/2} $},
$$
where $K(x,y)=\ti K(x,y)+\eta(x,y) \mu_1(x,y) $, 
$
 V(x)=\ti  V(x) -\int_{|y|\geq 1}\eta_1(y) \mu_1(x,y)  dy 
$
and  $  f(x)= \ti f(x) -\int_{|y|\geq 1}u(y)\eta_1(y) \mu_1(x,y)  dy  .$
 It is clear that $ K$ satisfies  \eqref{eq:Kernel-satisf}, for a new  constant $\k>0$. In addition $\|V-\ti V\|_{C^k(B_{1/2})}\leq C(N,s,k) $ and $\|f-\ti f\|_{C^k(B_{1/2})}\leq C(N,s,k) \|u\|_{\cL^1_s} $, for all $k\in\N$. 
\end{remark}
%
%
\subsubsection{Symmetric  translation invariant operators with semi-bounded measurable coefficients}
The class of operators we will consider next appears as limit of rescaled operators $\cL_K$, for $K\in \scrK(\l,b,\k)$ (see Section  \ref{ss:Non-loc-cont}).   
Let $(a_n)_n$ be a sequence  of functions, satisfying \eqref{eq:def-a-anisotropi}. Then, up to a subsequence,  it converges, in the weak-star topology of $L^\infty(S^{N-1})$, to some $b\in L^\infty(S^{N-1})$.  It follows that  $b$ is even on $S^{N-1}$ and  satisfies 
\be\label{eq:coerciv-RS}
0< \L\int_{S^{N-1}} |e_1\cdot \theta|^{2s} \,d\th \leq \inf_{\eta\in S^{N-1}} \int_{S^{N-1}} |\eta\cdot \theta|^{2s}b(\th)\,d\th\qquad \textrm{ and } \qquad \|b\|_{L^\infty(S^{N-1})}\leq \frac1\L.
\ee
For such function $b$, we denote by $L_b$ the corresponding operator, which is given by 
\be\label{eq:def-Lb}
L_b\psi(x):=PV\int_{\R^{N}}(\psi(x)-\psi(y) )\mu_b(x,y)\, dy\qquad\textrm{ for every $\psi\in C^\infty_c(\R^N)$,}
\ee
where $PV$ means that  the integral is   in the principle value sense. 
Here also solutions $u\in H^s_{loc}(\O)\cap\cL^1_s$ to the equation  $L_b u+Vu=f$ in an open set $\O$ are functions satisfying  \eqref{eq:nonition-of-sol} ---replacing $K$ with $\mu_b$.\\
%
%
%
%
%
 The  following result  is concerned with limiting  of  a sequence  of operators which are close to a translation invariant operator.    
\begin{lemma} \label{lem:Ds-a-n--to-La}
Let $(a_n)_n$ be a sequence  of functions, satisfying \eqref{eq:def-a-anisotropi} and converging in the weak-star sense    to some $b\in L^\infty(S^{N-1})$.   Let $\l_n:\R^{2N}\to [0,\k^{-1}] $, with $\l_n\to 0 $ pointwise on $\R^N\times \R^N$. Let  $K_n$ be a symmetric kernel satisfying 
$$
|K_n(x,y)-\mu_{a_n}(x,y)|\leq \l_n(x,y)\mu_1(x,y)  \qquad \textrm{  for all  $x\not=y\in \R^N $ and for all $n\in \N$. }
$$
If   $(v_n)_n$ is  a bounded  sequence in $\cL^1_s\cap H^s_{loc}(\R^N)$ such that $v_n\to v$ in $\cL^1_s$,   then 
$$
\int_{\R^{N}}v(x)L_b\psi(x)\,dx=\frac{1}{2} \lim_{n\to\infty}\int_{\R^{2N}}(v_n(x )-v_n(y))(\psi(x)-\psi(y)) {K_n(x,y)}\,dxdy    \quad\textrm{ for all  $\psi\in C^\infty_c(\R^N)$.}
$$
\end{lemma}
 
\begin{proof}
Letting $w_n= v_n-v$, then direct computations give
\begin{align}\label{eq:dist-L-b-L-k-n}
&   \int_{\R^N }v (x)  L_b \psi(x)\, dx- \frac{1}{2}\int_{\R^{2N}}(v_n(x )-v_n(y))(\psi(x)-\psi(y)) {K_n(x,y)}\,dxdy \nonumber\\
&  = \int_{\R^N } v (x)( L_{b}-L_{a_n})\psi(x)\, dx +  \frac{1}{2}\int_{\R^{2N} } (v (x)-v(y)) (\psi(x)-\psi(y))(\mu_{a_n}(x,y)-K_n(x,y))\, dxdy \nonumber\\
   &-  \frac{1}{2}\int_{\R^{2N}}(w_n(x )-w_n(y))(\psi(x)-\psi(y)) {K_n(x,y)}\,dxdy .
\end{align}
By eveness of $a_n$ and $b$,   Fubini's theorem and a change of variable, we can write
\begin{align*}
( L_{b}-L_{a_n})\psi(x)= \int_{S^{N-1}} \left[\int_{0}^\infty (\psi(x-t\th)+\psi(x+t\th)-2\psi(x)) t^{-1-2s} \, dt \right] (b(\th)-a_n(\th)) d\th .
\end{align*}
Clearly, the function $\th\mapsto  \int_{0}^\infty (\psi(x-t\th)+\psi(x+t\th)-2\psi(x)) t^{-1-2s} \, dt$ is bounded on $S^{N-1}$ and thus belongs to  $L^1(S^{N-1})$. Therefore the sequence of functions  $h_n(x):= ( L_{b}-L_{a_n})\psi(x)$ converges pointwise to zero on $\R^N$. Moreover by \eqref{eq:Dsa-abs-estim-C-infty-c},  we have that $|h_n(x)|\leq C_\psi(1+|x|)^{-N-2s}$. Since $v\in \cL ^1_s$,  it follows from the dominated convergence theorem that 
\be \label{eq:v-Lb-La-n-psi}
\int_{\R^N } v (x)( L_{b}-L_{a_n})\psi(x)\, dx=o(1) \qquad\textrm{  as $n\to \infty$.} 
\ee
 Next, we put $V_n(x,y):= |v (x)-v(y)| |\psi(x)-\psi(y)| |\mu_{a_n}(x,y)-K_n(x,y)| $.  Pick $R>0$ such that $\textrm{Supp}\psi   \subset B_{R/2}$. For $n$ large enough so that $  B_{R}  \subset B_{1/(2r_n)}$, we have 
 \begin{align*}
\int_{\R^{2N}}V_n(x,y)\, dxdy&\leq \int_{ B_{R}\times B_{R}} |v (x)-v(y)| |\psi(x)-\psi(y)| |\mu_{a_n}(x,y)-K_n(x,y)|  \,dxdy\\
&+ 2  \int_{ B_{R}}|\psi(y)| \left[\int_{  \R^N \setminus B_{R}} |v (x)-v(y)|  |\mu_{a_n}(x,y)-K_n(x,y)|  \,dx \right]\,dy.
%
%
%
 \end{align*}
 Note that $y\mapsto  |\psi(y)|  \int_{  \R^N \setminus B_{R}} |v (x)-v(y)|   |\mu_{a_n}(x,y)-K_n(x,y)|  \,dx$ is bounded and converges pointwise to zero, as $n\to \infty$. By the dominated convergence theorem, we then have that  $\int_{\R^{2N}}V_n(x,y)\, dxdy=o(1)$, as $n\to \infty$, so that
\be \label{eq:vx-vy-K-an-Kn}
 \int_{\R^{2N} } (v (x)-v(y)) (\psi(x)-\psi(y))(\mu_{a_n}(x,y)-K_n(x,y))\, dxdy=o(1) \qquad\textrm{  as $n\to \infty$.} 
\ee
Since $w_n=v_n-v$ is bounded in $ H^s_{loc}\cap \cL^1_s$ and $w_n\to 0$ in $ \cL^1_s$, by similar arguments as above, we get 
$$
 \int_{\R^{2N} } (w_n (x)-w_n(y)) (\psi(x)-\psi(y))(\mu_{a_n}(x,y)-K_n(x,y))\, dxdy=o(1) \qquad\textrm{  as $n\to \infty$.} 
$$
In addition, since $w_n\to  0$ in $\cL^1_s$,  by   \eqref{eq:Dsa-abs-estim-C-infty-c},
$$
\frac{1}{2} \int_{\R^{2N} } (w_n (x)-w_n(y)) (\psi(x)-\psi(y))\mu_{a_n}(x,y) \, dxdy  =\int_{\R^N}w_n(x) L_{a_n}\psi (x)\, dx =o(1) \qquad\textrm{  as $n\to \infty$.} 
$$
Combining the two estimates above, we conclude that
$$
 \int_{\R^{2N} } (w_n (x)-w_n(y)) (\psi(x)-\psi(y)) K_n(x,y)\, dxdy=o(1) \qquad\textrm{  as $n\to \infty$.} 
$$
 Using this, \eqref{eq:vx-vy-K-an-Kn} and \eqref{eq:v-Lb-La-n-psi} in  \eqref{eq:dist-L-b-L-k-n}, we get the conclusion in the lemma.
\end{proof}
%
%
%
\subsection{Coercivity and Caccioppoli type inequality with Kato class  potentials}
For $s>0$, we let  $\G_s:=(-\D)^{-s}$ be  the Riesz potential, which satisfy  $\Ds \G_s=\d_0$ in $\R^N$. Recall that for $N\not= 2s$, $\G_s(z)=c_{N,s}|z|^{2s-N}$  and for $N=2s$,  $\G_s(z)=c_{N,s} \log(|z|)$, for some normalization constant $c_{N,s}$.
We consider the Kato class of functions  given by 
\be\label{eq:Kato-class-potential}
\cK_s :=\left\{V \in L^1_{loc} (\R^N)\,:\,  \sup_{x\in \R^N}\int_{  B_1(x)} |V(y)|\o_s(|x-y|)\, dy<\infty \right\},
\ee
where for $N\geq 2s$,   $  \o_s(|z|)=|\G_s(z)| $ and if  $2s>N$, we set $\o_s\equiv 1$. 
Here and in the following, for every $V\in \cK_s$ and  $r\in (0,1]$, we define
\be \label{eq:def-eta_f}
\eta_V(r):=\sup_{x\in \R^N}\int_{  B_r(x)} |V(y)|\o_s(|x-y|)\, dy.
\ee
%
%
%
The following compactness result will be useful  in the following. We also note that it holds for all $s>0$, and in this case $ [u]_{{H}^s(\R^N)}^2:=\int_{\R^N}|\xi|^{2s}|\widehat{u}(\xi)|^2\, d\xi$,  for $u\in H^s(\R^N)$ with Fourier transform $\widehat u$.
\begin{lemma} \label{lem:coerciv}
Let  $s>0$ and $V\in \cK_s $.   
Then, there exists a constant $c=c(N,s)>0$ such that  for every $\d\in (0,1]$,   there exists an other constant $c_\d=c(N,s,\d)>0$  such that  for every $u\in H^s(\R^N)$,
\be \label{eq:Schetcher2} 
\||V|^{1/2} u\|_{L ^2(\R^N)}^2 \leq \eta_V(\d) \left(c [u]_{{H}^s(\R^N)}^2+c_\d \| u\|_{L ^2(\R^N)}^2 \right).
\ee
\end{lemma}
\begin{proof}
For $r>0$, we consider  the Bessel potential $G_{s,r}=(-\D+r^{-2})^{-s/2}$. See e.g. \cite[Section 6.1.2]{Grafakos}, there exists a constant $c=c(N,s)>0$  such that 
 \be\label{eq:estim-Bessel}
G_{s,1}(x)\leq c    \o_s(|x|) \quad\textrm{ for  $|x|\leq 1/2$} \quad \textrm{ and } \quad G_{s,1}(x)\leq c   |x|^{-N-1+s} \exp(-|x|/2) \quad\textrm{ for   $|x|\geq 1/2$},
\ee
where $\o_s$ is  defined in the beginning of this section.

\noindent
\textbf{Step 1:} We assume that $V\in L^\infty(\R^N)$.\\
\noindent
For $\d>0$, we consider the operator $ L: L^2(\R^N)\to L^2(\R^N)$ given by $$ L v= |V|^{1/2}(-\D+\d^{-2})^{-s/2} v.$$  We note that the adjoint of $L$ is given by $L^*= (-\D+\d^{-2})^{-s/2}|V|^{1/2}$.\\

\noindent
\textbf{Claim:} There exists $c=c(N,s)>0$ such that   for every $\d\in (0,1/2]$,  
\be \label{eq:claim-coerciv}
 \| LL^* v \|_{L ^2(\R^N)}^2  \leq  c \eta_V(\d)^2      \| v \|_{L ^2(\R^N)}^2    .
\ee
By H\"older's inequality and using the fact that $ G_{s,r}(x)=G_{s,1}(x/r)$, we obtain
\begin{align*}
 \| LL^* v \|_{L ^2(\R^N)}^2&=\| | V|^{1/2} (-\D+\d^{-2})^{-s} |V|^{1/2} v \|_{L ^2(\R^N)}^2    =\|  |V|^{1/2} G_{s,\d}\star(V^{1/2} v) \|_{L ^2(\R^N)}^2\\
 &\leq \int_{\R^N}| V (x) | \left(\int_{\R^N}|V|^{1/2}(y)| v(y)| G_{s,\d}((x-y)/\d)dy\right)^2 \,dx\\
 &\leq  \int_{\R^N}| V (x) |  \left(\int_{\R^N}|V(y)| G_{s,\d}((x-y)/\d)dy \int_{\R^N}|v(z)|^2 G_{s,\d}((x-z)/\d)dz \right) \,dx.
\end{align*}
Using a change of variable and \eqref{eq:estim-Bessel}, for $x\in \R^N$, we get
\begin{align*}
\int_{\R^N}|V(y)| G_{s,\d}(x-y)dy& \leq  c \eta_V(\d)+ \int_{\d\leq |x-y|}|V(y)| G_{s,1}((x-y)/\d)\,dy \\
&\leq c \eta_V(\d)+\sum_{i=1}^\infty \int_{i\d\leq   |x-y|\leq (i+1) \d}|V(y)| G_{s,1}((x-y)/\d)\,dy\\
&\leq  c \eta_V(\d)+\sum_{i=1}^\infty \d^{N} \int_{i\leq   |z|\leq i+1 }|V(\d z+ x)| G_{s,1}(z)\,dz\\
&\leq c  \eta_V(\d)+ C\sum_{i=1}^\infty i^{-N-1+s} \exp(-i/2)  \d^{N} \int_{i\leq   |z|\leq i+1 }|V(\d z+ x)|  \,dz.
\end{align*}
For every fixed $i$, we cover the annulus $A_i:= \{i \leq |y|\leq {i+1} \}$ by $  {n(i)}$ balls $B_1(z_j)$, with $z_j\in A_i $,  $n(i)\leq C i^{N-1}$ and $C$ a positive constant only depending on $N$. Letting $\rho_i:= i^{-N-1+s} \exp(-i/2) $, for every $x\in \R^N$ and $\d\in (0,1/2]$, we then have
\begin{align*}
\int_{\R^N}|V(y)| G_{s,\d}(x-y)dy
&\leq c  \eta_V(\d)+ C \sum_{i=1}^\infty \rho_i \sum_{j=1}^{n(i)} \d^{N}\int_{|z-z_j|\leq 1 }|V(\d z+ x)|  \,dz\\
&= c \eta_V(\d)+C \sum_{i=1}^\infty \rho_i\sum_{j=1}^{n(i)}   \int_{|y-x-\d z_j|\leq \d }|V(y)|  \,dz\\
&\leq  c  \eta_V(\d)+ C \sum_{i=1}^\infty \rho_i \sum_{j=1}^{n(i)}   \o_s(\d)^{-1} \int_{|y-x-\d z_j|\leq \d }|V(y)| \o_s(|y-x-\d z_j|)  \,dz\\
&\leq   c  \eta_V(\d)\left( 1+ C \sum_{i=1}^\infty i^{-2+s} \exp(-i/2)  \right) \o_s(\d)^{-1} \leq c \eta_V(\d).
%
\end{align*}
We then get, for every $\d\in (0,1/2]$
\begin{align*}
 \| LL^* v \|_{L ^2(\R^N)}^2 
 &\leq  c \eta_V(\d)   \int_{\R^N}| V (x) |    \int_{\R^N}|v(z)|^2 G_{s,1}((x-z)/\d)dz   \,dx \leq  c^2 \eta_V(\d)^2       \int_{\R^N}|v(z)|^2  dz  .
\end{align*}
 That is \eqref{eq:claim-coerciv} as claimed.\\ 
Since $\|L L^*\|=\|L\|^2$,  it then follows that
 \begin{align}\label{eq:L-continu-L2}
 \| Lv \|_{L ^2(\R^N)}^2   &\leq  c \eta_V(\d)      \| v \|_{L ^2(\R^N)}^2   .
\end{align}
Now given $u\in H^s(\R^N)$ and $\d\in (0,1]$, we plug $v= (-\D+(\d/2)^{-2})^{s/2} u\in L^2(\R^N)$ in \eqref{eq:L-continu-L2}, and noting that $ \| v \|_{L ^2(\R^N)}^2\leq c  \left( [u]_{{H}^s(\R^N)}^2+(\d/2)^{-2}\| u\|_{L ^2(\R^N)}^2 \right) $, for some positive constant $c=c(N,s)$. This with the fact that that $\eta_V(\d/2)\leq \eta_V(\d)$ give
 \eqref{eq:Schetcher2}, if $V\in L^\infty(\R^N)$.\\

\noindent
\textbf{Step 2:} For $V\in  L ^1_{loc}(\R^N)$, we consider $V_k=\min(|V|, k)$, for $k\in \N$. Thence since $\eta_{V_k}\leq \eta_{V}$, we get  \eqref{eq:Schetcher2} by Fatou lemma.
\end{proof}
 
 The following energy estimate is a consequence of the above  coercivity result and a nonlocal Caccioppoli-type inequality proved in an   appendix, Section \ref{s:appnd-2-techinacality}.
\begin{lemma} \label{lem:from-caciopp-ok}
We consider $\O$ an open set with $0\in \de\O$ and $K$ satisfying  \eqref{eq:def-a-anisotropi}.
Let $v\in H^s (\R^N)  $  and $V,f\in \cK_s$ satisfy
\be\label{eq:Dsv-eq-V-f}
\cL_{K}  v+ Vv=  f \qquad\textrm{ in $  B_{2R}\cap\O$}\qquad\textrm{ and }\qquad v=0 \qquad\textrm{ in  $  B_{2R}\cap\O^c$. }
\ee
Then there exists   $\ov C=\ov C(N,s, \k)>0$  such that    for every $\e>0$ and every $\d\in (0, 1]$, there exists $C=C(\e,\d, s,N, \k)$ such that   
\begin{align*}
\left\{\k-\right.& \e  \ov C (1+ \left. \eta_f(\d))- \ov C\eta_V(\d) \right\}\int_{\R^{2N}}(v(x)-v(y))^2\vp^2_R(y) \mu_1(x,y)\,dxdy\leq   C  \eta_f(1)  \|\vp_R\|^2_{H^s(\R^N)}\\    
&+ C \left( \eta_V(1)+  \eta_f(1) + 1 \right)   \int_{\R^N}  \frac{ R^N |v(x)|^2}{R^{N+2s}+|x|^{N+2s}}\,dx +    C \left( \eta_V(1)+ \eta_f(1) \right)   \|v\vp_R\|^2_{L^2(\R^N)}.
\end{align*}
\end{lemma}
\begin{proof}
By Lemma \ref{lem:caciopp} and \eqref{eq:Dsa-abs-estim-C-infty-c},  we  get
\begin{align}\label{eq:from-cacciop-to-coerciv}
 (\k-\e)\int_{\R^{2N}}(v(x)-v(y))^2&\vp^2_R(y) \mu_1(x,y)\,dxdy=  \int_{\R^N} | V(x)| |v(x) \vp_R(x)|^2\, dx \nonumber\\
& +      \int_{\R^N} |f(x)| |v(x)| \vp_R^2(x)\, dx  + C_\e  \int_{\R^N}  \frac{ R^N |v(x)|^2}{R^{N+2s}+|x|^{N+2s}}\,dx. 
\end{align}
Thanks to Lemma \ref{lem:coerciv},  \eqref{eq:Dsa-abs-estim-C-infty-c} and the fact that $\int_{\R^N}(\vp_1(x)-\vp_1(y))^2\mu_1(x,y)\,dy\leq C(1+|x|^{N-2s})$ for every $x\in \R^N$, we have  
\begin{align}\label{eq:used-coerciv-V}
 \int_{\R^N} |  V(x)| |v(x)\vp_R(x)|^2  \,dx   & \leq \ov C \eta_V(\d)    \int_{\R^{2N}}((v\vp_R)(x)-(v\vp_R)(y))^2  \mu_1(x,y)\,dxdy \nonumber\\
 &\quad+   C \eta_V(\d)    \|v\vp_R\|^2_{L^2(\R^N)} \nonumber \\
&\leq  \ov C \eta_V(\d)    \int_{\R^{2N}}(v(x)-v(y))^2\vp^2_R(y) \mu_1(x,y)\,dxdy \nonumber\\
&\quad+C \eta_V(\d)   \int_{\R^N}  \frac{ R^N |v(x)|^2}{R^{N+2s}+|x|^{N+2s}}\,dx +   C \eta_V(\d)    \|v\vp_R\|^2_{L^2(\R^N)}  .
\end{align}
Similarly,  by Young's inequality,  \eqref{eq:Schetcher2} and \eqref{eq:Dsa-abs-estim-C-infty-c}, we get 
\begin{align*}
&  \int_{\R^N} |  f(x)| |\vp_R(x)|^2|v(x)| \,dx  \leq \e     \int_{\R^N} | f(x)|   | v(x) \vp_R(x)|^2   dx+C_\e  \int_{\R^N} | f(x)| \vp_R(x)^2  \,dx \nonumber\\
&\leq   \e   \ov C \eta_f(\d)   \int_{\R^{2N}}((v\vp_R)(x)-(v\vp_R)(y))^2 \mu_1(x,y)\,dxdy+  C_\e   \eta_f(\d)  \|v\vp_R\|^2_{L^2(\R^N)}\nonumber \\
&\quad+  C_\e     \eta_f(\d)  \|\vp_R\|^2_{H^s(\R^N)} \nonumber\\
&\leq    \e  \ov C   \eta_f(\d)   \int_{\R^{2N}}(v(x)-v(y))^2\vp^2_R(y) \mu_1(x,y)\,dxdy +C    \eta_f(\d)   \int_{\R^N}  \frac{ R^N |v(x)|^2}{R^{N+2s}+|x|^{N+2s}}\,dx \nonumber\\
&\quad +   C   \eta_f(\d)    \|v\vp_R\|^2_{L^2(\R^N)}  + C_\e    \eta_f(\d)  \|\vp_R\|^2_{H^s(\R^N)}.
\end{align*}
Using the above estimate and  \eqref{eq:used-coerciv-V} in \eqref{eq:from-cacciop-to-coerciv} and using the monotonicity of $\eta_V$ and $\eta_f$, we get the result. 

\end{proof}
We close this section with the following result.
 \begin{lemma} \label{lem:convergence-very-weak}
We consider $\O$ an open set with $0\in \de\O$ and $K$   satisfying  \eqref{eq:Kernel-satisf}.
Let $v\in H^s_{loc}(B_{2R})\cap \cL^1_s $  and $V,f\in \cK_s$ satisfy
\be\label{eq:Dsv-eq-V-f-veary-weak}
\cL_{K}  v+ Vv=  f \qquad\textrm{ in $  B_{2R}\cap\O$}\qquad\textrm{ and }\qquad v=0 \qquad\textrm{ in  $  B_{2R}\cap\O^c$. }
\ee
Then for every $\psi\in C^\infty_c(B_{R}\cap\O )$, we have 
\begin{align*} 
\left|\int_{\R^{2N}}(v(x)-v(y))(\psi(x)-\psi(y))K(x,y)\, dxdy \right| &\leq  C \eta_V(1) \left( \|v\vp_R\|_{H^s(\R^N)}^2 +   \|\psi\|_{H^s(\R^N)}^2     \right)\\
& + C  \eta_f(1)    \left(  \|\vp_R\|_{H^s(\R^N)}^2 + \|\psi\|_{H^s(\R^N)}^2     \right),
%
\end{align*}
where $C>0$ is a constant,  only depending on $N$ and $s$.
\end{lemma}
\begin{proof}
Testing the equation \eqref{eq:Dsv-eq-V-f-veary-weak} with $\psi \in C^\infty_c(B_{R}\cap\O )$ and using Young's inequality, we get
\begin{align*} 
&\frac12\left|\int_{\R^{2N}}(v(x)-v(y))(\psi(x)-\psi(y))K(x,y)\, dxdy\right|\\
&\leq    \int_{\R^N  } |  V (x)| | v(x)\psi(x)| \vp_R (x) \, dx +  \int_{\R^N  } | f(x)| |\psi(x)|  \vp_R(x) \, dx \\
&\leq 2 \int_{\R^N  } | V(x)| |v (x)\vp_R (x)|^2\,dx +  2 \int_{\R^N  } | V(x)| \psi^2(x)   \, dx+    2   \int_{\R^N  } | f(x)| \psi^2(x)  \, dx+ 2      \int_{\R^N  } | f(x)| \vp_R^2(x)    \, dx  .
\end{align*}
Hence using  Lemma  \ref{lem:coerciv}, we conclude that
\begin{align*} 
\left|\int_{\R^{2N}}(v(x)-v(y))(\psi(x)-\psi(y))K(x,y)\, dxdy \right| &\leq  C  \eta_V(1) \left( \|v\vp_R\|_{H^s(\R^N)}^2 +   \|\psi\|_{H^s(\R^N)}^2     \right)\\
& + C  \eta_f(1)  \left( \|\psi\|_{H^s(\R^N)}^2 +   \|\vp_R\|_{H^s(\R^N)}^2     \right),
\end{align*}
which finishes the proof.
\end{proof} 
\section{Interior and boundary growth estimates}\label{s:L2-growths}
We recall the Morrey  space already introduced in the first section, for $\b\in [0,2s)$, defined as   
$$
\cM_\b :=\left\{f \in L^1_{loc} (\R^N)\,:\,  \|f\|_{ \cM_\b}:=\sup_{r\in (0,1), x\in \R^N} r^{\b-N}\int_{  B_r(x)} |f(y)| \, dy<\infty\right\}.
$$
Let $f\in \cM_\b$ and define $f_{r,x_0}(x)=r^{2s}f(rx+x_0)$ for $x_0\in \R^N$ and $r>0$. Recalling \eqref{eq:def-eta_f},  an important property of $\eta_f$  we will use frequently in the following is that, for every $x_0\in\R^N$ and $r\in (0,1]$, we have 
\be\label{eq:eta-f-scale-translate} 
 \eta_{f_{r,x_0}}(1)\leq C    \|f_{r,x_0} \|_{\cM_\b} \leq  C r^{2s-\b}  \|f\|_{\cM_\b} ,
 \ee
 with $C$ a positive constant, only depending on $N,s,$ and $\b$. The  first inequality in  \eqref{eq:eta-f-scale-translate} can be easily checked by change of variables and using   summations over annuli with small thickness.   We note that \eqref{eq:eta-f-scale-translate} and Lemma \ref{lem:coerciv} show  that functions $V,f\in \cM_\b$ satisfy   RTCP of order $\b$ (see  \eqref{eq:resc-coerciv-int}) as  mentioned in the first section.
\subsection{Interior growth estimates for solutions to Schr\"odinger equations}
The next  result is merely classical but we add the proof for the sake of completeness.
\begin{lemma}\label{lem:L2grothw-mean}
Let $u\in L^2_{loc}(\R^N)$ and $\a>0$.
\begin{enumerate}
\item  Suppose that
\be \label{eq:assum-L2growth-up}
 \|u-u_{B_{\rho  } }\|_{L^2(B_{\rho  })} \leq \ov C\rho^{N/2+\a}  \qquad\textrm{ for every $\rho\in[1,\infty)$.}
\ee
Then 
$$
   \| u-u_{B_{1  }} \|_{L^2(B_{\rho})}\leq  C  \ov C \rho^{N/2+\a} \qquad\textrm{ for every $\rho\in[1,\infty)$,}
$$
with $C$  depends only  on $N$ and $\a$.   \\
\item Suppose that $0$ is a Lebesgue point of $u$ and 
\be  \label{eq:assum-L2growth-down}
 \|u-u_{B_{\rho } }\|_{L^2(B_{\rho  })} \leq \ov C \rho^{N/2+\a}  \qquad\textrm{ for every $\rho\in(  0,1)$.}
\ee
Then 
$$
   \| u-u(0) \|_{L^2(B_{\rho})}\leq   C  \ov C  \rho^{N/2+\a}         \qquad\textrm{ for every $\rho\in (0,1)$},
$$
  with $C$ depends only  on $N$ and $\a$.
\end{enumerate}

\end{lemma}
\begin{proof}
First, to prove $(i)$, we note that, for  every $\rho\geq 1$, 
$$
|{B_{\rho} } |^{1/2} | u_{B_{\rho} } -u_{B_{2\rho} } |\leq \| u-u_{B_{\rho} }\|_{L^2(B_{\rho})}+ \|u-u_{B_{2\rho} } \|_{L^2(B_{2\rho})}\leq 2 \ov C  \rho^{N/2+\a} .
$$
 Therefore, for $\rho= 2^m $, with $m\geq 1$, we get 
 \begin{align*}
| u_{B_{\rho }} -u_{B_{1 }} | \leq \sum_{i=0}^{m-1} | u_{B_{2^{i}  } } -u_{B_{2^{i+1}  } } |& \leq  2 \ov C   \sum_{i=0}^{m-1}      2^{i\a} \leq  C \ov C \rho^{\a} ,
 \end{align*}
  where $C$ is independent on $m,  \rho$ and $u$.  Next, if $m$ is the smallest integer for which, $2^{m-1}\leq \rho\leq 2^m$, then  using \eqref{eq:assum-L2growth-up} and the above estimate,  we conclude that
  $$
  \| u-u_{B_1}\|_{L^2(B_{\rho})}\leq  \| u-u_{B_{2^m}} \|_{L^2(B_{2^m})}+|{B_{2^m} } |^{1/2} | u_{B_{2^m} } -u_{B_1}  |\leq   C \ov C  \rho^{N/2+\a}.
 $$ 
For $(ii)$,   by assumption, we have 
$$
|{B_{\rho/2} } |^{1/2} | u_{B_\rho } -u_{B_{\rho/2} } |\leq \| u-u_{B_\rho }\|_{L^2(B_{\rho})}+ \|u-u_{B_{\rho/2} } \|_{L^2(B_{\rho/2})}\leq 2 \ov C   \rho^{N/2+\a}.
$$
 Therefore 
 \begin{align*}
| u_{B_{\rho }} -u(0)  | \leq \sum_{i=0}^{\infty} | u_{B_{2^{-i} \rho} } -u_{B_{2^{-i} \rho/2} } |& \leq  2 \ov C   \sum_{i=1}^{\infty}  (2^{-i} \rho)^{\a}   \leq C \ov C  \rho^{\a}    . 
 \end{align*}
Using this and  \eqref{eq:assum-L2growth-down}, we obtain.
 $$
  \| u-u(0)\|_{L^2(B_{\rho})}\leq  \| u-u_{B_{\rho}} \|_{L^2(B_{\rho})}+|{B_{\rho} } |^{1/2} | u_{B_{\rho} } -u(0)  |\leq   C \ov C  \rho^{N/2+\a},
 $$
 where $C$ depends only on $N$ and $\a$. 
\end{proof}
Let   $a$ satisfy \eqref{eq:def-a-anisotropi} and $K\in \scrK(\l,a,\k)$ (satisfy \eqref{eq:Kernel-satisf})   $V,f \in  \cM_\b $ , we define  the set of  solutions to the Schr\"odinger equations with entries $V$ and $f$ by
$$
\cS_{K, V,f}:=\left\{ u\in   {H}^s (B_2)\cap L^2(\R^N)\,:\,  \cL_{K} u + Vu = f \quad\textrm{in $B_2$}   \right\},
$$
and we  note that this set is nonempty thanks to Lemma \ref{lem:coerciv} and a   direct minimization argument. In fact  this set is nonempty for all $f,V\in \cK_s$ for the same reason.\\
We consider the class of (normalized)   potentials    
\be \label{eq:def-cF-gamma}
\cV_{\b}:=\left\{V\in \cM_\b\,:\, \|V\|_{\cM_\b}\leq 1 \right\}.
\ee
Having these notations in mind, we now state the following  result.
 \begin{proposition} \label{prop:bound-Kato-abstract}
Let  $s\in (0,1)$, $\b\in [0,2s)$,       $\a\in (0,\min (1,2s-\b)) $ and $\L,\k>0$. Then there exists $\e_0>0$  and $C  >0$ such that  for every $\l:\R^{2N}\to [0,\k^{-1}] $  satisfying  $\|\l\|_{L^\infty(B_2\times B_2)}<\e_0$,  $a$ satisfying \eqref{eq:def-a-anisotropi},  $K\in \scrK(\l,a,\k)$,  $V\in \cV_\b$, $f\in \cM_\b$, $u\in \cS_{K,V,f}$ and   for every $r>0$,   we have 
\be\label{eq:abs-res-Propo}
     \sup_{x\in B_1} \|u-u_{B_r(x)}\|_{L^2(B_r(x))}^2\leq  C r^{N+2\a }(  \|u\|_{  L^2(\R^N) }+ \|f\|_{\cM_\b})^2.
\ee
\end{proposition}
 \begin{proof}
 The proof of \eqref{eq:abs-res-Propo} will be divided into two steps. Due to the presence of the potential $V$, the set $\cS_{K,V,f}$ might not be invariant   when  adding constants to its elements. As a way out to this difficulty,    we prove first a uniform estimate of the form $|u_{B_r(x)}|\leq C r^{-\varrho}( \|u\|_{  L^2(\R^N) }+ \|f\|_{\cM_\b})$,   for all   $\varrho>0$. Once we get this,  we complete the proof of \eqref{eq:abs-res-Propo} in the second step.\\
 
 \noindent
 \textbf{Step 1:}  We claim that for every $\varrho\in(0,1/2)$, there exist  $C>0$ and a  small number $\e_0>0$ such that for every  $\l:\R^{2N}\to [0,\k^{-1}] $  satisfying  $\|\l\|_{L^\infty(B_2\times B_2)}<\e_0$, every  function   $a $ satisfying \eqref{eq:def-a-anisotropi}, $K\in \scrK(\l,a,\k)$, $V\in \cV_\b$, $f\in \cM_\b$, $u\in \cS_{K,V,f}$, and $r>0$, we have 
 \be\label{first-estim-V}
\sup_{x\in B_1}     \|u \|_{L^2(B_r(x))}^2\leq  C r^{N-2\varrho } (  \|u\|_{  L^2(\R^N) }+ \|f\|_{\cM_\b})^2.
\ee
Assume that  \eqref{first-estim-V} does not hold, then there exists $\varrho\in (0,1/2)$ such that for every   $n\in \N$, we can find  $\l_n:\R^{2N}\to [0,\k^{-1}] $  satisfying  $\|\l_n\|_{L^\infty(B_2\times B_2)}<\frac{1}{n}$,  $a_n$ satisfying \eqref{eq:def-a-anisotropi}, $K_{a_n}\in  \scrK(\l_n,a_n,\k)$,  $V_n\in \cV_\b$,  $f_n\in \cM_\b$, $u_n\in \cS_{K_{a_n}, V_n,f_n}$, with $ \|u_n\|_{  L^2(\R^N) }+ \|f_n\|_{\cM_\b}\leq 1$ and  $\ov r_n>0$,  such that 
\be\label{eq:ov-r-n-un-larger-n-eps}
\ov r^{-N+2\varrho }_n  \sup_{x\in B_1}  \|u_n \|_{L^2(B_{\ov r_n}(x))}^2>n.
\ee
We consider the (well defined, because $\|u_n\|_{ L^2(\R^N)}\leq 1$) nonincreasing function $\Theta_n: (0,\infty)\to [0,\infty)$ given by
\be\label{eq:un-fn-normal}
\Theta_n(\ov r)=\sup_{ r\in [ \ov r,\infty)}   r^{-N+2\varrho } \sup_{x\in B_1}   \|u_n \|_{L^2(B_{ r}(x))}^2.
\ee
Obviously by \eqref{eq:ov-r-n-un-larger-n-eps},   
\be\label{eq:The-n-geq-n}
  \Theta_n(\ov r_n)> n .
\ee
Clearly,  there exists $r_n\in [\ov r_n,\infty)$ such that 
$$
\Theta_n( r_n)\geq   r^{-N+2\varrho }_n  \sup_{x\in B_1}  \|u_n \|_{L^2(B_{ r_n}(x))}^2 \geq (1-1/n)\Theta_n(\ov r_n)\geq (1-1/n)\Theta_n( r_n),
$$
where we used the monotonicity of $\Theta_n$ for the last inequality, while  the first inequality  comes from the definition of $\Theta_n$.
In particular, thanks to  \eqref{eq:The-n-geq-n},  $\Theta_n( r_n)\geq (1-1/n)n$. Now  since $ \|u_n \|_{L^2(\R^N)}\leq 1$,  we have that $  r^{-N+2\varrho }_n \geq (1-1/n)n$, so that $r_n\to 0$ as $n\to \infty$.  Moreover by \eqref{eq:The-n-geq-n}, it is clear that, there exists $x_n\in B_1$ such that 
\be\label{eq:rn-norm-u-n geq-eps}
  r^{-N+2\varrho }_n   \|u_n \|_{L^2(B_{ r_n}(x_n))}^2  \geq   (1-1/n-1/2)\Theta_n( r_n) .
\ee
 We now   define the blow-up sequence  of functions
$$
 {w}_n(x)=  \Theta_n(r_n)^{-1/2} r_n^{\varrho}   u_n(r_n x+x_n), 
$$
which, by \eqref{eq:rn-norm-u-n geq-eps},   satisfy
\be \label{eq:w-n-nonzero}
 \|w_n\|_{L^2(B_{1})}^2\geq \frac{3}{4} \qquad\textrm{ for every $n\geq2$.}
\ee
In view of  \eqref{eq:un-fn-normal}, we have  that
 \begin{align*}
 \|w_n\|_{L^2(B_{R})}^2&=   \Theta_n(r_n)^{-1}   r_n^{-N+2\varrho}   \|u_n\|_{L^2(B_{r_nR}(x_n) )}^2\\
&\leq  \Theta_n(r_n)^{-1}   r_n^{-N+2\varrho}   \Theta_n(r_nR)   (r_nR)^{N-2\varrho}  
\leq R^{N-2\varrho},
\end{align*}
where we have used the monotonicity of $\Theta_n$ for the last inequality.
Consequently, 
\be\label{eq:groht-w-n-eps}
 \|w_n\|_{L^2(B_{R})}^2\leq    R^{N-2\varrho}  \qquad\textrm{ for every $R\geq 1$ and $n \geq 2$.}
\ee
We define
$$
\ov f_{n} (x) := \Theta_n(r_n)^{-1/2}      r_n^{2s} f_n(r_n x+x_n) \qquad\textrm{ and } \qquad \ov V_{n} (x) :=  r_n^{2s} V_n(r_n x+x_n).
$$
Because $u_n\in \cS_{K_{a_n}, V_n,f_n} $, it is plain that 
$$
\cL_{K_{n}} w_n   +\ov  V_n w_n=   r_n^\varrho \ov f_n   \qquad\textrm{ in $B_{1/2{r_n}} $,}
$$
where
\be\label{eq:K_nnnnnnn}
K_n(x,y)=r_n^{N+2s}K_{a_n}(r_nx+x_n,r_ny+x_n). 
\ee
Clearly $K_n$ satisfies \eqref{eq:Kernel-satisf}. 
Therefore  applying Lemma \ref{lem:from-caciopp-ok}  and using  \eqref{eq:groht-w-n-eps}, for every  $1<M<\frac{1}{2r_n}$,  we  get
\begin{align*}
&\left\{\k-  \e  \ov C (1+ r_n^\varrho \eta_{\ov f_n}(1))- \ov C  \eta_{\ov V_n}(1) \right\}[w_n]_{H^s(B_M)}^2\leq       C  r_n^\varrho \eta_{\ov f_n}(1)  \|\vp_M\|^2_{H^s(\R^N)}\\    
&+ C \left(  \eta_{\ov V_n}(1)+r_n^\varrho \eta_{\ov f_n}(1) + 1 \right)   \int_{\R^N}  \frac{ M^N |w_n(x)|^2}{M^{N+2s}+|x|^{N+2s}}\,dx +    C \left( \eta_{\ov V_n}(1)+r_n^\varrho \eta_{\ov f_n}(1)\right)   \|w_n\|^2_{L^2(B_{2M})}.
\end{align*}
By   \eqref{eq:eta-f-scale-translate},  we have  that $ \eta_{\ov V_n}(1)+ \eta_{\ov f_n}(1) \leq   C r_n^{2s-\b }$   (recalling that $ \Theta_n(r_n)^{-1}\leq 1$). Hence, there exists a constant $C(M)$ independent on $n\geq 2$ such that 
\begin{align}\label{eq:semi-norm-w-n-bounded-eps}
\left(\k -\e \ov C(1 + r_n^{2s-\b+\varrho }) - \ov C r_n^{2s-\b}  \right) &[w_n]_{H^s(B_M)}^2\leq       C(M) .
\end{align}
Therefore   provided $\e$ is small and $n$ is   large enough, we deduce that $w_n$ is bounded in $H^s_{loc}(\R^N)$. Hence by Sobolev embedding, up to a subsequence,  $w_n$ converges strongly,  in $L^2_{loc}(\R^N)$,  to some  $w\in H^s_{loc}(\R^N)$.     In addition by \eqref{eq:groht-w-n-eps}, we have that $v_n \to v$ in $\cL^1_s$.
Moreover, by \eqref{eq:w-n-nonzero}  and \eqref{eq:groht-w-n-eps},   we deduce that 
\be\label{eq:w-nonzero-eps}
  \|w\|_{L^ 2(B_1)}^2 \geq \frac{3}{4} \qquad\textrm{ and } \qquad  \|w\|_{L^2(B_{R})}^2\leq    R^{N-2\varrho}  \qquad\textrm{ for every $R\geq 1$.}
\ee
We let
  $\psi\in C^\infty_c( B_M)$, with $M<\frac{1}{2r_n}$. By Lemma \ref{lem:convergence-very-weak}, \eqref{eq:semi-norm-w-n-bounded-eps}    and \eqref{eq:groht-w-n-eps}, we get
\begin{align}\label{eq:first-estim-to-Liou} 
&\left| \int_{\R^{2N} } (w_n (x)-w_n(y))(\psi(x)   -\psi(y)) {K_{n}(x,y)}\, dxdy \right| \nonumber\\
&\leq  C r_n^{2s-\b  } \left( \|w_n\vp_M\|_{H^s(\R^N)}^2 +   \|\psi\|_{H^s(\R^N)}^2     \right)   + r_n^{2s-\b +\varrho}   \left( \|\psi\|_{H^s(\R^N)}^2 +   \|\vp_M\|_{H^s(\R^N)}^2     \right)  \nonumber\\
& \leq r_n^{2s-\b }  C (M)  .
\end{align}
Next, we observe that  $K_n\in \scrK(\ov \l_n,a_n,\k)$, with $\ov \l_n(x,y)=\l_n (r_nx+x_n,r_ny+x_n)$ (see  \eqref{eq:Kernel-satisf}). On the other hand 
$$
 \|\ov \l_n\|_{L^\infty(B_{1/r_n}\times B_{1/r_n})}= \|\l_n\|_{L^\infty(B_{1}(x_n)\times B_{1}(x_n))}\leq  \|\l_n\|_{L^\infty(B_2\times B_2)}\leq \frac{
 1}{n}.
$$
In view of this and \eqref{eq:first-estim-to-Liou},  by  Lemma \ref{lem:Ds-a-n--to-La}, as $n\to \infty$,  we have that   
$$
\frac{1}{2} \int_{\R^{2N}}(w_n(x)-w_n(y))(\psi(x)-\psi(y)K_n(x,y)\, dxdy\to \int_{\R^N} w L_b \psi(x)\,dx=0, 
$$
 where $b$ is the weak-star limit of $a_n$, which satisfy \eqref{eq:coerciv-RS}.
We then conclude that $L_b w =0 \quad\textrm{ in $\R^N$}.$
Now    Lemma \ref{lem:Liouville} implies that  $w$ is an affine function. This is clearly in contradiction   with       \eqref{eq:w-nonzero-eps} since $\varrho>0$.\\
 
  \bigskip 
 \noindent
 \textbf{Step 2: } 
Assuming that \eqref{eq:abs-res-Propo} does not hold true,  then as in the first step, we can find sequences $\l_n:\R^{2N}\to [0,\k^{-1}] $  satisfying  $\|\l_n\|_{L^\infty(B_2\times B_2)}<\frac{1}{n}$,  $a_n$ satisfying \eqref{eq:def-a-anisotropi},  $K_{a_n}\in \scrK(\l_n,a_n,\k)$,  $x_n\in B_1$,  $ V_n\in \cV_\b $, $f_n\in \M_\b$,  $u_n\in \cS_{K_{a_n},V_n,f_n}$, with $ \|u_n\|_{  L^2(\R^N) }+\|f_n\|_{\cM_\b}\leq 1$ and $r_n\to 0$,   such   that 
 \be\label{eq:resclae-situation-pwb}
    r^{-N-2\a }_n     \|u_n-(u_n)_{B_{r_n}(x_n)}\|_{L^2(B_{r_n}(x_n))}^2\geq  \frac{1}{8} \Theta_n(r_n) .
\ee
Here, for every $n\geq 2$,     $ \Theta_n:(0,\infty)\to [0,\infty)$ is a  nonincreasing function   satisfying 
\be \label{eq:estim-u-min-u-n-by-Theta}
\Theta_n( r)  r^{N+2\a }  \geq      \|u_n-(u_n)_{B_r(x_n)}\|_{L^2(B_r(x_n))}^2 \qquad\textrm{ for every $r>0$}
\ee
and   $\Theta_n(r_n)\geq n/2$.
We define   
$$
 {v}_n(x)=  \Theta_n(r_n)^{-1/2}r_n^{-\a}   u_n(r_n x+x_n)- \Theta_n(r_n)^{-1/2} r_n^{-\a}  \frac{1}{|B_1|}\int_{B_1}  u_n(r_n x+x_n)\,dx, 
$$
so that 
\be \label{eq:v-n-nonzero-abs}
 \|v_n\|_{L^2(B_{1})}^2\geq \frac{1}{8} \qquad \textrm{ and} \qquad \int_{B_1} v_n(x)\, dx=0. 
\ee
\bigskip
\noindent
\textbf{Claim:} There exists $C= C(  s,  N,\a)>0$     such that  
\be\label{eq:groht-v-n-abs}
 \|v_n\|_{L^2(B_{R})}^2\leq  C   R^{N+2\a}  \qquad\textrm{ for every $R\geq 1$ and $n \geq 2$.}
\ee
To prove this claim, we   note that by a change of variable, we have
\begin{align}\label{eq:v-n-diff-u-BrR}
 \|v_n\|_{L^2(B_{R})}^2&=   \Theta_n(r_n)^{-1}     r_n^{-N-2\a}  \|u_n-(u_n)_{B_{r_n}(x_n)  }\|_{L^2(B_{r_nR}(x_n) )}^2.
\end{align}
Since, by    \eqref{eq:estim-u-min-u-n-by-Theta} and the monotonicity of $\Theta_n$,  
$$
\| u_n  -(u_n)_{B_{Rr_n}(x_n)  }\|_{L^2(B_{Rr_n}(x_n) )}^2\leq  (r_nR)^{N+2\a}    \Theta_n(R r_n)   \leq  r_n^{N+2\a}  \Theta_n( r_n) R^{N+2\a}     \qquad\textrm{ for every $R\geq 1$, }
$$
 it follows from Lemma \ref{lem:L2grothw-mean}$(i)$ that
$$
\| u_n  -(u_n)_{B_{r_n}(x_n)  }\|_{L^2(B_{Rr_n}(x_n) )}^2\leq  C r_n^{N+2\a}    \Theta_n( r_n)   R^{N+2\a}.
$$
Using this in \eqref{eq:v-n-diff-u-BrR},  we get \eqref{eq:groht-v-n-abs}
as  claimed.\\
Thanks to the choice of $\a<2s$, by    \eqref{eq:groht-v-n-abs}, we get  
\be \label{eq:good-growth-of-v_n-Kato-pw-0}
\|v_n\|_{\cL^1_s}\leq C.
\ee
Using the same notations as in \textbf{Step 1} for $\ov V_n$, $\ov f_n$ and $K_n\in \scrK(\ov \l_n,a_n,\k)$, we see that 
$$
\cL_{K_n} v_n   +\ov  V_n v_n=- \Theta_n(r_n)^{-1/2}  r_n ^{-\a}  \ov  V_n  A_n(r_n)    + r_n ^{-\a}   \ov f_n   \qquad\textrm{ in $B_{1/{2r_n}} $,}
$$
where  $A_n(r_n):= \frac{1}{|B_{r_n}|}\int_{B_{r_n}(x_n)} u_n(x)\, dx$. Note that from  \textbf{Step 1} and H\"older's inequality,      for every $\varrho\in (0, (2s-\b-\a)/2)$, we can find a  constant $C>0$ such that for every $n\geq 2$,
 \be\label{eq:mean-u-estim}
|A_n(r_n)|\leq  \frac{1}{|B_{r_n}|}\int_{B_{r_n}(x_n)} |u_n(x)|\, dx\leq C r^{-\varrho}_n.
 \ee 
We then define   
$$
F_n(x):=- \Theta_n(r_n)^{-1/2}     A_n(r_n)\ov V_n(x)+\ov f_n(x) .
$$
As above, we observe  that $ \eta_{\ov V_n}(1) \leq C r_n^{2s-\b }$, while by \eqref{eq:mean-u-estim}, we have    $ \eta_{F_n}(1)\leq C  r_n^{2s-\b-\varrho}$. 
On the other hand
\be \label{eq:v_n-solves-right-eq}
\cL_{K_n} v_n   +\ov  V_n v_n=  r_n ^{-\a}  F_n   \qquad\textrm{ in $B_{1/{2r_n}} $.}
\ee
By Lemma \ref{lem:cat-off-sol}, for $1<M<\frac{1}{2r_n}$,  we have 
$$
\cL_{K_n} ( \vp_Mv_n )  +\ov  V_n ( \vp_Mv_n ) =  r_n ^{-\a}   \vp_M F_n+ G_{v_n,M}   \qquad\textrm{ in $B_{M/{2}} $,}
$$
where $ \|G_{v_n,M} \|_{L^\infty(\R^N)}\leq C \|v_n\|_{\cL ^1_s}\leq C $, by  \eqref{eq:good-growth-of-v_n-Kato-pw-0}.
In view of \eqref{eq:groht-v-n-abs}, \eqref{eq:good-growth-of-v_n-Kato-pw-0} and Lemma \ref{lem:from-caciopp-ok}, we then get 
\begin{align*}
\left(  \k - \e \ov C(1+ r_n^{2s-\b-\varrho-\a} )-\ov C r_n^{2s-\b }  \right) &[\vp_M v_n]_{H^s(B_{M/4})}^2\leq       C(M) \qquad\textrm{ whenever  $1<M<\frac{1}{2r_n}$.}
\end{align*}
Consequently, provided $n$ is large enough and $\e$ small, we obtain
\be\label{eq:v-n-bnd-seminorm}
 [v_n]_{H^s(B_{M/4})}^2\leq       C(M) .
\ee
This with  \eqref{eq:groht-v-n-abs} imply that   $v_n$ is bounded in $H^s_{loc}(\R^N)$  and, up to a subsequence,  converges strongly, in $L^2_{loc}(\R^N)\cap \cL^1_{s }$,  to some $v\in H^s_{loc}(\R^N)$. Since $v_n$ satisfy \eqref{eq:v_n-solves-right-eq},   by Lemma \ref{lem:convergence-very-weak}, \eqref{eq:good-growth-of-v_n-Kato-pw-0} and \eqref{eq:v-n-bnd-seminorm},   we have
 \begin{align*} 
\Bigl|\int_{\R^{2N}}(v_n(x)  &    -v_n(y))(\psi(x)-\psi(y))K_n(x,y)\, dxdy \Bigr|\\
 &\leq  C r_n^{2s-\b-\a-\varrho} \left( \|v_n\vp_M\|_{H^s(\R^N)}^2 +    \|\psi\|_{H^s(\R^N)}^2 +   \|\vp_M\|_{H^s(\R^N)}^2     \right)\\
&\leq r_n^{2s-\b-\a-\varrho}  C (M)  
\end{align*}
for every $\psi\in C^\infty_c(B_M)$, with $M<\frac{1}{2r_n}$.
Letting $n\to \infty$ in the above inequality and using Lemma \ref{lem:Ds-a-n--to-La}, we find that  $L_b v=0$ in $ \R^N$, with $b$ the limit of $a_n$ in the weak-star topology of $L^\infty(S^{N-1})$. Moreover, from \eqref{eq:groht-v-n-abs}, we get 
$$
\int_{B_R} |v(x)|^2\, dx\leq   C    R^{N+2\a} .
$$
By  Lemma \ref{lem:Liouville},   $v$ is a constant function (because $\a<1$), which leads to a contradiction after passing to the limit in 
  \eqref{eq:v-n-nonzero-abs}.
\end{proof} 
\subsection{Uniform growth estimates at the boundary for solutions to Schr\"odinger equations}

Let $\O$ be an open subset of $\R^N$ such that $\de\O\cap B_2$ is a  $C^1$ hypersurface. We will assume that $0\in \de\O$ and that $\de\O$ separates $B_2$ into two domains.  As before,
for $K\in \scrK(\l,a,\k)$, $V\in \cV_\b$ and $f\in \cM_{\b}$,
we consider the (nonempty) set of solutions:
\be \label{eq:set-of-sol-Omega}
\cS_{K, V,f;\O}:=\left\{ u\in   {H}^s(B_2)\cap L^2(\R^N)\,:\, \cL_{K} u+ Vu = f \quad\textrm{in $B_2\cap \O$,  $\quad u =0$ \text{ in  $B_2\cap\O^c$}  } \right\}.
\ee
We have the following result.
 \begin{proposition} \label{prop:bdr-estim-Kato}		
Let  $s\in (0,1)$, $\b\in [0,2s)$,     $\a\in (0,\min (s,2s-\b)) $ and $\L,\k>0$.  Then   there exist $\e_1>0$ small and  $C>0$ such that for every $\l:\R^{2N}\to [0,\k^{-1}]$ satisfying $\|\l\|_{L^\infty(B_2\times B_2)}<\e_1$,    $a$ satisfying \eqref{eq:def-a-anisotropi}, ${K}\in \scrK(\l,a,\k) $,  $V\in \cV_\b$, $f\in \cM_{\b}$,  $u\in \cS_{K,V,f;\O}$  and for every $r>0$,    we have 
\be\label{eq:asser-inProp-bdr}
  \sup_{z\in B_{1} \cap\de\O}  \|u \|_{L^2(B_r(z))}^2\leq  C r^{N+2\a }(  \|u\|_{  L^2(\R^N) }+ \|f\|_{\cM_\b})^2.
\ee
\end{proposition}
 \begin{proof}
 As in the proof of Proposition \ref{prop:bound-Kato-abstract}, if   \eqref{eq:asser-inProp-bdr} does not hold, then we can a find sequence  of real numbers $r_n\to 0$, sequence of points    $z_n\in B_{1} \cap\de\O$,  sequences of   functions $\l_n$   satisfying $\|\l_n\|_{L^\infty(B_2\times B_2)}<\frac{1}{n}$,  $a_n$ satisfying \eqref{eq:def-a-anisotropi},  ${K_{a_n}} \in \scrK(\l_n,a_n,\k)$,  $V_n\in \cV_\b$, $f_n\in \cM_\b$ $u_n\in \cS_{K_{a_n},V_n,f_n;\O}$, with $ \|u_n\|_{  L^2(\R^N) }+\|f_n\|_{\cM_\b}\leq 1$,  such   that 
$$
    r^{-N-2\a }_n        \|u_n \|_{L^2(B_{r_n}(z_n))}^2\geq  \frac{1}{8} \Theta_n(r_n) ,
$$
where,  $ \Theta_n:(0,\infty)\to [0,\infty)$, is a  nonincreasing function   satisfying 
\be \label{eq:estim-u-n-by-Theta}
\Theta_n( r)  r^{N+2\a }  \geq        \|u_n \|_{L^2(B_r(z_n))}^2 \qquad\textrm{ for every $r>0$ and $n \geq 2$}
\ee
and $\Theta_n(r_n)\geq n/2$ for all integer  $n\geq 2$.
 We   define  
$$
 {v}_n(x)=  \Theta_n(r_n)^{-1/2}  r_n^{-\a} u_n (r_n x+z_n)     ,
$$
so that 
\be\label{eq:low-bnd-v-vn-Kato}
  \|v_n  \|_{L^2(B_1)}^2\geq \frac{1}{8}.
\ee
We also let 
$$
\O_n:=\frac{1}{r_n}(\O-z_n), 
$$
$$
\ov f_{n} (x) :=   \Theta_n(r_n)^{-1/2}    r_n^{2s} f_n (r_n x+z_n) \qquad \textrm{ and } \qquad \ov V_{n} (x) :=  r_n^{2s} V_n (r_n x+z_n).
$$
By   \eqref{eq:eta-f-scale-translate}, we get
\be\label{eq:et-of-n-ov-V-n}
\eta_{ \ov f_{n} } (1)+ \eta_{ \ov V_{n} } (1) \leq C r_n^{2s-\b },
\ee
 with    a constant $C=C(N,s,\b,\a )$.
It is clear that
\begin{align}\label{eq:v-n-satisf-system-bdr}
\begin{cases}
\cL_{K_n} v_n  +\ov V_nv_n= r_n^{-\a} \ov f_n&   \qquad\textrm{ in $B_{1/{2 r_n}}(-z_n)\cap \O_n$}\\
v_n=0&   \qquad\textrm{ in $B_{1/{2r_n}}(-z_n)\cap \O_n^c$},
\end{cases}
\end{align}
where 
$$
K_n(x,y)=r_n^{N+2s}K_{a_n}(r_nx+z_n,r_ny+z_n).
$$
Next, by the monotonicity of $\Theta_n$ and \eqref{eq:estim-u-n-by-Theta}, we get
\begin{align*}
 \|v_n  \|_{L^2(B_R)}^2 &=   \Theta_n(r_n)^{-1}   r_n^{-N-2\a}   \|u_n  \|_{L^2(B_{r_n R}(z_n))}^2\\
& \leq   \Theta_n(r_n)^{-1}  r_n ^{-N-2\a}    \Theta_n(R r_n)   (R r_n)^{N+2\a} .
\end{align*}
Hence,    for every $n\geq 2$,
\be\label{eq:growth-v-n}
 \|v_n  \|_{L^2(B_R)}^2 \leq     R^{N + 2\a}   \qquad\textrm{ for every $R\geq 1$.}
\ee
This with H\"older's inequality imply that   
\be \label{eq:good-growth-of-v_n-Kato-pw}
   \int_{B_R}|v_n(x)|dx \leq C   R^{N+\a} \qquad\textrm{ for every $R\geq 1$.}
\ee
From now on, we let $n_0$ large, so that $B_{1/(2r_n)}\subset B_{1/r_n}(-z_n)$ for every $n\geq n_0$.
Since $v_n$ satisfies \eqref{eq:v-n-satisf-system-bdr} and $K_n$ satisfies \eqref{eq:Kernel-satisf}$(i)$-$(ii)$,  by   Lemma \ref{lem:from-caciopp-ok}, \eqref{eq:et-of-n-ov-V-n} and \eqref{eq:growth-v-n},    there exists a constant $C(M)$ independent on $n\geq n_0$ such that 
  \begin{align}
[v_n]_{H^s(B_M)}^2\leq      C(M) ,
\label{eq:apply-cacciop-bdr1}
\end{align}
whenever $1\leq M\leq \frac{1}{2r_n}$.
We then  deduce that $v_n$ is bounded in $H^s_{loc}(\R^N)$. Hence by Sobolev embedding, \eqref{eq:good-growth-of-v_n-Kato-pw} and since $\a<2s$, we may assume that  the sequence  $v_n$ converges strongly,  in $L^2_{loc}(\R^N)\cap\cL^1_s$,  to some  $v\in H^s_{loc}(\R^N) $.    
Moreover, by \eqref{eq:low-bnd-v-vn-Kato}, we deduce that 
\be\label{eq:groth-v-and-low-bnd-Kato-pw}
  \|v\|_{L^2(B_1)}^2 \geq \frac{1}{8}.
\ee
Next, we note that $1_{\O_n\cap B_{1/(2r_n)}}\to 1_{H}$ in $L^1_{loc}(\R^N)$ as $n\to \infty$, where $H$ is a half-space, with $ 0\in  H$.  In fact $H=\{x\in \R^N\,:\, (x-\ov z)\cdot \nu(\ov z)>0\}$, where  $\ov z=\lim_{n\to \infty} z_n\in \de H$ and $\nu$ is the unit interior normal vector of $\de H$.
Now, we pick $\psi\in C^\infty_c( H\cap B_M)$, with $M<\frac{1}{2r_n}$. Since $\O$ is of class $C^1$, provided $n$ is large enough, we have that  $\psi \in  C^\infty_c( \O_n)$.  Therefore by Lemma \ref{lem:convergence-very-weak},  \eqref{eq:apply-cacciop-bdr1}, \eqref{eq:growth-v-n} and \eqref{eq:et-of-n-ov-V-n}, we obtain 
\begin{align*} 
\left| \int_{\R^{2N}}(v_n(x)-v_n(y))(\psi(x)-\psi(y))K_n(x,y)\, dxdy \right| &\leq  r_n^{2s-\b-\a } C(M) ,
\end{align*}
with $C(M)$ a constant not depending on $n\geq n_0$ and large. Denoting by $b$ the weak-star  limit of $a_n$, then by  Lemma \ref{lem:Ds-a-n--to-La},   we  get
$$
L_b v=0 \qquad\textrm{ in $H$}  \qquad \textrm{ and }\qquad v=0 \qquad \textrm{ on $\R^N\setminus  H$.}
$$
Furthermore by \eqref{eq:growth-v-n},    $ \|v\|_{L^2(B_R)}^2\leq C R^{N+2\a} $, for every $R\geq 1$. Since   $\a<s$,  it follows from  Lemma \ref{lem:Liouville}  that   $v= 0$,    which is impossible by \eqref{eq:groth-v-and-low-bnd-Kato-pw}.
\end{proof} 
\section{Interior and boundary H\"older regularity estimates}\label{s:int-and-ndr-reg}
\subsection{Interior H\"older  regularity}
We have the following regularity estimates.
\begin{corollary}\label{cor:int--Holder-reg-Morrey}
Let  $s\in (0,1)$, $\b\in [0,2s)$,     $\a\in (0,\min (1,2s-\b)) $, $\k,\L>0$. Let  $a$ satisfy \eqref{eq:def-a-anisotropi} and   $K\in \scrK(\l,a,\k)$.
  Let $f,V\in \cM_\b$ and  $u\in H^s (B_2)\cap \cL^1_s$   satisfy
$$
\cL_{K} u+Vu= f \qquad\textrm{ in $B_2$}.
$$
Then there exists $\e_0,C>0$, only depending on $ N,s,\b,\k,\a,\|V\|_{\cM_\b}$ and $\L$ such that   if $\|\l\|_{L^\infty(B_2\times B_2)}<\e_0$, then 
  $$
  \|u\|_{C^{\a}(B_{1})}\leq C \left( \| u\|_{L^2(B_{2})}+  \|u\|_{\cL^1_s}+ \|f\|_{\cM_\b} \right).
  $$
%
%
%
%
\end{corollary}
\begin{proof}
We let $x_0\in B_{3/2}$ and  $\d\in (0,1/8)$ and   we define  $\l_{\d}(x,y)={\l} (\d x+ x_0,\d y+ x_0)$ and 
 $K_{\d}(x,y)=\d^{N+2s}K(\d x+ x_0,\d y+ x_0)$. Then  $K_{\d}\in \scrK( \l_{\d}, a,\k)$. 
For $x\in B_{2}$, we define $u_\d(x)=u(\d x+ x_0)$, $f_\d(x)=\d^{2s}f(\d x+ x_0)$ and $V_\d(x)=\d^{2s} V(\d x+ x_0)$. Since $\d\in (0,1/8)$,  by direct computations, we get
\be \label{eq:u-delta-satisf-cor}
\cL_{K_{\d}} u_\d+V_\d u_\d= f_\d \qquad\textrm{ in $B_{8}$}.
\ee
By Lemma \ref{lem:cat-off-sol}, letting $v_\d:=\vp_{4}u_\d$ we have
\be\label{eq:v-delta-satisf-cor} 
\cL_{K_{\d}} v_\d+{V}_\d(x) v_\d= \ti {f}_\d \qquad\textrm{ in $B_{2}$},
\ee
with $ \|\ti f_\d\|_{\cM_\b}\leq \|f_\d\|_{\cM_\b}+ C_0  \|u_\d\|_{\cL^1_s} $ and 
$$
 \| V_\d\|_{\cM_\b}\leq   C  \d^{2s-\b}\| V\|_{\cM_\b} ,
 $$
 with $C_0$ depending only on $N,s,\k,\L$ and $\b$.
Hence, there exists $\d_0\in (0,1/8)$, only depending on $N,s,\k,\L,\b$ and $\|V\|_{\cM_\b}$,  such that   $ \|V_\d\|_{\cM_\b}\leq 1  $ for every $\d\in (0,\d_0)$. 
Obviously  $\|\l_{\d}\|_{L^\infty(B_2\times B_2)} \leq \|\l \|_{L^\infty(B_2\times B_2)}$ for every  $\d\in (0,\d_0)$.   Hence   by Proposition \ref{prop:bound-Kato-abstract}, there exists $\e_0>0$  and $C>0$ such that if $  \|\l \|_{L^\infty(B_2\times B_2)}<\e_0$, then     for every $x\in B_{1}$, $r>0 $ and $\d\in (0,\d_0)$, we have 
$$
 \| v_\d-(v_\d)_{B_r ( x)}\|_{L^2(B_r( x))}^2   \leq C   r^{N  +2\a}\left(  \|v_\d\|_{L^2(\R^N)}+ \|\ti f_\d\|_{\cM_\b} \right)^2  ,
 $$
where, the constant $C$ and $\e_0$  only depend  on $ N,s,\b,\a,\L,\|V\|_{\cM_\b}$ and $\k$.
By    Lemma \ref{lem:L2grothw-mean}$(ii)$,    for  almost all $ x\in B_{1}$ and for  all $r\in (0,1]$, we have 
\be\label{eq:est-sem-u-proof-cor}
  \| v_\d-v_\d ( x)\|_{L^2(B_r( x))}^2   \leq C   r^{N  +2\a}   \left(   \|v_\d\|_{L^2(\R^N)}+ \|\ti f_\d\|_{\cM_\b} \right)^2   .
\ee
In particular,
\be \label{eq:L-infty-bound-u-cor}
\|v_\d\|_{L^\infty(B_{1})}\leq C\left(   \|v_\d\|_{L^2(\R^N)}+ \|\ti f_\d\|_{\cM_\b} \right)  .
\ee
Let $x,y\in B_{1/4}$ be two Lebesgue points of $u$, and take  $\rho=  |x-y|/2 $. Then  $B_\rho(x)\subset B_{3\rho}(y)\subset B_1$. Therefore,   by \eqref{eq:est-sem-u-proof-cor}, we get
\begin{align*}
|B_\rho|^{N/2}|v_\d(x)-v_\d(y)|=\|v_\d(x)-v_\d(y)\|_{L^2(B_{\rho}(x))}&\leq \|v_\d(x)-v_\d\|_{L^2(B_{\rho}(x))}+  \|v_\d(y)-v_\d\|_{L^2(B_{3\rho}(y))}\\
&\leq C \rho^{N/2+\a }  \left(  \|v_\d\|_{L^2(\R^N)}+ \|\ti f_\d\|_{\cM_\b} \right)  .
\end{align*}
That is 
$$
 |v_\d(x)-v_\d(y)|\leq C |x-y|^\a   \left(   \|v_\d\|_{L^2(\R^N)}+ \| \ti f_\d\|_{\cM_\b} \right) \qquad\textrm{ for almost every $x,y\in B_{1/4}$. }
$$
We then conclude, from \eqref{eq:L-infty-bound-u-cor},    that 
\begin{align*}
\|v_\d\|_{C^\a(B_{1/4})} &\leq C  \left(  \|v_\d\|_{L^2(\R^N)}+ \|\ti f_\d \|_{\cM_\b} \right).
%
\end{align*}
It follows that
$$
\|u_\d\|_{C^\a(B_{1/4})}   \leq   C  \left(    \|  u_\d\|_{L^2(B_8)}+  \|\ti f_\d\|_{\cM_\b}+   \|u_\d\|_{\cL^1_s} \right) . 
$$
Scaling and translating back, we get
$$
\|u\|_{C^\a(B_{\d/4}(x_0))}   \leq   C  \left(    \|  u\|_{L^2(B_2)}+  \|f\|_{\cM_\b}+   \|u\|_{\cL^1_s} \right) ,
$$
where   $C$ depends only on $N,s,\b,\a, \L,\k,\d$ and $\|V\|_{\cM_\b}$. Since $B_1$ can be covered by a finite number of such balls $B_{\d/4}(x_0)$, with $x_0\in B_{3/4}$,   we get the desired estimate.
\end{proof}
 
As a consequence of  Corollary \ref{cor:int--Holder-reg-Morrey}, we  obtain regularity estimates for nonlocal operators with "uniformly continuous" coefficient.  For   $K\in \ti\scrK(\k)$,    we define  the functions 
\be \label{eq:def-ti-lambda-e-o}
\begin{aligned}
\ti\l_{e,K}(x,r,\th)&=\frac{1}{2} \left\{ \l_{K}(x,r,\th)+ \l_{K}(x,r,-\th) \right\}\\
 \ti\l_{o,K}(x,r,\th)&= \frac{1}{2} \left\{ \l_{K}(x,r,\th)- \l_{K}(x,r,-\th) \right\}.
\end{aligned}
\ee
If there is no ambiguity, we will simply write $ \ti\l_e$ and $\ti\l_o$ in the place of $\ti\l_{e,K}$ and $\ti\l_{o, K}$, respectively.
%
\begin{theorem}\label{th:int--Holder-reg-Morrey} 
Let  $s\in (0,1)$, $\b\in [0,2s)$ and     $\a\in (0,\min (1,2s-\b)) $.  Let $K\in  \ti\scrK(\k)$  and suppose that  $\ti\l_{e}$ and $\ti\l_{o}$   (defined in \eqref{eq:def-ti-lambda-e-o})  satisfy 
\begin{itemize}
\item   for every $x_1,x_2\in B_2$, $r\in (0,2)$, $\th \in S^{N-1}$,
 $$
  \left| \ti\l_{e}(x_1,r,\th)-\ti\l_{e}(x_2,0,\th) \right |\leq \t(|x_1-x_2|+r)  ;
 $$
 \item  for every $x\in B_2$, $r\in (0,2)$, $\th \in S^{N-1}$,
 $$
  \left| \ti\l_{o}(x,r,\th) \right |\leq \t( r),
  $$
\end{itemize}
 for some function $\t\in L^\infty(\R_+)$ and $\t(t)\to 0$ as $t\to 0$.
  Let $f,V\in \cM_\b$ and  $u\in H^s_{\loc}(B_2)\cap \cL^1_s$   satisfy
$$
\cL_{K} u+Vu= f \qquad\textrm{ in $B_2$}.
$$
Then    there exists $C>0$, only   depending   on $N,s,\b,\a,\k,\t$ and $\|V\|_{\cM_\b}$, such that 
  $$
  \|u\|_{C^{\a}(B_{1})}\leq C\left( \| u\|_{L^2(B_{2})}+\|u\|_{\cL^1_s}+\ \|f\|_{\cM_\b} \right).
  $$
\end{theorem}
\begin{proof}
Pick $x_0\in B_{3/2}$.   By assumption,  for every $x\in B_{2}, r\in (0,2)$ and $\th\in S^{N-1}$, 
$$
  \left| \ti\l_{o}(x,r,\th) \right |+   \left| \ti\l_{e}(x,r,\th)-\ti\l_{e}(x_0,0,\th) \right |\leq \t(r)+\t(|x-x_0|+r).
$$
Then  for every $\e>0$ there exists $\d=\d_{x_0,\e}\in (0,1/100)$ such that,  for every $x\in B_{4\d}(x_0)$ and $ r\in (0, 4\d)$, we have 
%
\begin{align*}
 \left| K(x,x+r\th )-\ti\l_{e}(x_0,0,\th)  r^{-N-2s} \right |\leq \e r^{-N-2s} .
\end{align*}
Therefore,   for every $x\in B_{4\d}(x_0)$ and   $0<|z|<4\d$,
\begin{align*}
 \left| K(x,x+z)- \ti \l_{e}(x_0,0,z/|z|)   |z|^{-N-2s} \right |\leq  \e   |z|^{-N-2s} 
\end{align*}
and thus,   for every $x,y\in B_{2\d}(x_0)$, with $x\not=y$,
\begin{align*}
 \left| K(x,y)-  \mu_{ a}(x,y)  \right |\leq  \e   \mu_{ 1}(x,y),
\end{align*}
where $  a(\th):=   \ti \l_{e}(x_0,0,\th)  $.  It is clear that $  a$ is even on $S^{N-1}$.
By changing variables, we find that for every  $x,y\in B_{2}$,  with $x\not =y$,   
\begin{align}\label{eq:near-srcK-proof-cor}
 \left| \d^{N+2s} K(\d x+ x_0,\d y+  x_0)- \mu_{ a}(x,y)  \right |\leq  \e \mu_1(x,y),
\end{align}
 In addition, 
$$
\k\leq  a(\th) \leq \k^{-1}\qquad\textrm{ for every $\th\in S^{N-1}$.}
$$
We    define
\begin{align*}
\begin{cases}
 \l (x,y)=   \e ,& \qquad\textrm{ for $x,y\in B_{2}$,}\\
 \l (x,y)=   \k^{-1}&  \qquad\textrm{ elsewhere. }
 \end{cases}
\end{align*}
We now let  $K_{\d}(x,y)=\d^{N+2s}K( \d x+ x_0,\d y+ x_0)$, which,  by \eqref{eq:near-srcK-proof-cor},  clearly satisfies  $K_{\d}\in \scrK( \l , a,\k)$.

For $x\in B_{2}$, we define $u_\d(x)=u(\d x+ x_0)$, $f_\d(x)=\d^{2s}f(\d x+ x_0)$ and $V_\d(x)= \d^{2s} V(\d x+ x_0)$. Since $\d\in (0,1/16)$,  by direct computations, we get
\be \label{eq:u-delta-satisf-final}
\cL_{K_{\d}} u_\d+V_\d u_\d= f_\d \qquad\textrm{ in $B_{8}$}.
\ee
Recall that $\| V_\d\|_{\cM_\b}\leq C \d^{2s-\b}\|V\|_{\cM_\b}$ and thus, decreasing $\d$ if necessary, we get $\| V_\d\|_{\cM_\b}\leq 1$.  
 Since  $\| \l \|_{L^\infty(B_2\times B_2)}= \e $, then  provided $\e>0$ small,    by  Corollary \ref{cor:int--Holder-reg-Morrey} and a change of variable,   we get
$$
\|u\|_{C^{\a}(B_{\d_{x_0,\e}}(x_0))}   \leq   C(x_0)  \left(    \|  u\|_{L^2(B_{2})}+  \|f\|_{\cM_\b}+   \|u\|_{\cL^1_s} \right) ,
$$
where $C(x_0)$ is a constant, only depending on $N,s,c_0,\d,\k,\t, x_0$ and $\|V\|_{\cM_\b}$.  Next, we cover $\ov B_{1}$ by  a finite number of balls $B_{\frac{1}{2}\d_{x_i,\e}}(x_i)$, for $i=1,\dots,n$, with $x_i\in \ov B_{1}$.  Put $ C:=\max_{1\leq i\leq n}C(x_i)$ and $\varrho=\frac{1}{2}\min_{1\leq i\leq n} \d_{x_i,\e}$. Then  on any  ball $B_{\varrho }(\ov x)$, with $\ov x\in \ov B_{1}$,  we have the estimate
$$
\|u\|_{C^{\a}(B_{\varrho}(\ov x))}   \leq   C  \left(    \|  u\|_{L^2(B_{2})}+  \|f\|_{\cM_\b}+   \|u\|_{\cL^1_s} \right) ,
$$
where $\varrho$ and $C$ depend only on $N,s,c_0,\k,\t$ and $\|V\|_{\cM_\b}$. Since $\ov B_1$ can be covered by a finite number  of balls  $B_{\varrho }(\ov x)$, with $\ov x\in \ov B_{1}$,     we get the result.
\end{proof}
\begin{remark}\label{eq:rem-scale-modulus-of-cont-}
We note that the conclusion of Theorem \ref{th:int--Holder-reg-Morrey} remains unchanged if we considered, say, a "better" modulus of continuity $\t$. More precisely, in Theorem \ref{th:int--Holder-reg-Morrey},  we could choose    $\t_\rho(r)=\t(\rho r)$ for some $\rho\in (0,1)$ and $\t$ as in the theorem.  In this case, the constant $C$ in the theorem will not depend on $\rho$.
\end{remark}
\subsection{ H\"older regularity estimates up to  the boundary }
H\"older regularity up to the boundary for the linear second order partial differential equations with coefficients in Morrey space was obtained in \cite{Chen}.
Coupling the interior regularity in Corollary \ref{cor:int--Holder-reg-Morrey} and the uniform $L^2$ growth estimates up to  the boundary given by Proposition \ref{prop:bdr-estim-Kato} together with some scaling arguments, we get the following result.
\begin{theorem}\label{th:bound-reg-alpha}
Let $\O$ be an open set such that $\de\O\cap B_{2}$ is a  $C^1$ hypersurface.  Suppose that $0\in \de\O$ and that $\de\O$ separates $B_2$ into two domains.  Let  $s\in (0,1)$, $\b\in [0,2s)$,   $\a\in (0,\min (s,2s-\b)) $ and $a$ satisfy \eqref{eq:def-a-anisotropi}.  Let $K\in \scrK(\l,a,\k)$.
 Let  $V, f \in \cM_\b$ and $u\in H^s (B_2)\cap \cL^1_s$ satisfy 
$$
\cL_{K} u+ Vu = f \qquad\textrm{ in $B_2\cap \O$} \qquad \textrm{ and } \qquad u =0 \qquad \text{ in  $B_2\cap\O^c$}.
$$
Then   there exists $C,\ov \e_0, r_0>0$ such that  if $\|\l  \| _{L^\infty(B_{2}\times B_{2}) }<\ov \e_0,$ we have 
  $$
  \|u\|_{C^{\a}(B_{r_0})}\leq C(\|u\|_{L^2(B_2)}+ \|u\|_{\cL^1_s}+ \|f\|_{\cM_\b}),
  $$
  with $C,\ov\e_0,r_0$ depending only on $N,s,\b,\a,\O, \|V\|_{\cM_\b},\k$ and  $\L$.
\end{theorem}
\begin{proof}
By similar scaling and cut-off  argument as in the proof of  Corollary \ref{cor:int--Holder-reg-Morrey} and using    Proposition \ref{prop:bdr-estim-Kato}, we get 
\be\label{eq:bdr-estim-Kato} 
 \sup_{z\in  B_{1}\cap\de\O}\|u\|_{L^2(B_{r}(z) )} ^2 \leq C r^{N+2\a} \left(    \|  u\|_{L^2(B_2)}+   \|u\|_{\cL^1_s}  +  \|f\|_{\cM_\b}\right) \qquad\textrm{ for every $r>0$}
\ee
with $C$ a constant depending only on $N,s,\b,\k,\L,\a,\O,   \|V\|_{\cM_\b}$.  
We assume in the following that $ \|  u\|_{L^2(B_2)}+   \|u\|_{\cL^1_s}  +  \|f\|_{\cM_\b}\leq 1$, up to dividing the equation by this quantity. Moreover by Corollary \ref{cor:int--Holder-reg-Morrey},   $u$ is continuous in $\O$, provided $ \|\l  \| _{L^\infty(B_{2}\times B_{2}) }$ is small.\\
Let $r_0>0$, only depending on $\O$, be such that every point $x_0\in \O\cap B_{r_0}$, with  $d(x_0)\leq r_0$, has a unique projection $z$ on $\de\O\cap B_1$.  
For such  $x_0 \in \O\cap B_{r_0} $,  we let $z\in \de\O\cap B_{1}$ be such that $|x_0-z|=d(x_0)$. Put $\rho=\frac{d(x_0)}{2}$, so that $B_\rho(x_0)\subset B_{3\rho}(z)\cap  \O. $  
Next, we define $v_\rho(x):=u(\rho x+ x_0)$, $V_\rho(x):= \rho^{2s} V (\rho x+ x_0)$  and $f_\rho(x):= \rho^{2s} f (\rho x+ x_0)$. It is plain that $v\in H^s_{loc}(B_2)\cap\cL^1_s$ and 
\be\label{eq:v-rho-solves-B1}
\cL_{K_{\rho,x_0}} v_\rho+ V_\rho v_\rho=f_\rho \qquad \textrm{ in $B_1$},
\ee
with
$$
K_{\rho,x_0}(x,y)=\rho^{N+2s}K(\rho x+x_0,\rho y+x_0).
$$
Recall from  \eqref{eq:eta-f-scale-translate} that $  \|  f_\rho \|_{\cM_\b}\leq \rho^{2s-\b} \|f\|_{\cM_\b} $ and $  \|  V_\rho \|_{\cM_\b}\leq \rho^{2s-\b} \|V\|_{\cM_\b} $. Hence decreasing $r_0$ if necessary, we may assume that $ \|  V_\rho \|_{\cM_\b}\leq 1 $.
 We note that $K_{\rho,x_0}\in \scrK(\l_{\rho,x_0},a,\k)$, where $\l_{\rho,x_0}(x,y)=\l(\rh x+ x_0,\rho y+y_0)$. In particular, decreasing $r_0$ if necessary,  
$$
\|\l_{\rho,x_0} \| _{L^\infty(B_2\times B_2) }\leq \|\l  \| _{L^\infty(B_{2}\times B_{2}) }.
$$
Therefore,     by Theorem \ref{th:int--Holder-reg-Morrey}, we can find $\ov \e_0>0$ small, independent on $\rho$, such that,  if  $\|\l  \| _{L^\infty(B_{2}\times B_{2}) }<\ov \e_0$,   we have  
\be \label{eq:estime-v-scale}
\|v_\rho\|_{L^\infty(B_{1/2})}\leq C\left(   \|v_\rho\|_{L^2(B_1)}+ \|v_\rho\|_{\cL^1_s}+  \|f_\rho\|_{\cM_\b}  \right)  .
\ee
By  \eqref{eq:bdr-estim-Kato}  and H\"older's inequality,  we have 
\begin{align}\label{eq:norm-vL2-cL1s-rho}
\|v_\rho\|_{L^2(B_1)}\leq \rho^{-N/2} \|u\|_{L^2(B_{3\rho}(z))}\leq C \rho^\a \qquad \textrm{ and  } \qquad  \|u\|_{L^1(B_{r}(z))}\leq C r^{N+\a}\quad
  \textrm{  for all $r>0$.}
\end{align}
Using   the second estimate in \eqref{eq:norm-vL2-cL1s-rho}, we get
\begin{align*}
\int_{|x|\geq 1}|x|^{-N-2s}|v_\rho(x)|\, dx&=\rho^{2s}\int_{|y-x_0|\geq \rho }|y-x_0|^{-N-2s}|u(y)|\, dy\\
&\leq \rho^{2s}\sum_{k=0}^\infty\int_{\rho 2^{k+1}\geq |y-x_0|\geq \rho2^k}|y-x_0|^{-N-2s}|u(y)|\, dy\\
&\leq \rho^{2s}\sum_{k=0}^\infty (2^k\rho)^{-N-2s}  \int_{\rho2^{k+1} \geq |y-x_0| } |u(y)|\, dy\\
&\leq  \rho^{2s}\sum_{k=0}^\infty (2^k\rho)^{-N-2s}\|u\|_{L^1(B_{\rho 2^{k+3} }(z))}  \\
&\leq  C \rho^{2s}\sum_{k=0}^\infty(2^k\rho)^{-N-2s} (\rho 2^{k})^{N+\a} \leq C \rho^{\a} \sum_{k=0}^\infty 2^{-k(2s-\a)}\leq C \rho^{\a}. 
\end{align*}
We then conclude that $\|v_\rho\|_{\cL^1_s}\leq C \rho^{\a}$. It follows from \eqref{eq:v-rho-solves-B1}, \eqref{eq:norm-vL2-cL1s-rho} and  \eqref{eq:estime-v-scale}, that 
 $$
  \|v_\rho\|_{L^\infty(B_{1/2})}\leq C(\|v_\rho\|_{L^2(B_1)}+ \|v_\rho\|_{\cL^1_s}+ \|f_\rho\|_{\cM_\b})\leq C (\rho^{\a}+  \rho^{2s-\b} ).
  $$
Scaling back, we get
 $$
  \|u\|_{L^\infty(B_{\rho/2}(x_0))}\leq C (\rho^{\a}+  \rho^{2s-\b}  ),
  $$
 which, in particular, yields
  $$
  |u(x_0)|\leq C  \rho^{\a}\leq {C} d(x_0)^\a \qquad\textrm{  for every $x_0\in B_{r_0}\cap\O$.}
  $$
  Now by a classical scaling argument as above and using the interior regularity estimates in Theorem \ref{th:int--Holder-reg-Morrey}, we get $u\in C^\a(B_{r_0}\cap\ov\O )$, with  $\|u\|_{C^{\a}(B_{r_0/2})}\leq   C$.\\
\end{proof}
 As a consequence, we have 
 \begin{theorem}\label{th:bdr-first-Holder-reg-Morrey} 
Let  $s\in (0,1)$, $\b\in [0,2s)$ and     $\a\in (0,\min (s,2s-\b)) $.  Let $K\in \ti\scrK(\k)$  and suppose that  $\ti\l_{e}$ and $\ti\l_{o}$   (defined in \eqref{eq:def-ti-lambda-e-o})  satisfy 
\begin{itemize}
\item   for every $x_1,x_2\in B_2 $, $r\in (0,2)$, $\th \in S^{N-1}$,
 $$
  \left| \ti\l_{e}(x_1,r,\th)-\ti\l_{e}(x_2,0,\th) \right |\leq \t(|x_1-x_2|+r)  
 $$
 \item  for every $x\in B_2$, $r\in (0,2)$, $\th \in S^{N-1}$,
 $$
  \left| \ti\l_{o}(x,r,\th) \right |\leq \t( r),
  $$
\end{itemize}
 for some function $\t\in L^\infty(\R_+)$ and $\t(t)\to 0$ as $t\to 0$.
  Let $f,V\in \cM_\b$ and  $u\in H^s_{\loc}(B_2)\cap \cL^1_s$   satisfy
$$
\cL_{K} u+ Vu = f \qquad\textrm{ in $B_2\cap \O$} \qquad \textrm{ and } \qquad u =0 \qquad \text{ in  $B_2\cap\O^c$}.
$$
Then    there exist $C,r_0>0$, only   depending  only on $N,s,\b,\a, \L,\k,\t,\O$ and $\|V\|_{\cM_\b}$, such that 
  $$
  \|u\|_{C^{\a}(B_{r_0})}\leq C\left( \| u\|_{L^2(B_{1})}+\|u\|_{\cL^1_s}+\ \|f\|_{\cM_\b} \right).
  $$
\end{theorem}
\begin{proof}
Adapting the scaling arguments as in the  Theorem \ref{th:int--Holder-reg-Morrey},  together with Theorem \ref{th:bound-reg-alpha}, we get the result.
\end{proof}
 \section{Higher   regularity estimates up to  the boundary}\label{s:Higher-reg}\label{s:high-int-reg}
In this section, we let $\O$ be an    open set such that $\de\O\cap B_{2}$ is a  $C^{1,\g}$ hypersurface,  with $0\in \de\O$ and $\g>0$.  We will assume that  $\de\O$ separates $B_2$ into two domains.    
We note that  there exists $r_0\in (0,1/2)$ small only depending on $\O$ such that, for all   $r\in (0,2r_0)$, $z\in \de\O\cap B_{3/2}$ and $\d>0$,
\be\label{eq:integ-dist-tubular-neigh}
 {C} r^{N+\d}\leq \int_{B_{r}(z)}d^{\d}(y)\, dy,
\ee
for some constant $C=C(N,s,\d,\O)>0$. On the other  hand since $d(y)\leq |y-z|$ for all $z\in \de\O$, for every $r>0$, we have 
\be\label{eq:integ-dist-tubular-neigh-up}
 \int_{B_{r}(z)}d^{\d}(y)\, dy\leq  Cr^{N+\d},
\ee
with $ C= C(N,s,\d)$.\\
We consider the cut-off of the distance function $d$ denoted by  $ d_2^s:= \vp_{2} d^s$.
For $u\in L^2(\R^N)$, $z\in \de\O\cap B_1$ and $r>0$, we let $P_{r,z}(u)$ be the $L^2_{loc}(B_r(z))$-projection of $u$ on $\la  d^s_2\ra=\R  d^s_2$,  the one-dimensional space  spanned by $d^s_2$.  Therefore 
\be\label{eq:def-project-dist}
\int_{B_r(z)}(u(y)-P_{r,z}(u)(y))  d^s_2(y)\, dy=0 \qquad  \textrm{ and } \qquad P_{r,z}(u)(x)=  d^s_2(x) \frac{\int_{B_r(z)} u(y)  d^s_2(y)\,dy}{\int_{B_r(z)}  d^{2s}_2(y)\, dy} .
\ee
For $z\in B_1\cap\de\O$ and $r>0$, we define
\be \label{eq:def-Qu-z-r}
Q_{u,z}(r):=\frac{\int_{B_{r}(z)} u(y) d^s_2(y)\,dy}{\int_{B_{r}(z)} d^{2s}_2(y)\, dy} .
\ee
Before going on, we explain the arguments in the next two main results  of this section.
Observe that by H\"older's inequality, for every $r\in (0,r_0]$ and $z\in B_1\cap \de\O$,
$$
 |Q_{u,z}(r)|\leq \|{d}^s_2\|_{L^2(B_r(z))}^{-1} \|u\|_{L^2(B_r(z))}   =  \| d^s\|_{L^2(B_r(z))}^{-1} \|u\|_{L^2(B_r(z))}  .
$$
Hence  by Proposition \ref{prop:bdr-estim-Kato} and \eqref{eq:integ-dist-tubular-neigh}, for every $\d_0\in (0,\min(s,2s-\b))$, there exist  constants $C,\ov \e_0>0$ such that  for every  $\l\in L^\infty(B_2\times B_2)$ satisfying $\|\l\|_{L^\infty(B_2\times B_2)}<\ov \e_0$,    for every $f\in \cM_\b$ and $u\in \cS_{K,0,f;\O}$ (recall the notation \eqref{eq:set-of-sol-Omega}), $r\in (0,r_0]$ and $z\in B_{1}\cap\de\O$, we have 
\be \label{eq:estim-Quz-r-d0}
 |Q_{u,z}(r)|\leq C r^{\d_0-s} \left(\|u\|_{L^2(\R^N)}+ \|f\|_{\cM_\b} \right)   .
\ee
Our objective is to get $u/d^s\in C^{s-\b}(B_{r_0}\cap\ov\O)$, whenever $\b\in (0,s)$ and $\O$ regular enough. This requires, at least,  we already know that $|u|\leq d^s$,  or equivalently $|Q_{u,z}(r)|\leq C$. For  this purpose,   we will use a  bootstrap argument in two steps to obtain \eqref{eq:estim-Quz-r-d0} with $\d_0=s$, as long as  $\b<s$ and under more regularity assumption on $K$ and $\de\O$. This will be  the content of the next two results.\\

In order to get the sharp boundary regularity, it will be crucial to quantify  the action of the operators $\cL_K$ on $d^s$ for  $K\in \scrK(\l,a,\k)$.  To this end, we first note that by Lemma \ref{lem:ds-in-Hs-loc}, up to decreasing $r_0$ if necessary, we may assume that $d^s_2\in H^s(B_{2 r_0})\cap\cL^1_s$. Next,  we introduce  $\scrK(\l,a,\k,\O)$,  the  class of kernels  $K\in \scrK(\l,a,\k) $ such that:  there exist $\b'=\b'(\O,K)\in [0,s)$  and  a   function $g_{\O,K}\in \cM_{\b'}$  such that 
\be\label{eq:def-Ds-ds-g}
\cL_K d^s_2=\cL_K ( \vp _2d^s)=g_{\O,K} \qquad\textrm{  in the weak sense, in $B_{2 r_0}\cap \O$.}
\ee
We note that  the class of kernels $\scrK(\l,a,\k,\O)$ is not empty.  This is the case for $K=\mu_a$, with $a$ satisfying \eqref{eq:def-a-anisotropi}, see  Section \ref{s:proof-main-results} below.
We  have  the following result.
 \begin{lemma} \label{lem:boundary-reg-Morrey}
 Let   $\O\subset\R^N$ be a $C^{1,\g}$ domain as above for some $\g>0$.  
 Let    $\b\in (0,2s)$, $\varrho\in [0,s)$ and  $c_0,\L,\k>0$.   Then    there  exist $C>0$ and $\e_1>0$ with the properties that    if
 \begin{itemize}
 \item  $a$ satisfies  \eqref{eq:def-a-anisotropi}, 
 \item  $\l: \R^N\times \R^N\to [0, k^{-1}]$ satisfies  $\|\l\|_{L^\infty(B_2\times B_2)}<\e_1$,
 \item ${K}\in \scrK(\l,a,\k,\O) $ with     $\|g_{\O,K}\|_{  \cM_{\b'}}\leq 1 $, for some $\b' \in [0,s)$,
 \item    $f\in \cM_\b$ and $u\in \cS_{K,0,f;\O}$ satisfies    
\be \label{eq:sup-z-Qzur}
 \sup_{z\in B_1\cap\de\O} |Q_{u,z}(r)|\leq c_0 r^{-\varrho} (\|u\|_{L^2(\R^N)}+\|f\|_{\cM_\b})  \qquad\textrm{ for every $r\in (0,r_0]$,}
\ee
\end{itemize}    
then, we have  
\be\label{eq:sharp-estim-bdr}
\sup_{z\in B_1\cap \de\O}  \| u- P_{r,z}(u)\|_{L^2(B_r(z))}^2\leq C r^{N +2(2s-\max(\b,\b')-\varrho) }(\|u\|_{L^2(\R^N)}+\|f\|_{\cM_\b})^2   \qquad\textrm{ for every $r>0$.}
\ee
\end{lemma}
 \begin{proof}
As in the proof of Proposition \ref{prop:bound-Kato-abstract}, if \eqref{eq:sharp-estim-bdr} does not hold, then we can find a sequence $r_n\to 0$,  points $z_n\in B_1\cap\de\O$ and sequences of functions, $a_n$ satisfying \eqref{eq:def-a-anisotropi}, $\l_n$ with $\|\l_n\|_{L^\infty(B_2\times B_2)}<\frac{1}{n}$, $K_{a_n}\in \scrK(\l_n,a_n,\k,\O)$ with $\b'_n\in [0,s)$ and     $\|g_{\O,K_{a_n}}\|_{  \cM_{\b'_n}}\leq 1 $,    $f_n\in  \cM_\b$  and $u_n\in \cS_{K_{a_n}, 0, f_n;\O}$, with $\|u_n\|_{L^2(\R^N)}+\|f_n\|_{\cM_\b}\leq 1 $ satisfying 
\be \label{eq:sup-z-Qzu_n-r}
 |Q_{u_n,z_n}(r_n)|\leq c_0 r^{-\varrho}_n,    
\ee
while, letting $\a_n:=2s-\max(\b,\b'_n)\in (s,2s)$, we have that
\be\label{eq:resclae-situation-pwb-Morrey}
      r^{-N-2(\a_n-\varrho)}_n  \|u_n -P_{r_n,z_n}(u_n) \|_{L^2(B_{r_n} (z_n))}^2\geq  \frac{1}{16} \Theta_n(r_n) \geq  \frac{n}{32} ,
\ee
for all $n\geq 2$. Here also $\Theta_n$ is a nonincreasing function on $(0,\infty)$ satisfying
\be \label{eq:sup-r-ok-Morrey-bdr}
\Theta_n( r)\geq  r^{-N -2(\a_n-\varrho)} \sup_{z\in B_1\cap\de\O}  \| u_n- P_{r,z}(u_n)\|_{L^2(B_{r}(z))}^2   \qquad\textrm{ for every $r>0$ and $n\geq 2$.}
\ee
 We   define  
$$
 {v}_n(x)=  \Theta_n(r_n)^{-1/2} r^{-(\a_n-\varrho)}_n \left\{ u_n (r_n x+z_n) -P_{r_n,z_n}(u_n)(r_n x+z_n)  \right\}    .
$$
Since for $r_n\leq 1/2$, we have $d_2^s=d^s$ on $B_{r_n}(z_n)$,  making a change of variable in \eqref{eq:resclae-situation-pwb-Morrey} and in \eqref{eq:def-project-dist}, we get
\be\label{eq:lower-bnd-vn-Morrey}
  \|v_n  \|_{L^2(B_1)}^2\geq \frac{1}{16}\qquad\textrm{ and }\qquad \int_{B_1}v_n(x)\textrm{dist}(x, \O_n^c)^s\, dx=0,
\ee
where 
$$
\O_n:=\frac{1}{r_n}(\O-z_n) .
$$
We further define $K_n(x,y):=r_n^{N+2s}K_{a_n}(r_n x+ z_n, r_n y+ z_n)$ for every $x,y\in \R^N$, 
\be\label{eq:def-hat-hat-g}
\widehat f_{n} (x) :=   r_n^{2s} f_n (r_n x+z_n)\qquad \textrm{ and } \qquad  g_{n} (x) :=    r_n^{2s} g_{\O,K_n} (r_n x+z_n).
\ee
Since $u_n\in \cS_{K_{a_n},0,f;\O}$, by \eqref{eq:def-Ds-ds-g},  it is plain that
\be \label{eq:of-vn-Morrey}
\begin{cases}
\cL_{K_n} v_n (x) =    r^{-(\a_n-\varrho)}_n  \Theta_n(r_n)^{-1/2} \left(  \widehat f_n(x)-        Q_{u_n,z_n}(r_n)   g_n(x) \right) & \qquad\textrm{ in $B_{2r_0/{r_n}}(-z_n)\cap \O_n$}\\
v_n=0  & \qquad\textrm{ in $B_{2 r_0/{r_n}}(-z_n)\cap \O_n^c$}.
\end{cases}
\ee
%
%
\bigskip
\noindent
\textbf{Claim:} There exists $C= C(  s,  N,\b,\varrho,r_0,c_0)>0$     such that  
\be\label{eq:growth-v-n-Morrey}
 \|v_n  \|_{L^2(B_R)}^2 \leq  C  R^{N + 2 \a_n  }   \qquad\textrm{ for every $R\geq 1$.}
\ee
 Let us  put $\a'=\a_n-\varrho>0$, and we note that $0<s-\varrho<\a'<2s-\b$, for every $n\geq 2$. By a change of variable, we have
\begin{align*}
 \|v_n\|_{L^2(B_{R})}^2&=   \Theta_n(r_n)^{-1}    r_n^{-N-2\a'}  \|u- P_{r_n,z_n}(u_n)\|_{L^2(B_{r_nR}(z_n))}^2\\
& \leq 2  \Theta_n(r_n)^{-1}     r_n^{-N-2\a'}   \|u- P_{r_n R,z_n}(u_n)\|_{L^2(B_{r_nR}(z_n))}^2\\
&+ 2  \Theta_n(r_n)^{-1}     r_n^{-N-2\a'}   \|P_{r_n R,z_n}(u_n)- P_{r_n,z_n }(u_n)\|_{L^2(B_{r_nR}(z_n))}^2.
\end{align*}
Hence by \eqref{eq:sup-r-ok-Morrey-bdr} and the monotonicity of $\Theta_n$, we get
\begin{align}\label{eq:first-estim-v-n}
 \|v_n\|_{L^2(B_{R})}^2&  \leq 2  R^{2\a'}+ 2  \Theta_n(r_n)^{-1}     r_n^{-N-2\a'}   \|P_{r_n R,z_n}(u_n)- P_{r_n,z_n }(u_n)\|_{L^2(B_{r_nR}(z_n))}^2.
\end{align}
Now by \eqref{eq:integ-dist-tubular-neigh},    for all $r\in (0, r_0]$ and $z\in B_1\cap\de\O$, we have 
\begin{align*}
C |Q_{u_n,z}(2r)-Q_{u_n,z}(r)||B_{r}(z)|^{\frac{1}{2}} r^{s}&\leq  \|P_{2r,z }(u_n)- P_{r,z }(u_n)\|_{L^2(B_{ r}(z))}\\
& \leq  \|P_{2r,z }(u)-  u\|_{L^2(B_{2r}(z) )}+ \|P_{r ,z}(u_n)- u_n\|_{L^2(B_{r}(z) )}\\
&\leq   (2 r)^{\a'}    \Theta_n(2 r)^{1/2}   (2 r) ^{N/2}+    r^{\a'}      \Theta_n( r)^{1/2}    r ^{N/2}.
\end{align*}
Now  using  the monotonicity of $\Theta_n$, we then deduce that there exists a constant $C>0$ such that  for every $n\geq 2$,  $r\in (0, r_0]$ and $z\in B_1\cap\de\O$,
$$
|Q_{u_n,z}(2r)-Q_{u_n,z}(r)| \leq C  r ^{\a'-s}  \Theta_n( r)^{1/2} .
$$
Hence, for $m\geq 0$, with $2^m\leq \frac{r_0}{ r}$, using \eqref{eq:integ-dist-tubular-neigh-up} and the monotonicity of $\Theta_n$,  we get 
\begin{align*}
\|P_{2^m r }(u_n)- P_{r }(u_n)\|_{L^2(B_{2^m r}(z))}&\leq \sum_{i=0}^m  \|P_{2^i r }(u_n)-  P_{2^{i-1} r }(u_n)\|_{L^2(B_{2^i r}(z))}\\
&\leq C \sum_{i=0}^m   |Q_{u_n,z}(2^{i}r)-Q_{u_n,z}(2^{i-1}r)||B_{2^i r}|^{\frac{1}{2}} (2^i r)^s\\
%
%
%
&\leq  C r^{N/2+\a'}   \Theta_n(  r )^{1/2}     \sum_{i=0}^m 2^{i(N/2+ \a')}.
%
%
\end{align*}
As a consequence, we find that
\begin{align*}
\|P_{2^m r,z }(u_n)- P_{r,z }(u_n)\|_{L^2(B_{2^m r}(z))}&\leq    C  r^{N/2+\a'}   \Theta_n(  r )^{1/2}  2^{m (N/2+ \a')}.
\end{align*}
Now using this in  \eqref{eq:first-estim-v-n}, we then get, for $2^m\leq \frac{r_0}{ r_n}$,
\begin{align*}
 \|v_n\|_{L^2(B_{2^{m}})}^2& \leq 2 2^{m(N+2\a')}+  2  \Theta_n(r_n)^{-1}     r_n^{-N-2\a'}  \|P_{2^m r ,z_n}(u_n)- P_{r,z_n }(u_n)\|_{L^2(B_{2^m r}(z))}\\
%
%
&\leq 2 2^{m(N+2\a')}+ C 2^{mN} \Theta_n(r_n)^{-1}     r_n ^{-2\a'}  (r_n2^m)^{2\a}    \Theta_n(2^m r_n)    \\
&\leq  2 2^{m(N+2\a')}+ C      2^{m(N+2\a')}   \Theta_n(r_n)^{-1}   \Theta_n(2^m r_n)\\
&\leq  C 2^{m(N+2\a')},
\end{align*}
with $C$ is a positive constant   depending neither on  $n$ nor on $m$.  We then conclude that
\be\label{eq:estim-L2-v-R-rrn-small}
 \|v_n\|_{L^2(B_{R})}^2 \leq C R^{N+2\a'} \qquad \textrm{  for every $R\geq 1$, with  $R r_n\leq r_0$.}
\ee
We now consider the case $R\geq 1$ and  $R r_n\geq r_0$. Using the fact  that   $ \Theta_n(r_n)^{-1}   \leq 1$  and  $\a'=\a_n-\varrho>0$ together with  \eqref{eq:def-project-dist} and \eqref{eq:integ-dist-tubular-neigh-up},    we obtain 
\begin{align*}
 \|v_n\|_{L^2(B_{R})}^2&=   \Theta_n(r_n)^{-1}     r_n^{-N-2\a'}  \|u_n- P_{r_n,z_n}(u_n)\|_{L^2(B_{r_nR}(z_n))}^2  \leq    r_n^{-N-2\a'} \|u_n \|_{L^2(\R^N)}^2 \leq     (R r_0^{-1})^{N+2\a'}  .
\end{align*}
This with \eqref{eq:estim-L2-v-R-rrn-small} give \eqref{eq:growth-v-n-Morrey}, since $\a'=\a_n-\varrho $. This finishes the proof of the  claim.\\

Now by \eqref{eq:growth-v-n-Morrey} and H\"older's inequality, we get
\begin{align}\label{eq:growth-v-n-Morrey-L1}
 \|v_n\|_{L^1(B_{R})} &\leq C R^{N+\a_n } \qquad\textrm{ for all $R\geq 1$ and $n\geq 2$.}
\end{align}
Since $g_{\O,K_{a_n}}\in \cM_{\b'_n}$     (recall \eqref{eq:def-hat-hat-g} and  \eqref{eq:eta-f-scale-translate}), we have $\eta_{g_n}(1)\leq C r_n^{2s-\b'_n} \|g_{\O,K_{a_n}}\|_{\cM_{\b'_n}}$.
Therefore  by  \eqref{eq:sup-z-Qzu_n-r},  we deduce that 
\be\label{eq:bound-eta-Qn-gn}
   r^{-(\a_n -\varrho)}_n |Q_{u_n,z_n}(r_n)|   \eta_{  g_n}(1)\leq c_0 C r^{-(\a_n -\varrho)+2s-\b'_n}_n   \|g_{\O,K_{a_n}}\|_{\cM_{\b'_n}}\leq C \|g_{\O,K_{a_n}}\|_{\cM_{\b'_n}}\leq C,
\ee
with $C>0$ independent on $n$.
From now on, we let $n$ large, so that $B_{r_0/(2r_n)}\subset B_{2r_0/r_n}(-z_n)$. Since $v_n$ satisfy \eqref{eq:of-vn-Morrey}, then letting $v_{n,M}=\vp_{M} v_n$,  we can apply  Lemma \ref{lem:cat-off-sol}, for $1<M<\frac{r_0}{2 r_n}$, to   get 
$$
\begin{cases}
\cL_{K_n} v_{n,M}  = \Theta_n(r_n)^{-1/2} r^{-(\a_n-\varrho)}_n \left(     \widehat f_n-       Q_{u_n,z_n}(r_n)  \widehat g_n(x)\right)+ F_n  & \quad\textrm{ in $B_{M/2}\cap \O_n$}\\
v_{n,M}  =0  & \quad\textrm{ in $B_{M/2} \cap \O_n^c$},
\end{cases}
$$
where    $\|F_n\|_{L^\infty(\R^N)}\leq C_0 \|v_n\|_{\cL^1_s}  $.  Using \eqref{eq:growth-v-n-Morrey} and H\"older's inequality, we get  $\|F_n\|_{L^\infty(\R^N)}\leq C$.    It then follows that 
\be\label{eq:eta-newterms-delta}
  \eta_{F_{v_n}}(1)\leq C \quad\textrm{ for every $n\geq 2$.}
\ee
 In addition   (recalling  \eqref{eq:def-hat-hat-g})  by    \eqref{eq:eta-f-scale-translate},  
\be \label{eq:up-eta-hatfn}
\eta_{ \widehat f_n }(1)\leq C   r^{\a_n}_n.
\ee
Now by  \eqref{eq:growth-v-n-Morrey}, 
   Lemma \ref{lem:from-caciopp-ok},   \eqref{eq:bound-eta-Qn-gn},   \eqref{eq:eta-newterms-delta} and \eqref{eq:up-eta-hatfn}, we obtain
  \begin{align*}
\left\{\k- \e\ov C \left(1+ \Theta_n(r_n)^{-1/2}r_n^\varrho+\Theta_n(r_n)^{-1/2}    +\eta_{F_{v_n}}(1) \right) \right\}&[v_{n,M}]_{H^s(B_{M/4})}^2\leq     C(M) . 
%
\end{align*}
Since $\Theta_n(r_n)^{-1/2}\to0 $ as $n\to \infty$, by  \eqref{eq:growth-v-n-Morrey}, we then    deduce that $v_n$ is bounded in $H^s_{loc}(\R^N)$. Hence by Sobolev embedding,   $v_n\to  v$ in $L^2_{loc}(\R^N)$,  for some  $v\in H^s_{loc}(\R^N) $.  In addition,  by \eqref{eq:growth-v-n-Morrey-L1} and since $\a_n=2s-\max(\b,\b'_n) <2s$, we deduce that $v_n\to v$ in $\cL^1_s$.   We also have that $1_{\O_n\cap B_{1/(2r_n)}}\to 1_{H}$ in $L^1_{loc}$ as $n\to \infty$, where $H$ is a half-space, with $0\in \de H$.   
Moreover, passing to the limit in \eqref{eq:lower-bnd-vn-Morrey}, we get
\be\label{eq:groth-v-and-low-bnd-Kato-pw-Morrey}
  \|v\|_{L^2(B_1)}^2 \geq \frac{1}{16} \qquad \textrm{ and }  \qquad \int_{B_1}v(x)\textrm{dist}(x,\R^n\setminus  H)^s\, dx=0.
\ee
Now, given  $\psi\in C^\infty_c( H\cap B_M)$, since $\O$ is of class $C^1$, for $n$ large enough, we obtain $\psi \in  C^\infty_c( \O_n)$.  Since $v_n$ satisfy \eqref{eq:of-vn-Morrey}, then by Lemma \ref{lem:convergence-very-weak}, \eqref{eq:up-eta-hatfn} and  \eqref{eq:bound-eta-Qn-gn},  we obtain 
\begin{align*} 
\bigl| \int_{\R^{2N}}(v_n(x)  &   -v_n(y))(\psi(x)-\psi(y))K_n(x,y)\, dxdy \bigr|  \\
& \leq r_n^\varrho   \Theta_n(r_n)^{-1/2} C(M)\left(  1 +  \|g_{\O,K_{a_n}}\|_{\cM_{\b'_n}}\right)\left( \|\psi\|_{H^s(\R^N)}^2 +   \|\vp_M\|_{H^s(\R^N)}^2     \right)\\
&\leq  \Theta_n(r_n)^{-1/2} C(M).
\end{align*}
Thanks to Lemma \ref{lem:Ds-a-n--to-La}, letting $n\to\infty$, we thus get
$$
L_b v=0 \qquad\textrm{ in $H$}  \qquad \textrm{ and }\qquad v=0 \qquad \textrm{ on $\R^N\setminus  H$.}
$$
Here $b$ denotes the weak limit of $a_n$. Letting $\a:=\lim_{n\to \infty} \a_n\in [0,2s)$,  by \eqref{eq:growth-v-n-Morrey}, we have that   $ \|v\|_{L^2(B_{R})}^2 \leq C R^{N+2\a}$ for every $R\geq 1$.  It follows from Lemma \ref{lem:Liouville}  that   $v $ does not change sign on $\R^N$,    which is in contradiction with \eqref{eq:groth-v-and-low-bnd-Kato-pw-Morrey}.
\end{proof} 
The next, result  finalizes the two-step bootstrap argument mentioned earlier.  
\begin{lemma}\label{lem:unif-cont-f-Morrey}
Let $N\geq 1$, $s\in (0,1)$,    $\b,\d\in (0,s)$ and $\O$ a $C^{1,\g}$ domain, with $0\in \de\O$ as above.   Let $K\in \scrK(\l,a,\k,\O)$ with  $\|g_{\O,K}\|_{\cM_{\b'}}\leq 1$, for some $\b'\in [0,s-\d)$  and $\|\l\|_{L^\infty(B_2\times B_2)}<\min(\ov \e_0,\e_1)$, where $\e_1$ and $\ov \e_0$ are given by Lemma \ref{lem:boundary-reg-Morrey} and Proposition \ref{prop:bdr-estim-Kato}, respectively. 
Let  $f\in \cM_\b$,  and $u\in H^s_{loc}(B_2)\cap L^2(\R^N)$ satisfy  
\be \label{eq:u-satf-bdr-reg-fin}
\cL_K  u= f \qquad\textrm{ in $  B_{2}\cap\O$}\qquad\textrm{ and }\qquad u=0 \qquad\textrm{ in  $  B_{2}\cap\O^c$, }
\ee
Then    there exists $C >0$, only depending on $N,s,\b, \L, \k,\e_1,\ov \e_0,\d $ and $\O$,  and a  function    $\psi\in L^\infty(B_{1}\cap\de\O)$, with $\|\psi\|_{L^\infty(B_{1}\cap\de\O )}\leq C$, such that 
$$
 \sup_{z\in {B_{1}\cap\de\O}}  \|u-\psi(z) d^s\|_{L^2(B_r(z))}^2\leq  C r^{N+2(2s-\max(\b,\b'))}(\|u\|_{L^2(\R^N)}+  \|f\|_{\cM_\b})^2   \qquad\textrm{ for all   $r\in (0,r_0/4)$.}
$$
\end{lemma}
\begin{proof}
For simplicity, we assume that $\|u\|_{L^2(\R^N)}+  \|f\|_{\cM_\b}\leq 1$, up to dividing  \eqref{eq:u-satf-bdr-reg-fin} with this quantity. 

Letting  $\a:=2s-\max(\b,\b')\in (s,2s)$, by Proposition \ref{prop:bdr-estim-Kato}, for  every  $\varrho\in (0,\a-s)$, there exists $c_0>0$, only depending on  $N,s,\b,\varrho,\O ,\k $ and $\L$, such that  
\be \label{eq:estim-Quz-r-d0-first}
   |Q_{u,z}(r)|\leq c_0  r^{-\varrho}      \qquad\textrm{ for all $z\in B_{1}\cap\de\O$ and $r\in (0,2r_0)$.}
\ee
 %
We can   apply    Lemma \ref{lem:boundary-reg-Morrey}   to get 
 \begin{align}\label{eq:estim-High-reg-cut}
\sup_{  r>0}  \sup_{z\in {B_{1}\cap \de\O}}  r^{-N-2(\a-\varrho)} \|u- P_{r,z}(u) d^s\|_{L^2(B_r(z))}^2&\leq  C ,
 \end{align}
 with the letter $C$ denoting, here an in the following,   a positive constant which may vary from line to line but will depend only on  $N,s,\b,\varrho,\O, \k $ and $\L$. \\
 
 %
 %
 %
\noindent
\textbf{Claim:}  There exits    $\psi_0\in L^\infty(\de\O\cap B_{1})$, satisfying $\|\psi_0\|_{L^\infty(\de\O\cap B_{1} )}\leq C$, such that
 \begin{align}\label{lem:boundary-reg-Morrey-final}
 \|u-\psi_0(z) d^s\|_{L^2(B_r(z))}&\leq  C r^{N/2+\a-\varrho}   \qquad\textrm{ for all $z\in B_{1}\cap\de\O$ and $r\in (0,r_0/2)$.}
 \end{align} 
Indeed, for  $r\in (0, 2r_0)$ and $z\in B_{1}\cap\de\O$, we define  
$$
Q_{z}(r):=Q_{u,z}(r)=\frac{\int_{B_{r}(z)} u(y) d^s(y)\,dy}{\int_{B_{r}(z)}d^{2s}(y)\, dy} ,
$$
and recalling    \eqref{eq:def-project-dist}, we have  that $P_{r,z}(u)(x)= Q_{z}(r) d^s(x) $ because $d_2^s=d^s$ on $B_{r}(z)$.\\
Let   $0<\rho_2\leq  \rho_1/4\leq r/4 $. Pick $k\in \N$ and   $\s\in [1/4,1/2]$ such that $\rho_2=\s^k\rho_1$.  Then provided $r\in (0,2r_0)$,  by \eqref{eq:integ-dist-tubular-neigh} and \eqref{eq:estim-High-reg-cut},  we get 
\begin{align*}
|Q_z(\rho_1)&-Q_z(\rho_2)|\leq \sum_{i=0}^{k-1}|Q_z(\s^{i+1}\rho_1)-Q_z(\s^i\rho_1)|\\
&\leq C   \sum_{i=0}^{k-1} (\s^{i+1}\rho_1)^{- \frac{N}{2}- s} \left(  \|u-P_{\s^i\rho_1 ,z}(u) \|_{L^2\left(B_{\s^i\rho_1}(z)\right)}+ \|u- P_{\s^{i+1}\rho_1,z }(u)\|_{L^2\left(B_{\s^{i+1}\rho_1}(z) \right)}  \right) \\
&\leq C  \sum_{i=0}^{k-1} (\s^{i+1}\rho_1)^{- \frac{N}{2}- s}  \left( (\s^{i}\rho_1)^{\frac{N}{2}+ \a-\varrho} +(\s^{i+1}\rho_1)^{\frac{N}{2}+ \a-\varrho}     \right)\\
&\leq C \rho_1^{\a-s-\varrho}  \s ^{- \frac{N}{2}- s} \sum_{i=0}^{k-1} \s^{i(\a-s-\varrho)}\leq C  \rho_1^{\a-s-\varrho}  ,
\end{align*}
where we used the fact that $\b'\in (0,s-\d)$, so that $C$ does not depend on $\d$ but only on the quantities mentioned above.
Therefore, there exists $C>0$ such that   for  $0<\rho_2\leq  \rho_1/4\leq r/4 $, with $r\in (0,2r_0)$,  and $z\in B_1\cap \de\O$, we have 
\begin{align}\label{eq:Qz-r-Cauchy}
|Q_z( \rho_1)-Q_z(\rho_2)| \leq    C \rho_1^{\a-s-\varrho}\leq C r^{\a-s-\varrho}.
\end{align}
By \eqref{eq:integ-dist-tubular-neigh} and  \eqref{eq:estim-High-reg-cut}, for $r\in (0, 2r_0)$,   we get 
$$
   \| P_{r_0/4,z}(u)\|_{L^2(B_{r_0/4}(z))}\leq \| u-P_{r_0/4,z}(u)\|_{L^2(B_{r_0/4}(z))}+ \| u\|_{L^2(B_{r_0/4}(z))}\leq C   +1.
$$
Hence there exists $C>0$, such that  for all $  z\in B_{1}\cap  \de \O $, 
\be\label{eq:Q-z-bounded-Morrey} 
  |Q_z(r_0/4)|\leq   C,
\ee
From  \eqref{eq:Qz-r-Cauchy}, we   deduce that, for every fixed $z\in \de B_{1}\cap  \de \O$ and any sequence  $(r_n)_{n\in \N}\subset ( 0,r_0/4]$ tending to zero,    $(Q_z(r_n))_{n\in \N}$ is a Cauchy sequence, which is bounded by  \eqref{eq:Q-z-bounded-Morrey}.     We can thus define  $\psi_0(z):=\lim_{r\to 0} Q_z(r)$. Now by \eqref{eq:integ-dist-tubular-neigh} and  \eqref{eq:Qz-r-Cauchy} (letting $\rho_2\to 0$), for $r\in (0,r_0/2)$,  we get
\begin{align*}
C r^{-N/2-s} \| P_{r,z}(u)-   \psi_0(z) d^s   \|_{L^2(B_r(z))}& \leq |Q_z(r)-\psi_0(z)|\leq C   r^{\a-s-\varrho } .
\end{align*}
This in 
particular yields $|\psi_0(z)|r_0^s\leq C r^{\a-s-\varrho } _0+  Q_z(r_0/4)\leq C   ,$  by \eqref{eq:Q-z-bounded-Morrey}. Consequently $\|\psi_0\|_{L^\infty(\de B_{1}\cap  \de \O)}\leq C$.   Finally using \eqref{eq:estim-High-reg-cut} and the above inequality, we can estimate
 \begin{align*}
 \|u-\psi_0(z) d^s\|_{L^2(B_r(z))}&\leq \|u - P_{r,z}(u) \|_{L^2(B_r(z))}+ \| P_{r,z}(u)-   \psi_0(z) d^s   \|_{L^2(B_r(z))} \\
& \leq C r^{N/2+\a-\varrho}  .
 \end{align*} 
This proves  \eqref{lem:boundary-reg-Morrey-final}, as claimed.  \\
  
 Now  \eqref{lem:boundary-reg-Morrey-final}, implies in particular that   $ \|u\|_{L^2(B_r(z))}\leq C r^{N/2+s}$. Hence by H\"older's inequality and since $\a-\varrho>s$, there exists $c_1=c_1(N,s,\b,\O,\L,\varrho, \k,\d)>0$ such that 
$$
  |Q_{u,z}(r)|\leq c_1  \qquad\textrm{ for every $r\in (0,r_0/2)$ and $z\in B_{1}\cap\de\O$}.
$$
We can therefore apply Lemma \ref{lem:boundary-reg-Morrey} with $\varrho=0$ and thus use   the same argument above starting from \eqref{eq:estim-High-reg-cut}.   We then conclude that there exists    $\psi\in L^\infty(B_{1}\cap\de\O)$, with $\|\psi\|_{L^\infty(B_{1}\cap\de\O)}\leq C$ and such that 
 \begin{align*}
 \|u-\psi(z) d^s\|_{L^2(B_r(z))}&\leq  C r^{N/2+\a}   \qquad\textrm{ for all $z\in B_{1}\cap\de\O$ and $r\in (0,r_0/4)$.}
 \end{align*} 
 \end{proof}
Combining Lemma \ref{lem:unif-cont-f-Morrey} and  the interior estimates in  Theorem \ref{th:int--Holder-reg-Morrey}, we get the following result. %
\begin{corollary}\label{cor:unif-cont-f-Morrey}
Let $N\geq 1$, $s\in (0,1)$,    $\b,\d\in (0,s)$ and $\O$ a $C^{1,\g}$ domain, with $0\in \de\O$  as above. Let $a$ satisfy \ref{eq:def-a-anisotropi} and  $K\in \scrK(\l,a,\k)$.   
Suppose that   $\|g_{\O,K} \|_{ \cM_{\b'}}\leq  c_0$, as defined in \eqref{eq:def-Ds-ds-g}, with $\b'\in [0,s-\d)$.  
Let   $f\in \cM_\b$,  and $u\in H^s(B_2)\cap \cL^1_s$ satisfy
\be \label{eq:u-satf-bdr-reg-fin}
\cL_K  u= f \qquad\textrm{ in $  B_{2}\cap\O$}\qquad\textrm{ and }\qquad u=0 \qquad\textrm{ in  $  B_{2}\cap\O^c$, }
\ee
 Then there exist $C ,\e_2>0$ and  $r_1 >0$, only depending on $N,s,\b, \L,  \k,c_0,\d$ and $\O$, such that if $\|\l\|_{L^\infty(B_2\times B_2)}<\e_2$,
 we have 
$$
 \|u/ d^s\|_{C^{s-\max(\b,\b')} ( B_{r_1}\cap \ov\O)}\leq C\left(\|u\|_{L^2(B_2)}+\|u\|_{\cL^1_s}+  \|f\|_{\cM_\b}  \right).
 $$

\end{corollary}
\begin{proof}
We assume that $\|u\|_{L^2(\R^N)}+  \|f\|_{\cM_\b}\leq 1$.  Consider $\psi\in L^\infty(B_1\cap\de\O)$ given by Lemma \ref{lem:unif-cont-f-Morrey}.
Let  $x_0\in \O\cap B_{r_0/4}$ and  $z_0\in \de\O\cap B_{1}$ be such that $|x_0-z_0|=d(x_0)\leq r_0/4$. Put $\rho={d(x_0)}/{2}$, so that $B_\rho(x_0)\subset B_{3\rho}(z_0)$ and $B_\rho(x_0)\subset  \subset \O. $ 
We define $v(x)=u(x)-\psi(z_0) d^s_2(x)$ and  $w_\rho(x):=\rho^{-s}   v(\rho x+ x_0)$. Then, letting $f_\rho(x):=\rho^sf(\rho x+ x_0)$ and   $ g_{\rho }(x)=\rho^s g_{\O, K}(\rho x+ x_0)$, we then have 
\be \label{eq:w-rho-solves-B1}
\cL_{K_{\rho,x_0}} w_\rho= f_\rho-\psi(z_0) g_{\rho }\qquad\textrm{ in $B_1$},
\ee
with $K_{\rho,x_0}(x,y)=\rho^{N+2s}K(\rho x+ x_0, \rho y+ x_0).$ We note that $K_{\rho,x_0}\in \scrK(\l_{\rho,x_0},a,\k)$, where $\l_{\rho,x_0}(x,y)=\l(\rh x+ x_0,\rho y+y_0)$. In particular, decreasing $r_0$ if necessary,  
$$
\|\l_{\rho,x_0} \| _{L^\infty(B_1\times B_1) }\leq \|\l  \| _{L^\infty(B_{2}\times B_{2}) }.
$$
Therefore,     by Lemma \ref{lem:unif-cont-f-Morrey}, we can find $\e_2>0$ small, independent on $\rho$ such that,  if  $\|\l  \| _{L^\infty(B_{2}\times B_{2}) }<\e_2$,   we have  
\begin{align}\label{eq:L2-norm-w-rho-B1}
\|w_\rho \|_{L^2(B_1)}=  \rho^{-N/2}  \|v\|_{L^2(B_{\rho}(x_0))}  \leq \rho^{-N/2} \rho^{-s} \|u-\psi(z_0)d^s\|_{L^2(B_{3\rho}(z_0))}\leq C \rho^{s-\max(\b,\b')}.
\end{align}
By  H\"older's inequality, we also  have $ \|v\|_{L^1(B_{r}(z_0))}\leq C r^{N+2s-\max(\b,\b')}$, for every $r>0$. We can thus proceed as in the proof of Theorem \ref{th:bound-reg-alpha}, to get  $\|w_\rho\|_{\cL^1_s}\leq C \rho^{s-\max(\b,\b')}$, with $C$ independent on $\b'$.
We note that by \eqref{eq:eta-f-scale-translate},  $\|f_\rho\|_{\cM_\b}\leq  \rho^{s-\b}$ and $\|g_\rho\|_{\cM_{\b'}}\leq  \rho^{s-\b'}\|g_{\O, K}\|_{\cM_{\b'}}\leq  c_0  $.   
It then  follows from  \eqref{eq:w-rho-solves-B1},  Corollary \ref{cor:int--Holder-reg-Morrey} and \eqref{eq:L2-norm-w-rho-B1}, that 
\begin{align*}
  \|w_\rho\|_{C^{s }(B_{1/2})}&\leq C\left(\|w_\rho \|_{L^2(B_1)}+ \|w_\rho\|_{\cL^1_s}+ \| f_\rho-\psi(z_0) g_{\rho }\|_{\cM_{s}} \right)\\
&\leq  C\left(\|w_\rho \|_{L^2(B_1)}+ \|w_\rho\|_{\cL^1_s}+ \| f_\rho-\psi(z_0) g_{\rho }\|_{\cM_{\max(\b,\b')}} \right) \\
&\leq C \rho^{s-\max(\b,\b')} .
\end{align*}
Hence, 
\begin{align*}
  \|w_\rho\|_{C^{s-\max(\b,\b')}(B_{1/2})}    \leq  C\rho^{s-\max(\b,\b')}.
\end{align*}
Scaling back, and since $d_2^s=d^s$ on $B_{\rho/2}(x_0)$,  we get 
 $$
  \|u-\psi(z_0) d^s\|_{L^\infty(B_{\rho/2}(x_0))}\leq C \rho^{2s-\max(\b,\b')}  \qquad \textrm{ and } \qquad    [u-\psi(z_0) d^s]_{C^{s-\max(\b,\b')} (B_{\rho/2}(x_0))}\leq C \rho^{s}.
  $$
Since $\|\psi\|_{L^\infty(B_{1}\cap\de\O)}\leq C$, the two  inequalities above imply that 
$$
   [u/ d^s]_{C^{s-\max(\b,\b')} (B_{\rho/2}(x_0))}\leq C ,
$$ 
which yields  (see the proof of Proposition 1.1 in     \cite{RS1})  
 $$
  [u/ d^s]_{C^{s-\max(\b,\b')} ( B_{r_1}\cap \ov\O)}\leq C,
 $$
 for some $r_1\leq r_0/4$ and $C>0$, only depending on $\O,N,s,\b, \L,\k$ and $\g$. By Lemma \ref{lem:unif-cont-f-Morrey},  $\|u\|_{L^2(B_r)}\leq C  r^{N/2+s}$.   Then using similar arguments as in the proof of  Theorem \ref{th:bound-reg-alpha}, we find that    $$\|u/d^s\|_{L^\infty( B_{r_1}\cap \O)} \leq C.$$
Finally in the general case $u\in H^s(B_2)\cap \cL^1_s$, we can use similar cut-off arguments as in the proof of Corollary \ref{cor:int--Holder-reg-Morrey} to get  the estimate involving only $\|u\|_{L^2(B_2)}+\|u\|_{\cL^1_s}$ in the place of $\|u\|_{L^2(\R^N)}$.
  The proof of the corollary is thus finished.
\end{proof} 

\section{Higher order interior regularity}\label{s:Higher-reg-int}
For $K$ a kernel satisfying \eqref{eq:Kernel-satisf}, we define the functions 
$$
J_{e,K}(x;y)=\frac{1}{2}(K(x,x+y) + K(x,x-y))  \qquad\textrm{ and }\qquad J_{o,K}(x;y)=\frac{1}{2}(K(x,x+y) - K(x,x-y)) .
$$
We suppose  in the following in this section that, for $2s>1$, the function $x\mapsto PV \int_{\R^N}y J_{o,K}(x;y)\,dy$ belongs to  $L^1_{loc}(B_2;\R^N)$.  We then consider the map  $j_{o,K}: B_2\to \R^N$ defined as 
\be\label{eq:def-joK}
j_{o,K}(x):=(2s-1)_+PV \int_{\R^N } y J_{o, K}(x;y)\, dy=(2s-1)_+\sum_{i=1}^N e_i  PV\int_{\R^N } y_i J_{o, K}(x;y)\, dy,
\ee
where $\ell_+:=\max(\ell,0)$ for all $\ell\in \R$.

We note that if $u\in C^{2s+\e}(\O)\cap \cL^1_s$, for some $\e>0$ and an open set $\O$,  then for every $\psi\in C^\infty_c(\O)$, we have 
  \begin{align}\label{eq:decomp-weak-sol}
\frac{1}{2} \int_{\R^{2N}}(u(x)-u(y))(&\psi(x)-\psi(y))K(x,y)\,dxdy \nonumber\\
&=  \int_{\R^{N}}\psi(x)\left [PV\int_{\R^N}(u(x)-u(x+y)) J_{e,K}(x;y)\, dy \right]dx \nonumber\\
&+   \int_{\R^{N}}\psi(x) \left[PV \int_{\R^N}(u(x)-u(x+y)) J_{o,K}(x;y)\,dy \right]dx.
\end{align}
Moreover  for every $x\in \O$, we have 
\be\label{eq:ev-part-cL_K-int}
PV\int_{\R^N}(u(x)-u(x+y)) J_{e,K}(x;y)\, dy  = \frac{1}{2} \int_{\R^N}(2u(x)-u(x+y)-u(x-y)) J_{e,K}(x;y)\, dy .
\ee

We consider the family of affine functions 
$$
q_{t,T}(x)= t +(2s-1)_+ T\cdot x  \qquad\textrm{ $ t\in \R$ and $T\in \R^N$.}
$$
For $z\in \R^N$, we define the following  finite dimensional  subspace of $L^2(B_r(z))$, given  by     
$$
\cH_{z}:=\{q_{t,T}(\cdot -z)\,:\, t\in \R,\, T \in\R^N\}. 
$$
For $u\in L^2_{loc}(\R^N)$,    $r>0$ and $z\in \R^N $, we let  $\textbf{P}_{r,z}(u)\in \cH_{z}$ its $L^2(B_r(z))$-projection on $\cH_{z}$. Then 
\be \label{eq:u-perp-poly}
\int_{B_r(z)}(u(x)-\textbf{P}_{r,z}(u)(x)) p(x)\, dx=0 \qquad\textrm{ for every $p\in \cH_{z}$}.
\ee
 \begin{lemma} \label{lem:bound-Morrey}
Let $s\in (1/2,1)$,     $\b\in (0,2s-1)$,   $\L,\k>0$ and $\d\in (0,2s-1)$.   Then    there  exist $C>0$ and $\e_0>0$ such that for every 
 \begin{itemize}
 \item  $a$ satisfying  \eqref{eq:def-a-anisotropi}, 
 \item  $\l: \R^N\times \R^N\to [0, k^{-1}]$ satisfying  $\|\l\|_{L^\infty(B_2\times B_2)}<\e_0$,
 \item ${K}\in \scrK(\l,a,\k) $ satisfying      $\| \vp_2 j_{o,K}\|_{\cM_{\b'}}\leq 1$ (see \eqref{eq:def-joK}), for some $\b'\in [0,2s-1-\d)$,
 \item    $f\in \cM_\b$ and $u\in \cS_{K,0,f}$ satisfying $\|u\|_{L^\infty(\R^N)}+ \|f\|_{\cM_\b}\leq 1$,   
\end{itemize}
we have 
\be\label{eq:L-infty-growth-u-Hint}
\sup_{r>0} r^{-(2s-\max(\b,\b'))}   \sup_{z\in B_{1}}    \|u-\textbf{P}_{r,z}(u)\|_{L^\infty(B_r(z))} \leq C  ,
\ee
provided $ 2s-\b>1$ if $2s>1$.
\end{lemma}
 \begin{proof}
Then as in the proof of Proposition \ref{prop:bound-Kato-abstract}, if \eqref{eq:L-infty-growth-u-Hint} does not hold, then we can find a sequence $r_n\to 0$,  points $z_n\in B_1 $ and sequences of functions, $a_n$ satisfying \eqref{eq:def-a-anisotropi}, $\l_n$ with $\|\l_n\|_{L^\infty(B_2\times B_2)}<\frac{1}{n}$, $K_{a_n}\in \scrK(\l_n,a_n,\k)$ with $\| \vp_2 j_{o,K_{a_n}}\|_{\cM_{\b'_n}}\leq 1$ and $\b'_n \in  (0, 2s-1-\d)$,      $f_n\in  \cM_\b$  and $u_n\in \cS_{K_{a_n}, 0, f_n}$, with $\|u_n\|_{L^\infty(\R^N)}+\|f_n\|_{\cM_\b}\leq 1 $,  such that 
\be\label{eq:resclae-situation-pwb-Morrey-hi-Int}
      r^{ -(2s-\max(\b,\b_n'))}_n  \|u_n -\textbf{P}_{r_n,z_n}(u_n) \|_{L^\infty(B_{r_n} (z_n))}\geq  \frac{1}{16} \Theta_n(r_n) \geq  \frac{n}{32} .
\ee
 Here also $\Theta_n$ is a nonincreasing function on $(0,\infty)$ satisfying
\be \label{eq:sup-r-ok-Morrey-hi-Int}
\Theta_n( r)\geq  r^{-(2s-\max(\b,\b_n'))} \sup_{z\in B_1}  \| u_n- \textbf{P}_{r,z}(u_n)\|_{L^\infty(B_{r}(z))}   \qquad\textrm{ for every $r>0$ and $n\geq 2$.}
\ee
To alleviate the notations, we put $\b_n:=\max(\b,\b'_n)\leq \max(\b, 2s-\d)<2s$, for every $n\geq 2$.
 We   define  
$$
 {w}_n(x)=  \Theta_n(r_n)  r^{-(2s-\b_n)}_n  [u_n(r_n x+ z_n) - \textbf{P}_{r_n,z_n}(u_n)(r_n x+ z_n)], 
$$
so that 
\be \label{eq:w-n-nonzero-Morrey-int-reg}
 \|w_n\|_{L^\infty(B_{1})}^2\geq \frac{1}{16}
\ee
and, thanks to \eqref{eq:u-perp-poly},   by a change of variable,
\be\label{eq:w-n-zero-mean}
\int_{B_1} w_n(x) p(x)\, dx=0 \qquad\textrm{ for every $p\in \cH_{0}$.}
\ee

\bigskip
\noindent
\textbf{Claim:} There exists $C= C(  s,\b,  N,\d)>0$     such that  
\be\label{eq:groht-w-n-Morrey-HI}
 \|w_n\|_{L^\infty(B_{R})}\leq  C   R^{2s-\b_n}  \qquad\textrm{ for every $R\geq 1$.}
\ee
To prove this claim, we   note that by a change of variable, we have
\begin{align*}
 \|w_n\|_{L^\infty(B_{R})}&=   \Theta_n(r_n)^{-1}   r_n^{-N} r^{-(2s-\b_n)}_n \|u- \textbf{P}_{r_n,z_n}(u_n)\|_{L^\infty(B_{r_nR}(z_n))}\\
& \leq   \Theta_n(r_n)^{-1}     r^{-(2s-\b_n)}_n \|u- \textbf{P}_{r_n R,z_n}(u_n)\|_{L^\infty(B_{r_nR}(z_n))}\\
&+   \Theta_n(r_n)^{-1}  r^{-(2s-\b_n)}_n  \|\textbf{P}_{r_n R,z_n}(u_n)- \textbf{P}_{r_n,z_n }(u_n)\|_{L^\infty(B_{r_nR}(z_n))}.
%
%
\end{align*}
We write $ \textbf{P}_{r,z }(u_n)(x)=t(r)+T(r)\cdot (x-z)$,      for $r>0$ and $z\in B_1$. Then,  we have 
\begin{align*}
\left(|t(2r)-t_z(r)|^2+ |T(2r)-T(r)|^2 r^2\right)|B_{r}|^{{2}} &= \|\textbf{P}_{2r,z }(u_n)- \textbf{P}_{r,z }(u_n)\|_{L^2(B_{ r}(z))}^2\\
& \leq2   \|\textbf{P}_{2r,z }(u_n)-  u_n\|_{L^2(B_{2r}(z) )}^2+2  \|\textbf{P}_{r ,z}(u_n)- u_n\|_{L^2(B_{r}(z) )}^2\\
&\leq  2  (2 r)^{2(2s-\b_n)}    \Theta_n(2 r)^{2}   (2 r) ^{N}+  2 r^{2(2s-\b_n)}    \Theta_n( r)^{2}    r ^{N}\\
&\leq C  r ^{N}  \Theta_n( r)^{2}   r^{2(2s-\b_n)}  ,
\end{align*}
where we have used the monotonicity of $\Theta_n$.
We then have, for every $r>0$,
$$
|t(2r)-t(r)|+ |T(2r)-T(r)| r \leq C     \Theta_n( r)   r^{2s-\b_n} .
$$
Hence, since $2s-\b_n\geq 1$ if $2s>1$, for every integer  $m\geq 1$, we get 
\begin{align*}
  |T(2^mr)-T(r)| &=   \sum_{i=1}^m |T(2^{i}r)-T(2^{i-1}r)|\\
  & \leq  C  \Theta_n(2^{i-1}  r )  r^{2s-\b_n-1} \sum_{i=1}^m 2^{(i-1)(2s-\b_n-1)}\leq  C     \Theta_n(  r )   (2^{m} r)^{2s-\b_n-1}  ,
\end{align*}
with $C>0$ a constant independent on   $m$  and  on $n\geq 2$, since $2s-\b_n-1\geq \min(2s-1-\b,d) >0$.
Similarly, we also have that $ |t(2^mr)-t(r)| \leq  C      \Theta_n(  r )  (2^{m} r)^{2s-\b_n}     $.

Now for $R\geq 1$, letting $m$ be the smallest integer such that  $2^{m-1}\leq R \leq 2^{m}$, we then get
\begin{align*}
 \|w_n\|_{L^\infty(B_{R})}^2&=   \Theta_n(r_n)^{-1}     r_n^{-(2s-\b_n)}  \|u- \textbf{P}_{r_n,z_n}(u_n)\|_{L^\infty(B_{r_nR}(z_n))}\\
& \leq   \Theta_n(r_n)^{-1}    r_n^{-(2s-\b_n)}  \|u- \textbf{P}_{r_n R,z_n}(u_n)\|_{L^\infty(B_{r_nR}(z_n))}\\
&+   \Theta_n(r_n)^{-1}     r_n^{-(2s-\b_n)}  \|\textbf{P}_{r_n R,z_n}(u_n)- \textbf{P}_{r_n,z_n }(u_n)\|_{L^\infty(B_{r_nR}(z_n))}\\
&\leq  C  \Theta_n(r_n)^{-1}    r_n ^{-(2s-\b_n)}  (r_nR)^{(2s-\b_n)}    \Theta_n(R r_n)     \\
&+    \Theta_n(r_n)^{-1}     r_n^{-(2s-\b_n)}   \left(   |t(Rr_n)-t(r_n)|+ |T(Rr_n)-T(r_n)| r_nR  \right)\\
&\leq   C  \Theta_n(r_n)^{-1}    r_n ^{-(2s-\b_n)}  (r_nR)^{ 2s-\b_n }    \Theta_n(R r_n).
\end{align*}
By the monotonicity of $\Theta_n$, we get  the  claim.\\
It follows from \eqref{eq:groht-w-n-Morrey-HI} that
\be \label{eq:good-growth-of-v_n-Kato-pw-Morrey-int}
\|w_n\|_{\cL ^1_s}\leq C \qquad \textrm{ for every $n\geq 2$.}
\ee
We define $K_n(x,y):=r_n^{N+2s}K_{a_n}(r_n x+z_n, r_n y+z_n)$,  and  we note that  
$$
J_{o,K_n}(x;y) =r_n^{N+2s}J_{o,K_{a_n}}(r_n x+ z_n ; r_n y).
$$
We put $\textbf{P}_n(x):=\textbf{P}_{r_n,z_n}(u_n)(r_n x+z_n)$ and let $\psi\in C^\infty_c(\R^N)$. We use \eqref{eq:ev-part-cL_K-int},  to get 
   \begin{align*}
\frac{1}{2} \int_{\R^{2N}}&(\textbf{P}_n(x)-\textbf{P}_n(y))(\psi(x)-\psi(y))K_n(x,y)\,dxdy\\
%
%
%
&=  0+ \int_{\R^{N}}\psi(x) \left[PV\int_{\R^N}(\textbf{P}_n(x)-\textbf{P}_n(x+y)) J_{o,K_n}(x;y)\,dy \right]dx.
&
\end{align*}
Therefore writing  $\textbf{P}_{r_n,z_n }(u_n)(x)=t_n+(2s-1)_+T_n\cdot (x-z_n)$, we see that 
   \begin{align*}
PV\int_{\R^N}(\textbf{P}_n(x)-\textbf{P}_n(x+y)) J_{o,K_n}(x;y)\,dy=(2s-1)_+  (r_n T_n)\cdot \left( r_n^{2s} j_{o,K_{a_n}}(r_n x+ z_n) \right).
&
\end{align*}
We then conclude that  
\be \label{eq:of-vn-Kato-Morrey}
\cL_{K_n} w_n  =  r_n ^{-(2s-\b_n)} \Theta_n(r_n)^{-1}\left(  \ov f_n+  (2s-1)_+  h_n \right)   \qquad\textrm{ in $B_{1/{2r_n}} $,}
\ee 
where, noting that $\vp_2\equiv 1$ on $B_2$ and recalling \eqref{eq:def-joK},     
$$
\ov f_{n} (x) :=    r_n^{2s} f_n(r_n x+z_n) \qquad\textrm{ and } \qquad h_n(x)= (r_n T_n)\cdot \left( r_n^{2s} j_{o,K_{a_n}}(r_n x+ z_n)  \right)\vp_{2}(r_n x+ z_n).
$$
   Since $\|u_n\|_{L^\infty(\R^N)}\leq 1$, then  $|T_n|\leq r_n ^{-1}$. Therefore, since by assumption, $ \| \vp_2 j_{o,K_{a_n}}\|_{\cM_{\b_n'}} \leq 1$, we deduce that 
\be\label{eq:ov-fn-r2s-b}
 \|\ov f_{n}\|_{\cM_\b}+  \|h_{n}\|_{\cM_{\b'_n}}\leq  2 r^{2s-\b_n}_n\leq 2.
\ee
Next, we note that $K_n\in \scrK(\ti \l_n, a_n,\k)$, with $\ti \l_n(x,y)=\l_n(r_n x+ z_n, r_n y+z_n)$. By assumption,   $\|\ti \l_n\|_{L^\infty(B_{1/(2r_n)}\times B_{1/(2r_n)})}\leq \frac{1}{n}$. 
 Now by Corollary \ref{cor:int--Holder-reg-Morrey}, \eqref{eq:good-growth-of-v_n-Kato-pw-Morrey-int} and   \eqref{eq:ov-fn-r2s-b},  we deduce that $w_n$ is bounded in $C^\d_{loc}(\R^N) $, for some $\d>0$. In addition  thanks to  \eqref{eq:groht-w-n-Morrey-HI}, up to a subsequence,  it converges in $\cL^1_s\cap C^{\d/2}_{loc}(\R^N)$  to some  $w\in C^\d_{loc}(\R^N)\cap\cL^1_s$.  
Moreover, by \eqref{eq:w-n-nonzero-Morrey-int-reg} and \eqref{eq:w-n-zero-mean}, we deduce that 
\be\label{eq:w-nonzero-Morrey-HI}
  \|w\|_{L^\infty(B_1)} \geq \frac{1}{16}
\ee
 and 
 \be\label{eq:w-zero-mean-Morrey-HI} 
 \int_{B_1} w(x) p(x)\,dx=0 \qquad\textrm{ for every $p\in \cH_{0}$}.
 \ee
 We apply   Lemma \ref{lem:from-caciopp-ok} (after a cut-off argument as in the proof of Proposition \ref{prop:bound-Kato-abstract}),  use \eqref{eq:groht-w-n-Morrey-HI} and  \eqref{eq:ov-fn-r2s-b}   to   get 
\begin{align*}
\left\{  \k - \e \ov C \right\} &[\vp _M w_n]_{H^s(B_{M/2})}^2\leq       C(M) \qquad\textrm{ whenever  $1<M<\frac{1}{2r_n}$,}
\end{align*}
where we used the fact that $ \Theta_n(r_n)^{-1}\leq 1$, for every $n\geq 2$.
Therefore, provided   $\e$ is small enough, we find that   $w_n$ is bounded in $H^s_{loc}(\R^N)$ and thus  $w\in H^s_{loc}(\R^N)$. Now by Lemma \ref{lem:convergence-very-weak} and \eqref{eq:ov-fn-r2s-b},      we have
 \begin{align*} 
\left|\int_{\R^{2N}}(v_n(x)    -v_n(y))(\psi(x)-\psi(y))K_n(x,y)\, dxdy \right| 
&\leq \Theta_n(r_n)^{-1}  C (M)  .
\end{align*}
Letting $n\to \infty$ in the above inequality and using Lemma \ref{lem:Ds-a-n--to-La}, we find that  $L_b w=0$ in $ \R^N$, with $b$ the limit of $a_n$ in the weak-star topology of $L^\infty(S^{N-1})$.  
By   \eqref{eq:groht-w-n-Morrey-HI} and   Lemma \ref{lem:Liouville},   $w\in \cH_0$, which   contradicts \eqref{eq:w-zero-mean-Morrey-HI}  and \eqref{eq:w-nonzero-Morrey-HI}.

\end{proof} 
We  now have the following  $C^{2s-\b}$ regularity estimates for $2s-\b>1$ --- understanding that    $C^{2s-\b}=C^{1, 2s-\b-1}$ if $2s-\b>1$ by an abuse of notation. 
\begin{corollary}\label{cor:int--Holder-reg-Morrey-fin}
Let  $s\in (1/2,1)$, $\b,\d\in (0,2s-1)$ and      $\k,\L>0$. Let  $a$ satisfy \eqref{eq:def-a-anisotropi} and   $K\in \scrK(\l,a,\k)$ satisfy   $\| \vp_2 j_{o,K}\|_{ \cM_{\b'}}\leq c_0$, for some $c_0\geq 0$ and $\b'\in [0, 2s-1-\d)$.
  Let $f\in \cM_\b$ and  $u\in H^s (B_2)\cap \cL^1_s$   satisfy
$$
\cL_{K} u= f \qquad\textrm{ in $B_2$}.
$$
Then there exists $\e_0>0$, only depending on $ N,s,\b,\k,c_0,\d$ and $\L$, such that   if $\|\l\|_{L^\infty(B_2\times B_2)}<\e_0$, then 
  $u\in C^{2s-\max(\b,\b')}(B_{1/2})$. Moreover, there exists $C=C(N,s,\b,\L,\k,c_0,\d )$ such that 
  $$
  \|u\|_{C^{2s-\max(\b,\b')}(B_{1/2})}\leq C(\|u\|_{L^2(B_2)}+ \|u\|_{\cL^1_s}+ \|f\|_{\cM_\b}).
  $$
\end{corollary}
 \begin{proof}
We first assume that $\|u\|_{L^\infty(\R^N)}+  \|f\|_{\cM_\b} \leq 1$ and $\|j_{o,K}\|_{\cM_{\b'}}\leq 1$.  By a well known iteration argument (see e.g \cite{Serra-OK} or the proof of Lemma \ref{lem:unif-cont-f-Morrey}), we find that 
$$
|u(x)- u(z)-   (2s-1) T(z)\cdot( x-z) |\leq C |x-z|^{2s-\max(\b,\b')} \qquad\textrm{ for every $x,z\in B_1$,}
$$
with $\|T\|_{L^\infty(B_1)}\leq C$, 
provided  $\|\l\|_{L^\infty(B_2\times B_2)}<\e_0$, with $\e_0$ given by Lemma \ref{lem:bound-Morrey}. In particular, since $2s-\max(\b,\b')>1$ then $\n u(z)=(2s-1)T(z)$. Note that since $\b'\in [0, 2s-1-\d)$, the constant $C$ does not depend on $\b'$ but on $\d$ (see the proof of Lemma \ref{lem:unif-cont-f-Morrey}).
By   a classical extension theorem (see e.g. \cite{Stein}[Page 177], we deduce that 
  $u\in C^{2s-\max(\b,\b')}(\ov {B_{1/2}})$.    Moreover
$$
\|u \|_{C^{2s-\max(\b,\b')}(\ov {B_{1/2}})}\leq C.
$$
 Now  for the general case   $u\in {H^s(B_2)},   f\in \cM_\b$ and $ \|\vp_2 j_{o,K}\|_{\cM_\b}\leq c_0 $, we  use  cut-off and scaling  arguments as in the proof of Corollary \ref{cor:int--Holder-reg-Morrey} to get
$$
\|u \|_{C^{2s-\max(\b,\b')}(B_{1/4} )}\leq C\left( \|u\|_{L^\infty(B_{1/2})}+ \|u\|_{\cL^1_s}+ \|f\|_{\cM_\b} \right).
$$ 
Now, decreasing $\e_0$ if necessary,  by  Corollary \ref{cor:int--Holder-reg-Morrey} we have   
$$
\|u \|_{{L^\infty}(B_{1/2} )}\leq C\left( \|u\|_{L^2(B_{1})}+ \|u\|_{\cL^1_s}+ \|f\|_{\cM_\b} \right).
$$
 The proof of the corollary is thus finished.

\end{proof}

 \section{Proof of the main results}\label{s:proof-main-results}

\begin{proof}[Proof of Theorem \ref{th:mainth}]
Suppose that   $\O$ is domain of class $C^{1,\g}$, with $0\in \de\O$. We consider $\O'$ a bounded domain of class $C^{1,\g}$ which coincides with $\de\O$ in a neighborhood of $0$.  We let  $\ov r>0$ small so that the distance function $d=\textrm{dist}(\cdot, \R^N\setminus\O)$ is of class $C^{1,\g}$ in $ \O\cap B_{4 \ov r}$ and $d_{\O'}(x):=\textrm{dist}(x, \R^N\setminus\O')=d(x)$ for every $x\in \ov  \O\cap B_{4 \ov r}$. Now, for $x\in \O\cap B_{\ov r}$, we have 
$$
\Ds_a(\vp_{2\ov r} d^s)= \Ds_a(\vp_{2\ov r} (d^s-d^s_{\O'}))+ \Ds_a(\vp_{2\ov r} d^s_{\O'}).
$$ 
By   \cite[Proposition 2.3 and 2.6]{RS3} and Lemma \ref{lem:cat-off-sol}, for  $\g\not=s$,  there exists a constant $C=C(\O,N,s,\L)>0$, such that   
$$
|\Ds_a(\vp_{2\ov r} d^s_{\O'})(x)|\leq  C d^{(s-\g)_+}_{\O'}(x) \qquad\textrm{ for every $x\in \O'$}.
$$
Since  $d^s-d^s_{\O'}=0$ on  $\O\cap B_{4 \ov r}$, we get  $| \Ds_a(\vp_{2\ov r} (d^s-d^s_{\O'}))|\leq C$ on   $\O\cap B_{ \ov r}$. 
We define $g_{\O,\mu_a}(x)=\Ds_a (\vp_{2\ov r} d^s)(x)$ for $x\in \O\cap B_{\ov r}$ and $g_{\O,\mu_a}(x)=0$ for $x\in \R^N\setminus (\O\cap B_{\ov r})$. Then, there exists a constant $C=C(\O,N,s,\L)>0$ such that   
$$
|g_{\O,\mu_a} (x)|\leq C\max(d^{\g-s}(x), 1) \qquad\textrm{ for every $x\in \R^N$}. 
$$
Letting $\b'=(s-\g)_+$, for $\g\not=s$, we then deduce that  $\|g_{\O,\mu_a}\|_{\cM_{\b'}}\leq C_1(s,N,\O,\g,\L)$. When $\g=s$, we can let $\b'=\e$ with $\e<\b$.  Moreover, up to scaling $\O$ to  $\frac{1}{\d}\O$, for some small  $\d>0$ depending only on $C_1$,   we may assume that $\|g_{\O,\mu_a}\|_{\cM_{\b'}}\leq 1$ and that $\de\O $ seperates $B_2$ into two domains.\\
   By Theorem \ref{th:bound-reg-alpha},  there exists $C>0$,  only depending on $N,s,\O,\b,\L,\g$ and $\|V\|_{\cM_\b}$, such that       
$$
    \|u\|_{L^\infty(B_{\ov r/2})}\leq C(\|u\|_{L^2(B_{\ov r})}+ \|u\|_{\cL^1_s}+ \|f\|_{\cM_\b}).
$$
Then applying Corollary \ref{cor:unif-cont-f-Morrey}, we get the Theorem \ref{th:mainth}(ii).   Now Theorem \ref{th:mainth}(i)  follows immediately from Corollary \ref{cor:int--Holder-reg-Morrey} and Corollary \ref{cor:int--Holder-reg-Morrey-fin}.
\end{proof}
%
%
%
%
%
\begin{proof}[Proof of Theorem \ref{th:int-reg-Morrey-intro}]
It suffices to apply    Corollary \ref{cor:int--Holder-reg-Morrey} to get $L^\infty$-bound and then apply   Theorem \ref{th:int--Holder-reg-Morrey} to get the result for $2s-\max(\b,\b')>1$. If $2s\leq 1$, then the result follows from Corollary \ref{cor:int--Holder-reg-Morrey}.
\end{proof}

\begin{proof}[Proof of Corollary \ref{cor:int-reg-Morrey-intro}]
Using a scaling and a covering argument as in the proof of Theorem \ref{th:int--Holder-reg-Morrey}  together with Theorem \ref{th:int-reg-Morrey-intro}, we get the result.
\end{proof}

\begin{proof}[Proof of Theorem \ref{th:mainth-gen}]
It suffices to apply  Theorem \ref{th:bound-reg-alpha} to get $L^\infty$-bound  and then apply   Corollary \ref{cor:unif-cont-f-Morrey}.
\end{proof}

\begin{proof}[Proof of Corollary \ref{cor:mainth-gen}]
Using a scaling and a covering argument as in the proof of Theorem \ref{th:int--Holder-reg-Morrey} and applying Theorem \ref{th:mainth-gen}, we get the result.
\end{proof}

\begin{proof}[Proof of Theorem  \ref{th:Small-Shauder}]
%
%
By assumption,  for every  $x_0\in \ov{B_{1}}$ and  $\e>0$, there exists $r_\e=r(\e,x_0)\in (0,1/100)$ such that
$$
|\l_{K}(x,r,\th)-\l_{K}(x_0,0,\th)|<\e\qquad\textrm{ for all $x\in B_{16 r_\e}(x_0)$ and   all $r\in (0,16r_\e)$} .
$$
This implies that 
\be\label{eq:cont-l-K}
|K(x,y)-\mu_a(x,y)|<\e \mu_1(x,y)\qquad\textrm{ for $x\not= y\in B_{8r_\e}(x_0)$} ,
\ee
where $a(\th)= \l_K(x_0,0,\th)$, which is even, since $\l_{o,K}(x_0,0,\th)=0$.  
Letting $ K_\e(x,y)=r_\e^{N+2s}K(r_\e x+ x_0, r_\e y+ x_0)$, $v_\e(x)=v(r_\e x+x_0)$ and  $f_\e(x)=r_\e^{2s}f(r_\e x+ x_0)$, we then have 
\be
\cL_{ K_\e} v_\e=  f_\e  \qquad B_8.
\ee
We note that
\begin{align*}
\ti   \l_{ K_{\e}}(x,r,\th)=\ti  \l_{K}(r_\e x+x_0,r_\e r,\th) 
\end{align*}
and thus, by assumption, 
\begin{align*}
\ti   \l_{o, K_{\e}}(x,r,\th)=\ti  \l_{o, K}(r_\e x+x_0,r_\e r,\th)\leq C  r^{\a+(2s-1)_+}.
\end{align*}
Clearly,  by \eqref{eq:cont-l-K}, 
\be\label{eq:K-eps-close-to-mua}
|K_{\e}(x,y)-\mu_a(x,y)|<\e \mu_1(x,y) \qquad\textrm{ for $x\not= y\in B_{8 }$.} 
\ee
By  Corollary \ref{cor:int--Holder-reg-Morrey}, provided $\e$ is small,  for every $\varrho\in (0,s/2)$,  we have 
\be\label{eq:first-estim-v-CNM}
\|v_\e\|_{C^{2s-\varrho}(B_7)}\leq C  \left( \|v_\e \|_{L^{\infty}(\R^{N})}+ \|f_\e\|_{L^\infty(B_8) } \right) ,
\ee
provided $2s-\varrho\not=1$.
We let  $v_{1,\e} :=\vp_1 v_\e\in H^{s} (\R^{N})\cap C^{2s-\varrho}_c(B_2)$. Then, we have 
\be \label{eq:v-1-solves-CNMC}
\cL_{K _\e}v_{1,\e}=f_{1,\e} \qquad\textrm{in $B_2$},
\ee
with 
\begin{align*}
f_{1,\e} (x)&=   f_\e(x)+ G_{v_\e}(x), 
\end{align*}
where 
\begin{align*}
 G_{v_\e}(x)&=  \int_{\R^{N}}{v_\e}(y)(\vp_1(x)-\vp_1(y))  {K_\e}(x,y) \, dy\\
 &= \int_{\R^{N}}({v_\e}(y)-{v_\e}(x))(\vp_1(x)-\vp_1(y))  K_\e(x,y) \, dy+{v_\e}(x) \int_{\R^{N}} (\vp_1(x)-\vp_1(y))  {K_\e}(x,y) \, dy \\
 &=\ti G_\e(x)  + \frac{v_\e(x)}{2} \int_{S^{N-1}}\int_{0}^\infty (2\vp_1(x)- \vp_1(x+ r\th)- \vp_1(x- r\th) )r^{-1-2s} \ti \l_{e,{K_\e}}(x,r,\th) \, dr d\th \\
 &+ \frac{v_\e(x)}{2} \int_{S^{N-1}}\int_{0}^\infty (\vp_1(x+ r\th)- \vp_1(x- r\th) ) r^{-1-2s} \ti \l_{o,{K_\e}}(x,r,\th) \, dr d\th ,
\end{align*}
where 
\begin{align*}
\ti G_\e(x) &:=\int_{S^{N-1}}\int_{0}^\infty({v_\e}(x+ r\th )-{v_\e}(x))(\vp_1(x)-\vp_1(x+r\th ))   r^{-1-2s}\ti \l_{K_\e}(x,r,\th ) \, dr d\th .
\end{align*}
Since $v_\e\in  C^{2s-\varrho}(B_7)\cap C^\a(\R^N)$, and $\vp_1\in C^1(\R^N)$, while the map $x\mapsto \ti \l_{K_\e}(x,r,\th )\in C^{\a }(\R^N)$, direct computations show that
$$
\|\ti G_{v_\e}\|_{C^{\b}(B_3)}\leq C( \|{v_{\e}}\|_{C^{2s-\varrho}(B_4)}+  \|{v_\e}\|_{C^{\b}(\R^N)}   ) ,
$$
for every $\b\in(0, \min(\a, 2s-\varrho))$.
Now since $\vp_1\in C^2(\R^N)$ and  
\be \label{eq:til-o-x1-x-2}
|\ti\l_{o,K_\e}(x_1,r,\th)-\ti\l_{o,K_\e}(x_2,r,\th)|\leq C \min (r,|x_1-x_2|)^{\a+(2s-1)_+},
\ee
 then     we can find an  $\a_0>0$, only depending on $\a,s$ and $\varrho$, such that  for all $\b\in (0,\a_0)$, 
$$
\| G_{v_\e}\|_{C^{\b}(B_3)}\leq C( \|{v_{\e}}\|_{C^{2s-\varrho}(B_4)}+  \|{v_\e}\|_{C^{\b}(\R^N)}   ) .
$$
By this, we deduce that 
\begin{align} \label{eq:estime-f_1-epsilon}
\|f_{1,\e}\|_{C^{\b}(B_3)}& \leq C(\|f_{\e}\|_{C^{\b}(B_3)} +   \|{v_{\e}}\|_{C^{2s-\varrho}(B_4)}+  \|{v_\e}\|_{C^{\b}(\R^N)}   ) \nonumber\\
&\leq C(\|f_{\e}\|_{C^{\b}(B_8)} +  \|{v_\e}\|_{C^{\b}(\R^N)}  )  ,
\end{align}
where, we used   \eqref{eq:first-estim-v-CNM} for the last inequality, provided $\a_0<2s-\varrho.$ 

We consider a nonegative function $\eta\in C^\infty_c(\R)$  satisfying  $\eta(t)=1$ for $|t|\leq 1$ and $\eta(t)=0$ for $|t|\geq 2$. We put $\eta_\d(t)=\eta(t/\d)$.
We now define 
$$
K_{\e,\d}(x,y)=\eta_\d(|x-y|)\mu_a(x,y)+(1-\eta_\d(|x-y|))  K_\e(x,y),
$$
and we note that,  by \eqref{eq:K-eps-close-to-mua},  
\be\label{eq:K-e-delt-near-mu-a}
|K_{\e,\d}(x,y)-\mu_a(x,y)|<\e (1+\|\eta\|_{L^\infty(\R)} ) \qquad\textrm{ for $x\not= y\in B_{8 }$} .
\ee
In addition, 
\be\label{eq:ti-lambda-e-delta}
\ti\l_{K_{\e,\d}}(x,r,\th)=\eta_\d(r)a(\th)+   (1-\eta_\d(r))\ti\l_{K_\e}(x,r,\th) .
\ee
  Now for  $\e,\d> 0$,  we consider $w_{\e,\d}\in H^s (\R^{N}) $, the (unique) weak solution to 
\be\label{eq:w-eps-delt}
\begin{cases}
\cL_{K_{\e,\d}} w_{\e,\d}=f_{1,\e} &\qquad\textrm{in $B_2$}\\
w_{\e,\d}= v_{1,\e}=0 & \qquad\textrm{in $\R^{N}\setminus B_2$}.
\end{cases} 
\ee
Multiplying \eqref{eq:w-eps-delt} by $ w_{\e,\d}$, integrating, using the symmetry of $K_{\e,\d}$ and H\"older's inequality, we deduce that 
$$
 [w_{\e,\d} ]_{ {H}^{s}(\R^{N} ) }^2\leq   \|f_{1,\e}\|_{L^2(B_2)} \|w_{\e,\d}\|_{L^2(B_2)}.
$$
Hence by the Poincar\'e inequality, we find that 
\be \label{eq:w-e-delta-L-2-bound}
 \|w_{\e,\d} \|_{ {H}^{s}(\R^{N} ) }\leq C   \|f_{1,\e}\|_{L^2(B_2)} ,
\ee
where here and in the following,  the letter    $C$ denotes a constant, which may vary from line to line but   independent on $\d,f$ and $v$.
Thanks to  \eqref{eq:K-e-delt-near-mu-a},  we can apply Theorem \ref{th:bound-reg-alpha}   together with  \eqref{eq:w-e-delta-L-2-bound},  to get
\be \label{eq:-estim-2s--eps-NMC-R-N}
\|w_{\e,\d}\|_{C^{s-\varrho}(\R^N)}\leq C  (\|w_{\e,\d}\|_{L^2(\R^N)}+ \|f_{1,\e}\|_{L^{\infty}(B_2) })\leq C \|f_{1,\e}\|_{L^{\infty}(B_2) },
\ee
 provided $\e$ small, independent on $\d$.
Furthermore by  Corollary \ref{cor:int--Holder-reg-Morrey},    provided $\e$ small and  independent on $\d$, using  \eqref{eq:w-e-delta-L-2-bound},     we have  that 
\be \label{eq:-estim-2s--eps-NMC}
\|w_{\e,\d}\|_{C^{2s-\varrho}(B_{1})}\leq C  (\|w_{\e,\d}\|_{L^2(\R^N)}+ \|f_{1,\e}\|_{L^{\infty}(B_2) })\leq C \|f_{1,\e}\|_{L^{\infty}(B_2) }.
\ee
On the other hand,  multiplying \eqref{eq:w-eps-delt} by  $w_{\e,\d}-v_{1,\e}$, we see    that
\be \label{eq:w-delt-v-1}
 [w_{\e,\d}-v_{1,\e}]_{ {H}^{s}(\R^{N} ) }\leq  \int_{\R^N\times \R^N }(v_{1,\e}(x)-v_{1,\e}(y))^2|K_{\e,\d}(x,y)- K_\e(x,y)|\,dxdy .
\ee
 Hence,  by the Poincar\'e inequality,  the dominated convergence and \eqref{eq:-estim-2s--eps-NMC-R-N},  as  $\d\to 0$, we may assume that $w_{\e, \d}\to v_{1,\e}$  in $C^{s-2\varrho}(\R^N)\cap  C^{2s-2\varrho}(\ov{B_{1/2}})$.  
Now for $\d>0$,   we have  
\begin{align*}
\cL_{\mu_a} w_{\e,\d}&=f_{1,\e}+ H_{\e,\d} \qquad\textrm{in $B_2$,}
\end{align*}
where  $H_\e\in C^{\min(\a,s-\varrho)}(\R^N)$ is given by 
\begin{align*}
H_{\e,\d}(x)&:= \int_{S^{N-1}}\int_{0}^\infty (w_{\e,\d}(x)- w_{\e,\d}(x+r\th ) ) (1-\eta_\d(r)) r^{-1-2s}(\ti \l_{{\mu_a } }(x,r,\th)-\ti  \l_{{ K_{\e}} }(x,r,\th))\, dr d\th. 
\end{align*}
It follows from \cite[Theorem 1.1]{RS2} and \eqref{eq:estime-f_1-epsilon} that  $w_{\e,\d}\in C^{2s+\a_1}_{loc}(B_2)$, for  some $\a_1>0$, only depending on $\a,s$ and $\varrho$.  Now by \eqref{eq:decomp-weak-sol}, 
\be \label{eq:L-b-delt-w-d}
L_{K_{\e,\d}}(x, w_{\e,\d})=   F_{\e,\d} \qquad\textrm{in $B_{2}$},
\ee
where  the $x$-dependent  operator $L_{K_{\e,\d}}(x,\cdot)$ is given by 
\be
L_{K_{\e,\d}}(x, u)= \frac{1}{2}\int_{\R^N}(2u(x)-u(x+y)-u(x-y))J_{e,K_{\e,\d}}(x;y)\, dy
\ee
   and 
\begin{align*}
 F_{\e, \d}(x)&:=  f_{1,\e}(x)  + \frac{1}{2}\int_{S^{N-1}}\int_{0}^\infty(  w_{\e,\d}(x+r\th )- w_{\e,\d}(x-r\th )) r^{-1-2s}\ti\l_{o, {K_{\e,\d}} }(x,r,\th)\, dr d\th.
\end{align*}
We observe that  $\ti\l_{o, {K_{\e,\d}} }(x,r,\th)=(1-\eta_\d(r))\ti \l_{o, {K_{\e}} }(x,r,\th)$.  Hence by  \eqref{eq:til-o-x1-x-2} and the fact that $w_{\e,\d}\in C^{ 2s-\varrho}(B_{1})\cap C^{s-\varrho}(\R^N)$, then  provided $\varrho<\a$, there exists $\a_2>0$, depending only on $s,\a$ and $\varrho$  such that
$$
  \|  F_{\e, \d}\|_	{C^{ \b}(B_{1/2})}\leq C ( \| f_{1,\e}\|_	{C^{ \b}(B_2)} + \| w_{\e,\d}\|_{C^{ 2s-\varrho}(B_{1})}+   \| w_{\e,\d}\|_{C^{s-\varrho}(\R^N)}   ),
$$
for all $\b\in (0,\a_2)$.
Hence by    \eqref{eq:-estim-2s--eps-NMC-R-N} and \eqref{eq:-estim-2s--eps-NMC}, 
\be \label{eq:norm-F-eps-delt}
  \|  F_{\e, \d}\|_	{C^{ \b}(B_{1/2})}\leq C     \| f_{1,\e}\|_	{C^{ \b}(B_2)}  .
\ee
It is easy to see, from \eqref{eq:ti-lambda-e-delta} and our assumption  that,  for very $x_1,x_2\in \R^N$, 
\begin{align*}
\int_{B_{2\rho}\setminus B_\rho}&|J_{e, {K_{\e,\d}}  }(x_1;y)- J_{e, {K_{\e,\d}}  }(x_2;y)|\, dy=\\
&\int_{S^{N-1}}\int_{\rho}^{2\rho}  |\ti\l_{e, {K_{\e,\d}} }(x_1,r,\th)- \ti\l_{e, {K_{\e,\d}} }(x_2,r,\th)| r^{-1-2s}\, drd\th \leq C {|x_1-x_2|^\a}\rho^{-2s}.
\end{align*}
We now  apply  \cite[Theorem 1.1]{Serra} to the equation \eqref{eq:L-b-delt-w-d} and use \eqref{eq:norm-F-eps-delt} together with \eqref{eq:-estim-2s--eps-NMC-R-N},  to deduce that,  there exists  $\ov\a\in (0, \min(\a_0,\a_1,\a_2) )$, independent on $\d$, $v$ and $f$, such that   
$$
\|w_{\e,\d}\|_{C^{2s+ \b}(B_{1/4})}\leq C(   \|  F_{\e, \d}\|_	{C^{ \b}(B_1)}+   \|  w_{\e, \d}\|_	{C^\b(\R^N)}  )\leq  C   \| f_{1,\e}\|_	{C^{ \b}(B_2)} , 
$$
whenever $\b\in (0,\ov \a)$, with $2s+\b\not\in\N$.  
In view of \eqref{eq:estime-f_1-epsilon} and recalling that $w_{\e, \d}\to v_\e$  in $C^{s-2\varrho}(\R^N)\cap  C^{2s-2\varrho}(\ov{B_{1/2}})$, then   decreasing $\ov \a$ if necessary, we can send $\d\to 0$ and  get 
$$
\|v_\e \|_{C^{2s+ \b}(B_{1/4})}\leq    C(   \| f_\e \|_	{C^{ \b}(B_8)} +   \|  v_\e\|_	{C^\b(\R^N)}  ) ,
$$
provided $2s+\b<2$ ---noting that  if $2s>1$ then choosing $\varrho$ small, we have $\n w_{\e,\d} \to \n v_\e$ pointwise on $B_{1/4}$.
Scaling and translating back, we have thus proved  that for every $x_0\in \ov{B_{1}}$, there exists $\d_{x_0}\in (0,1)$ such that 
$$
\|v \|_{C^{2s+ \b}(B_{\d_{x_0}}(x_0))}\leq    C(x_0)(   \| f \|_	{C^{ \b}(B_2)} +   \|  v\|_	{C^\b(\R^N)}  ) ,
$$
with $C(x_0)$ is a constant depending only on $x_0,N,s,\a,\b$ and $ \k$.  Now by a  covering argument as in the proof of Theorem \ref{th:int--Holder-reg-Morrey}, we get the desired estimate.
%


\end{proof}

\subsection{Proof of Theorem \ref{th:bdr-reg-C11-domain}}\label{ss:proof-bdr-reg-change-var}
We start with the following result which provides a global diffeomorphism that locally flattens the boundary of $\de\O$ near the origin.
 \begin{lemma}\label{lem:glob-diff}
 Let $\O$ be an open set with boundary of class $C^{k,\g}$, for some $k\geq 1$ and $\g\in [0,1]$. Suppose that $0\in \de\O$ and that the interior unit normal of $\de \O$ at $0$  coincides with $e_N$. Then there exists $\rho_0>0$ such that for every $\rho\in (0,\rho_0)$, there exists   a (global) diffeomorphism $\Phi_\rho : \R^N\to \R^N$ with the following properties
 \begin{itemize}
  \item $D\Phi_\rho-id \in {C^{k,\g}_c(B_{2\rho};\R^N)}$,
 \item $\Phi_\rho(0)=0$ and   $D\Upsilon_\rho(0)=id $,
 \item  $\|D\Upsilon_\rho-id \|_{L^\infty(\R^N)}\to  0$, { as $\rho\to0$},
 \item  $\Phi_\rho(B_{r}',x_N)\subset \O$  if and only if  $x_N\in (0, \rho)$,
\item   $\Phi_\rho(B_{\rho}',0)\subset \de\O$,
 \item  the distance function to $\O^c$, satisfies $d(\Phi_\rho(x))=x_N,$ for all $x\in B_\rho^+$,
  \item  $\Phi_\rho$ is volume preserving in $\R^N$, i.e.   $\textrm{Det} D\Phi_\rho(x)=1$ for every $x\in \R^N$.
 \end{itemize}

 Here $B_\rho'$ denotes the ball in $\R^{N-1}$ centered at $0$, with radius $\rho>0$.
 \end{lemma}
 \begin{proof}
We consider  $\rho_1>0$ and a $C^{k,\g}$-function $\phi: B_{\rho_1}'\to \R$, $\phi(0)=0$ such that its graph  $x'\mapsto (x',\phi(x))\in \de \O$ is a parameterization of   a neighborhood of $0$ in $\de\O$.   Since the normal of $\de\O$ at 0 coincides with $e_N$ and $\O$ is of class $C^1$, we get 
\be\label{eq:nabla-phi-global-param}
\|\n \phi\|_{L^\infty(B_\rho)}\to 0\qquad\textrm{as $\rho\to 0$}. 
\ee
Moreover decreasing $\rho_1$, if necessary,   we have
$$
 (x',x_N+\phi(x'))\in \O \qquad\textrm{ for all $x_N\in (0, \rho_1)$.}
 $$
 Let $\eta\in C^\infty_c(B_1')$, with $\eta\equiv 1 $ on $B_{1/2}'$. 
For $\rho\in (0, \rho_1)$, we consider $\eta_\rho\in C^\infty_c(B_{\rho}')$ given by $\eta_\rho(x)=\eta(x/\rho)$. We now define  $\Phi_\rho:\R^N\to \R^N$,  by 
$$
\Phi_\rho(x',x_N):= (x',x_N+\eta_\rho(x')\phi(x')).
$$
 By \eqref{eq:nabla-phi-global-param}, we have 
$$
 \| D \Phi_\rho-id\|_{L^\infty(\R^N)}\leq  C \rho^{-1} \|\phi\|_{L^\infty(B_\rho)}+ C \|\n \phi\|_{L^\infty(B_\rho)}\to 0\qquad\textrm{as $\rho\to 0$}.
$$
 This  implies that there exists $\rho_0>0$ such that for every $\rho\in (0, \rho_0)$ and  for every $x\in \R^N$, the  Jacboian   of $\Phi_{\rho}$ at $x$,  $\textrm{Det} D \Phi_{\rho}(x)=1$. Since,   $\lim_{|x|\to \infty} |\Phi_{\rho}(x)|\to +\infty $, it follows from  Hadamard's Global Inversion Theorem (see e.g. \cite{Gordon}) that $\Phi_{\rho}$  is a global diffeomorphism, for every $\rho\in (0, \rho_0)$. Clearly $\Phi_\rho$ satisfies all properties stated in the lemma.
 \end{proof}
 %
 %
 \begin{proof}[Proof of Theorem \ref{th:bdr-reg-C11-domain}  (completed)]

 We assume that the interior unit normal of $\de \O$ at $0$  coincides with $e_N$.
 Consider $\Phi_\rho\in C^{1,1}(\R^N;\R^N)$,    given by Lemma \ref{lem:glob-diff}. In the following, we fix   $\rho>0$  small, so that
\be\label{eq:estim-Upsilon-rho-14}
\sup_{x\in \R^N}| D\Phi_\rho(x  )-id| <\frac{1}{4}  .
\ee 
By Theorem \ref{th:bound-reg-alpha},  there exists $C>0$,  only depending on $N,s,\O,\b,\L,\g,\d$ and $\|V\|_{\cM_\b}$, such that       
\be \label{eq:L-infty-u-proof-bdr-reg}
    \|u\|_{L^\infty(B_{\rho })}\leq C(\|u\|_{L^2(B_{2 \rho})}+ \|u\|_{\cL^1_s}+ \|f\|_{\cM_\b}).
\ee
We then have that $\cL_K u=f-uV$ on $\O$. 
Letting  $U(x)=u(\Phi_\rho(x))$ and $F(x)= f(\Phi_\rho(x)) - U(x)V(\Phi_\rho(x))$, then by a change of variable,  we have 
\begin{align*}
\frac{1}{2}\int_{\R^{2N}}(U(x)-U(y))(\psi(x)-\psi(y))  K(\Phi_\rho(x), \Phi_\rho(y))dxdy=\int_{\R^N}F(x)\psi(x)\, dx,
\end{align*}
for every $\psi\in C^\infty_c(B_{\rho/2}^+)$. Therefore, $U\in H^s (B_{\rho})\cap \cL^1_s$, 
$$
\cL_{K_\rho}U=  F \qquad\textrm{ on  $B_{\rho/2}^+$}, \qquad\textrm{ and } \qquad U=0 \qquad\textrm{ on $B_{\rho/2}^-$},
$$
where $K_\rho(x,y):= K(\Phi_\rho(x), \Phi_\rho(y)) $,   $B_\rho^-:=B_\rho\cap \{x_N<0\}$ and  $B_\rho^+:=B_\rho\cap \{x_N>0\}$.
In particular, we have 
\be\label{eq:capital-U-solves}
\cL_{K_\rho}U=  \vp_{\rho/4 } F \qquad\textrm{ on  $B_{\rho/4}^+$}, \qquad\textrm{ and } \qquad U=0 \qquad\textrm{ on $B_{\rho/4}^-$},
\ee 
and we note that by \eqref{eq:L-infty-u-proof-bdr-reg}, $  \vp_{\rho/4 } F\in \cM_\b$. 
 We observe that 
\begin{align*}
K_\rho(x,x+y)&= K(\Phi_\rho(x), \Phi_\rho(x+y))\\
&= K\left(\Phi_\rho(x),\Phi_\rho(x)+ [ \Phi_\rho(x+y) -\Phi_\rho(x)] \right).
%
\end{align*}
 We   define  $A\in C^{0,1}(\R^N\times[0, \infty) \times S^{N-1};\R^N) $, by 
$$
 A(x,r,\th):=  \int_0^1D\Phi_\rho(x+ r \tau \th )\th \, d\tau,
$$
so that, for $r>0$, 
$$
K_\rho(x,x+r\th )=  K\left(\Phi_\rho(x),\Phi_\rho(x)+ r A(x,r,\th)\right).
$$
Now  by  \eqref{eq:estim-Upsilon-rho-14}, we can find  constants   $C,C'>0$ such that, for every $x_1,x_2\in B_2$, $r_1,r_2\in [0,2)$ and $\th\in S^{N-1}$, 
\be\label{eq:Holder-A-est-tiK-rho}
|A(x_1,r_1,\th)-A(x_2,r_2,\th) |\leq C (|x_1-x_2| + |r_1-r_2| )
\ee
and 
\be\label{eq:Holder-A-est-tiK-rho-invers}
\frac{1}{|A(x_1,r_1,\th)|}  \geq  C'  .
\ee
Since $ \ti \l_K\in C^{s+\d}( B_2\times[0, 2) \times S^{N-1} )$,  we  then have  that $ K_\rho\in \ti \scrK(\k')$ and satisfies  \eqref{eq:Kernel-satisf}, for some $\k'>0$, only depending on $\O,N,s$ and $\k$, with  
$$
\ti \l_{K_\rho}(x,r,\th)=  |  A(x,r,\th) |^{-N-2s}\ti \l_K\left(\Phi_\rho(x), r | A(x,r,\th) |,\frac{ A(x,r,\th) }{| A(x,r,\th) |} \right). 
$$
By \eqref{eq:Holder-A-est-tiK-rho} and \eqref{eq:Holder-A-est-tiK-rho-invers}, it is clear that  
\be\label{eq:est-tiK-rho}
|\ti \l_{K_\rho}(x_1,r_1,\th)-\ti \l_{K_\rho}(x_2,r_2,\th) |\leq C (|x_1-x_2|^{s+\d}+ |r_1-r_2|^{s+\d}).
\ee
We note that by assumption, 
\begin{align*}
\left| \ti \l_K\left(\Phi_\rho(x), r | A(x,r,\th) |,\frac{ A(x,r,\th) }{| A(x,r,\th) |} \right) \right.&\left.- \ti \l_K\left(\Phi_\rho(x), r | A(x,r,\th) |,-\frac{ A(x,r,\th) }{| A(x,r,\th) |} \right) \right|\\
&\leq c_0 r^{s+\d}   | A(x,r,\th) | \leq  C r^{s+\d} . 
\end{align*}
Next, we put $$G(x,r,\th):=\ti \l_K\left(\Phi_\rho(x), r | A(x,r,\th) |,\frac{A( x,r,\th) }{| A(x,r,\th) |} \right).$$  Then, using the above estimate, \eqref {eq:Holder-A-est-tiK-rho} and \eqref{eq:Holder-A-est-tiK-rho-invers}, we get 
\begin{align*}
&|G(x,r,\th)-G(x,r,-\th)|\\
&\leq \left| \ti \l_K\left(\Phi_\rho(x), r | A(x,r,\th) |,\frac{ A(x,r,\th) }{| A(x,r,\th) |} \right)- \ti \l_K\left(\Phi_\rho(x), r | A(x,r,\th) |,-\frac{ A(x,r,\th) }{| A(x,r,\th) |} \right) \right|\\
&+ \left| \ti \l_K\left(\Phi_\rho(x), r | A(x,r,\th) |,- \frac{ A(x,r,\th) }{| A(x,r,\th) |} \right)- \ti \l_K\left(\Phi_\rho(x), r | A(x,r,-\th) |,-\frac{ A(x,r,\th) }{| A(x,r,-\th) |} \right) \right|\\
&+  \left|\ti \l_K\left(\Phi_\rho(x), r | A(x,r,-\th) |,-\frac{ A(x,r,\th) }{| A(x,r,-\th) |}\right)-  \ti \l_K\left(\Phi_\rho(x), r | A(x,r,-\th) |,\frac{ A(x,r,-\th) }{| A(x,r,-\th) |} \right) \right|\\
%
%
&\leq C r^{s+\d} + C \left|  \frac{ 1 }{| A(x,r,\th) |}-  \frac{ 1}{| A(x,r,-\th) |}  \right|^{s+\d}  + C |  { A(x,r,\th) } +  { A(x,r,-\th) }  |^{s+\d}\\
&\leq C r^{s+\d},
\end{align*}
with $C$ depends only on $\O,N,s,$ and $\d$. From this, \eqref{eq:Holder-A-est-tiK-rho} and \eqref{eq:Holder-A-est-tiK-rho-invers}, we easily deduce that
\be\label{eq:est-tiK-rho-oo}
|\ti \l_{K_\rho}(x,r,\th)-\ti \l_{K_\rho}(x,r,-\th) |\leq C( r^{s+\d} +r).
\ee
We put  $d_+(x)=\max(x_N,0)$.  Using \eqref{eq:decomp-weak-sol}, for $\psi\in C^\infty_c(B_2^+)$, we then have
\begin{align*}
\frac{1}{2}\int_{\R^{2N}}& (d_+^s(x)-d_+^s(y)) (\psi(x)-\psi(y))K_\rho(x,y)\, dxdy\\
%
%
%
&=\int_{\R^N}\psi(x)\int_{S^{N-1}}\int_0^\infty (d_+^s(x)-d_+^s(x+r\th ))r^{-1-2s} \ti \l_{e, K_\rho}(x,0,\th)\,drd\theta\, dx\\
&+ \int_{\R^N}\psi(x)\int_{S^{N-1}}\int_0^\infty (d_+^s(x)-d_+^s(x+r\th ))r^{-1-2s} (\ti \l_{e, K_\rho}(x,r,\th)-\ti \l_{e, K_\rho}(x,0,\th))\,drd\theta\, dx\\
&+ \int_{\R^N}\psi(x)\int_{S^{N-1}}\int_0^\infty (d_+^s(x )-d_+^s(x+r\th )) r^{-1-2s} \ti \l_{o, K_\rho}(x,r,\th)\,drd\theta\, dx.
\end{align*}
Using the fact that $\Ds_1 d^s_+=0$  on  $\{x_N>0\}$, we see that  
$$
 \int_{S^{N-1}}\int_0^\infty (d_+^s(x)-d_+^s(x+r\th ))r^{-1-2s} \ti \l_{e, K_\rho}(x,0,\th)\,drd\theta=0.
$$
It follows that 
\be \label{eq:d-+-S-solves}
\cL_{K_\rho} d_+^s=F_e+F_o \qquad\textrm{ in $B_2^+$},
\ee
where 
$$
F_e(x):=\int_{S^{N-1}}\int_0^\infty (d_+^s(x)-d_+^s(x+r\th ))r^{-1-2s} (\ti \l_{e, K_\rho}(x,r,\th)-\ti \l_{e, K_\rho}(x,0,\th))\,drd\theta
$$
and
$$
F_o(x):=\int_{S^{N-1}}\int_0^\infty (d_+^s(x )-d_+^s(x+r\th )) r^{-1-2s} \ti \l_{o, K_\rho}(x,r,\th)\,drd\theta.
$$ 
Now from \eqref{eq:est-tiK-rho} and \eqref{eq:est-tiK-rho-oo}, we obtain   
$$
 | \ti \l_{e, K_\rho}(x,r,\th)-\ti \l_{e, K_\rho}(x,0,\th)|+|  \ti \l_{o, K_\rho}(x,r,\th)|\leq C \min (r^{s+\d}+r,1) .
 $$
This with the fact that $|d_+^s(x)-d_+^s(x+r\th )|\leq C r^s$ imply that 
$$
\|F_e\|_{L^\infty(B_{1})}+ \|F_o\|_{L^\infty(B_{1})}\leq C.
$$
In view of Lemma \ref{lem:cat-off-sol} and \eqref{eq:d-+-S-solves}, we find that 
$$
\cL_{K_\rho} (\vp_{1}d_+^s )= g_\rho \qquad\textrm{ in $B_{1/2}^+$},
$$
with $\|g_\rho\|_{L^\infty(B_{1})} \leq C.$  By this,   \eqref{eq:capital-U-solves}, \eqref{eq:est-tiK-rho} and  \eqref{eq:est-tiK-rho-oo},   we can thus apply Corollary \ref{cor:mainth-gen}, to get 
$$
 \|U/ d^s_+\|_{C^{s-\b} \left( \ov {B_{\varrho}^+}\right)}\leq C\left(\|U\|_{L^2(B_{\rho })}+ \|U\|_{\cL^1_s}+ \|F\|_{\cM_\b}  \right),
 $$
for some $  C>0$ and $\varrho\in (0, \rho/4)$, only depending on $N,s,\O,\d,\b,\k,\rho$.  Since $d(\Phi_\rho(x))=x_N$ on $B_{\rho}^+$, then by a change of variable and using \eqref{eq:L-infty-u-proof-bdr-reg},    we get the desired result.
  \end{proof}
 
\section{Appendix 1: Liouville theorems}\label{s:append-Liouville-them}
In this section we consider $\cH$ being either $\R^N$ or the half-space  $\R^N_+=\{x\in \R^N\,:\,x_N>0\}$.  We  prove a classification result for all functions $u\in H^s_{loc}(\R^N)\cap\cL^1_s$ satisfying  $L_b u=0$ in $\cH$ and $u=0$ in $\cH^c$, provided $b$ is a weak limit of $a_n$ satisying \eqref{eq:def-a-anisotropi} and $u$ satisfying some growth conditions. We note that in the case $u\in L^\infty_{loc}(\R^N)$ such classification results (for more general nonolcal operators $L_b$) are proved in \cite{RS2}. 
 We will need the following result for the proof of the Liouville theorems.
  \begin{lemma}\label{lem:L-infty-bounds-harmonic-coerciv}
  Let $\O$ be an open set with $0\in \de\O$ and   let $a$ satisfy \eqref{eq:def-a-anisotropi}.  We consider $u\in H^s_{loc}(\R^N)\cap \cL^1_s$  satisfying 
\begin{align*}
\Ds_a u=0  \qquad \textrm{ in $ B_2\cap\O$}, \qquad\qquad  u=0 \qquad\textrm{ in $B_2\cap  \O^c$}.
%
\end{align*}
Then  there exists $C=C(N,s,\L,\O)>0$ such that
$$
\|u\|_{L^\infty(B_1\cap\O)}\leq C\left(  \|u\|_{\cL^1_s}+ \|u\|_{L^2(B_2)}\right).
$$ 
  \end{lemma}
  \begin{proof}
The interior $L^\infty_{loc}(B_2\cap\O)$ estimate follows from  \cite{Dicastro}, where the authors used the De Giorgi iteration argument.  We note that in \cite{Dicastro}, it is assumed that $u\in H^s(\R^N)$  but by carefully  looking at their arguments, we see that   this  can be weakened to $ u\in H^s_{loc}(\R^N)\cap \cL^1_s$. 
 For the $L^\infty(B_1\cap\O)$  estimate, the proof is precisely the same. 
  \end{proof}
 In the following,  for $b$ satisfying \eqref{eq:coerciv-RS} and $f\in H^s(\R^N)$, we put  
$$
[f]_{H^s_b(\R^N)}:=\left(\int_{\R^{2N}}(f(x)-f(y))^2\mu_b(x,y)\,dxdy \right)^{1/2}.
$$ 
Recall the Poincar\'e-type inequality related to this seminorm, see \cite{Ros-Real,Geisinger},  
\be \label{eq:Poincare-gen-stab-rpocess}
C \|f\|_{L^2(A )} \leq [f]_{H^s_b(\R^N)} \qquad\textrm{ for every $f\in H^s(\R^N)$, with $f=0$ on $A^c$,}
\ee
whenever $|A|<\infty$. Here  $C$ is positive constant, only depending on $N,s,b$ and $A$. We note that if $b$ satisfies \eqref{eq:def-a-anisotropi},  then the constant $C$ in \eqref{eq:Poincare-gen-stab-rpocess} can be chosen to depend only on $N,s,\L$ and $A$.
 \begin{lemma}\label{lem:L-infty-bounds-harmonic}
  Let $\O$ be an open set with $0\in \de\O$. Suppose that there exists a sequence of functions $a_n$ satisfying \eqref{eq:def-a-anisotropi} and $a_n\stackrel{*}{\rightharpoonup} b$ in $L^\infty(S^{N-1})$.
We consider $u\in H^s_{loc}(\R^N)\cap \cL^1_s$  satisfying 
\begin{align*}
L_b u=0  \qquad \textrm{ in $ B_2\cap\O$} \qquad \textrm{ and} \qquad  u=0 \qquad\textrm{ in $B_2\cap  \O^c$}.
%
\end{align*}
Then  there exists $C=C(N,s,\L,\O)>0$ such that
$$
\|u\|_{L^\infty(B_1\cap\O)}\leq C\left(  \|u\|_{\cL^1_s}+ \|u\|_{L^2(B_2)}\right).
$$ 
  \end{lemma}
  \begin{proof}
%
Let $M\geq 1989$, so that   $\vp_M u\in H^s(\R^N)$.  We let $v_{n,M}\in H^s(\R^N) $ be the (unique) solution to 
\begin{align}\label{eq:u-stf-grow-harmmm}
\Ds_{a_n} v_{n,M}=0  \qquad \textrm{ in $ B_2\cap\O$} \qquad \textrm{ and} \qquad  v_{n,M}=\vp_M u \qquad\textrm{ in $(B_2\cap  \O)^c$}.
\end{align}
By Lemma \ref{lem:L-infty-bounds-harmonic}, 
\be \label{eq:L-infty-bound-vnk}
\|v_{n,M}\|_{L^\infty(B_1\cap\O)}\leq C\left(  \|v_{n,M}\|_{\cL^1_s}+ \|v_{n,M}\|_{L^2(B_2)}\right).
\ee
Moreover, 
$$
\Ds_{a_n} (v_{n,M}- \vp_M u)=-\Ds_{a_n}(\vp_M u)  \qquad \textrm{ in $ B_2\cap\O$} \qquad \textrm{ and} \qquad  v_{n,M}- \vp_M u=0 \qquad\textrm{ in $(B_2\cap  \O)^c$}.
$$
Multiplying this with $ v_{n,M}- \vp_M u$, integrating and using H\"older's inequality, we deduce that
\begin{align*}
[v_{n,M}- \vp_M u]^2_{H^s_{a_n} (\R^N)}\leq [v_{n,M}- \vp_M u]_{H^s_{a_n}(\R^N)}[  \vp_M u]_{H^s_{a_n}(\R^N)}.
\end{align*}
It follows that $ [v_{n,M}- \vp_M u]_{H^s_{a_n}(\R^N)}\leq   [  \vp_M u]_{H^s_{a_n}(\R^N)}$ and thus by \eqref{eq:Poincare-gen-stab-rpocess} and  \eqref{eq:def-a-anisotropi}, we deduce that the sequence  $ (v_{n,M})_n$ is bounded in $H^s(\R^N)$. Hence, up to a subsequence,  $(v_{n,M})_n$ converges in $L^2(B_2\cap \O)$ and  in $\cL ^1_s$ to some function $v_M$. Passing to the limit in \eqref{eq:u-stf-grow-harmmm} as $n\to\infty$ and using Lemma \ref{lem:Ds-a-n--to-La}, we find that 
\begin{align}\label{eq:u-stf-grow-harmmm0}
L_{b} v_{M}=0   \qquad \textrm{ in $ B_2\cap\O$} \qquad \textrm{ and} \qquad  v_{M}=\vp_M u \qquad\textrm{ in $(B_2\cap  \O)^c$},
\end{align}
  Moreover passing to the limit in \eqref{eq:L-infty-bound-vnk}, we get
\be \label{eq:L-infty-bound-vM}
\|v_{M}\|_{L^\infty(B_1\cap\O)}\leq C\left(  \|v_{M}\|_{\cL^1_s}+ \|v_{M}\|_{L^2(B_2)}\right).
\ee
Next, letting $w_M:=v_{M}- \vp_M u$, using  Lemma \ref{lem:cat-off-sol}, we find that    
\begin{align}\label{eq:u-stf-grow-harmmm01}
L_{b} w_M=- L_b( \vp_M u)= -G_M   \qquad \textrm{ in $ B_2\cap\O$} \qquad \textrm{ and }\qquad  w_{M}=0\qquad\textrm{ in $(B_2\cap  \O)^c$},
%
\end{align}
with  $\|G_M\|_{L^\infty(B_2)}\leq C\int_{|y|\geq M}|u(y)||y|^{-N-2s}\,dy$. Here, $C$ is a positive constant only depending on $N,s$ and $\L$. 
Multiplying the first equation in \eqref{eq:u-stf-grow-harmmm01} by $w_M$,  integrating and using \eqref{eq:Poincare-gen-stab-rpocess}, we obtain 
\begin{align*}
[w_M]_{H^s_b(\R^N)}^2&\leq  \|G_M\|_{L^\infty(B_2)}\int_{B_{2}\cap\O}  |w_M(x)|\, dx\\
& \leq  C\int_{|y|\geq M}|u(y)||y|^{-N-2s}\,dy     [w_M]_{H^s_b(\R^N)},
\end{align*}
where $C=C(s,N,\L,\O,b)$.
 This then implies that 
 \begin{align*}
[w_M]_{H^s_b(\R^N)}&\leq   C\int_{|y|\geq M}|u(y)||y|^{-N-2s}\,dy      .
\end{align*}
We then  deduce that $w_M\to 0 $ in $L^2(\R^N)$ as $M\to \infty$, by \eqref{eq:Poincare-gen-stab-rpocess}.   In addition, we have 
\begin{align*}
\|v_M-u\|_{\cL^1_s}\leq \|v_M-\vp_Mu\|_{\cL^1_s}+ \|(1-\vp_M) u\|_{\cL^1_s}=\int_{B_2\cap\O}\frac{|w_M(y)|}{1+|y|^{N+2s}} \,dy+ \|(1-\vp_M) u\|_{\cL^1_s}.
\end{align*}
We  conclude that $v_M\to u $ in $\cL^1_s$ as $M\to \infty$. Letting $M\to\infty$ in \eqref{eq:L-infty-bound-vM},  we get
$$
\|u\|_{L^\infty(B_1\cap\O)}\leq C\left(  \|u\|_{\cL^1_s}+ \|u\|_{L^2(B_2)}\right).
$$
  \end{proof}
\begin{lemma}\label{lem:Liouville}
Suppose that there exists a sequence of functions $a_n$ satisfying \eqref{eq:def-a-anisotropi} and $a_n\stackrel{*}{\rightharpoonup} b$ in $L^\infty(S^{N-1})$.
 We consider $u\in H^s_{loc}(\R^N)$  satisfying 
\begin{align}\label{eq:u-stf-grow-harm}
\begin{cases}
L_b u=0  \qquad \textrm{ in $\cH$}, \qquad\qquad  u=0 \qquad\textrm{ in $\R^N\setminus \cH$},\vspace{3mm}\\
\|u\|_{L^2(B_R)}^2 \leq  R^{N+2\g} \qquad \textrm{ for  some $\g<2s$ and for every $R\geq 1$}.
\end{cases}
\end{align}
Then $u$ is an affine function if  $\cH=\R^N$, while $u$ is proportional to $\max(x_N,0)^s$ if $\cH=\R^N_+$.
  \end{lemma}
  \begin{proof}
%
%
%
 We put $v_R(z)=  u(R z )$, for $R\geq 1$ and $z\in \R^N$. Since $L_b v_R=0$ in $\cH$ and $v_R=0$ on $\cH^c$,   by Lemma \ref{lem:L-infty-bounds-harmonic},    
$$
\|v_R\|_{L^\infty(B_1)}\leq C\left(  \|v_R\|_{\cL^1_s}+ \|v_R\|_{L^2(B_2)}\right).
$$ 
Scaling back, we get
$$
\|u\|_{L^\infty(B_R)}\leq C\left(  R^{2s} \int_{\R^N}\frac{|u(x)|}{R^{N+2s}+|x|^{N+2s}}\, dx + 2^\g  R^\g \right).
$$
 Now, using H\"older's inequality and   \eqref{eq:u-stf-grow-harm}, we get $\|u\|_{L^1(B_R)} \leq  C R^{N+\g} $. We then  have
 \begin{align*}
 \int_{\R^N}\frac{|u(x)|}{R^{N+2s}+|x|^{N+2s}}\, dx&\leq  R^{-N-2s} \int_{B_R}{|u(x)|} \, dx+  \int_{|x|\geq R}{|u(x)|}{ |x|^{-N-2s}}\, dx\\
 &\leq  C R^{-N-2s} \int_{B_R}{|u(x)|} \, dx+  \sum_{i=0}^\infty  \int_{2^i R \leq |x|\leq 2^{i+1}R}{|u(x)|}{ |x|^{-N-2s}}\, dx\\
 &\leq C R^{-2s+\g} + C  R^{-2s+\g}    \sum_{i=0}^\infty 2^{-(N+2s)i} 2^{(i+1)(N+\g)}  \\
 &\leq C(1+  \sum_{i=0}^\infty2^{-i(2s-\g)} )  R^{-2s+\g}\leq C  R^{-2s+\g}    .
 \end{align*}
 We  deduce that 
$$
\|u\|_{L^\infty(B_R)} \leq C R^{\g}\qquad\textrm{ for all $R\geq 1 $}.
$$
It follows from the Liouville theorems in   \cite{RS3},  that $u$ is an affine function if $\cH=\R^N$, while $u$ is proportional to $\max(x_N,0)^s$ when $\cH=\R^N_+$.
\end{proof}
\section{Appendix 2: Some technical results}\label{s:appnd-2-techinacality}
The following result is a Caccioppoli type inequality, see e.g. \cite{Dicastro,Kassmann,Franz} for other versions. Note in the following lemma that $K$ could be any nonegative and nontrivial  symmetric function on $\R^N\times\R^N$.
\begin{lemma} \label{lem:caciopp}
Let   $R>0$ and $K$ satisfy \eqref{eq:Kernel-satisf}. Let  $v\in H^s_{loc}(B_{2R})\cap \cL^1_s $ and $f\in L^1_{loc}(\R^N)$ satisfy  
\be\label{eq:Dsv-eq-f-Liou}
\cL_{K}  v= f \qquad\textrm{ in $  B_{2R}\cap\O$},\qquad\textrm{ and }\qquad v=0 \qquad\textrm{ in  $  B_{2R}\cap\O^c$. }
\ee
 Then for every $\e>0$ and   for every $\vp \in C^\infty_c( B_{2R})$, we have
\begin{align*}
(1-\e)\int_{\R^{2N}}(v(x)-v(y))^2\vp^2(y) K(x,y)\,dy dx\leq & \int_{\R^N} |f (  x)| |v(x)|  \vp^2(x)\, dx\\
&+   \e^{-1}   \int_{\R^{2N}}v^2(x)  (\vp(x)-\vp(y))^2K(x,y)\, dy dx.  
\end{align*}
\end{lemma}
\begin{proof}
Direct computations give 
\begin{align*}
(v(x)-v(y))[v(x)\vp^2(x)-v(y)\vp^2(y)]&=  (v(x)-v(y))^2\vp^2(y)+ v(x)(v(x)-v(y))[\vp^2(x)-\vp^2(y)]\\
&=(v(x)-v(y))^2\vp^2(y)+ \vp(x)v(x)(v(x)-v(y))(\vp(x)-\vp(y))\\
& + \vp(y)v(x)(v(x)-v(y))(\vp(x)-\vp(y)).
\end{align*}
Testing the equation \eqref{eq:Dsv-eq-f-Liou} with $v\vp^2$,   and use the identity above together with the symmetry of $K$,  to  get
\begin{align}\label{eq:before-cacciop}
\frac12\int_{\R^{2N}}(v(x)-v(y))^2\vp^2(y) K(x,y)\,dy dx&=-\frac12 \int_{\R^{2N}} \vp(x)v(x)(v(x)-v(y))(\vp(x)-\vp(y))K(x,y)\, dydx   \nonumber\\
&\quad- \frac12\int_{\R^{2N}} \vp(y)v(x)(v(x)-v(y))(\vp(x)-\vp(y))K(x,y)\, dxdy\nonumber\\
& \quad +\int_{\R^N} f (  x) v(x)  \vp^2(x)\, dx. 
\end{align}
By H\"older and Young's inequalities, we get 
\begin{align*}
&\left|\int_{\R^{2N}} \vp(y)v(x)(v(x)-v(y))(\vp(x)-\vp(y)) K(x,y)dydx  \right|\\
&\leq \int_{\R^N} |v(x)|\left( \int_{\R^N} \vp^2(y) (v(x)-v(y))^2 K(x,y)\, dy\right)^{1/2}\left( \int_{\R^N}  (\vp(x)-\vp(y))^2 K(x,y)\, dy\right)^{1/2}\,dx\\
&\leq  \left(  \int_{\R^{2N}} \vp^2(y) (v(x)-v(y))^2 K(x,y)\, dydx\right)^{1/2}\left( \int_{\R^{2N}}v^2(x)  (\vp(x)-\vp(y))^2 K(x,y)\, dy dx\right)^{1/2}\\
&\leq \e   \int_{\R^{2N}} \vp^2(y) (v(x)-v(y))^2K(x,y)\, dy  dx +\e^{-1}   \int_{\R^{2N}}v^2(x)  (\vp(x)-\vp(y))^2 K(x,y)\, dy dx .
\end{align*}
By similar arguments, we also have 
\begin{align*}
&\left|\int_{\R^{2N}} \vp(x)v(x)(v(x)-v(y))(\vp(x)-\vp(y)) K(x,y)dydx  \right|\\
%
%
%
&\leq \e   \int_{\R^{2N}} \vp^2(x) (v(x)-v(y))^2K(x,y)\, dy  dx +\e^{-1}    \int_{\R^{2N}}v^2(x)  (\vp(x)-\vp(y))^2 K(x,y)\, dy dx .
\end{align*}
Using the above two estimates above in \eqref{eq:before-cacciop}, we get the result.
%
\end{proof}

The following result provides a localization of solutions for nonlocal equations.
\begin{lemma} \label{lem:cat-off-sol}
Let $K $ satisfy \eqref{eq:Kernel-satisf} or $K=\mu_b$, for some $ b\in L^\infty(S^{N-1})$ with $b>0$.
Let $v\in H^s_{loc}(B_{2R})\cap \cL^1_s $  and $f\in L^1_{loc}(\R^N)$ satisfy
\be\label{eq:Dsv-eq-f}
\cL_{K}  v= f \qquad\textrm{ in $  B_{2R}\cap\O$}\qquad\textrm{ and }\qquad v=0 \qquad\textrm{ in  $  B_{2R}\cap\O^c$, }
\ee
for $R>0$. We let $v_R:=\vp_R v$. Then 
$$
\cL_{K} v_R  =  f+     G_{v,R} \qquad\textrm{in $B_{R/2}\cap \O$},
$$
where
$
G_{v,R}(x)=\vp_{R/2}(x)  \int_{\R^N} v(y)  (\vp_R(x)-\vp_R(y)) K(x,y) \, dy.
$
\end{lemma}
\begin{proof}
For simplicity, we assume that $K=\mu_b$ for some even function $b\in L^\infty(S^{N-1})$.
Recall  the identity, which follows from the symmetry of $K$,
$$
\cL_{K}  (v \vp_R)= v \cL_{K}  \vp_R+\vp_R\cL_{K} v- I(v,\vp_R), 
$$
with $I(v,\vp)(x) =\int_{\R^N}(v(x)-v(y))(\vp_R(x)-\vp_R(y)) K(x,y) \, dy$. Since $\vp_R\equiv 1$ on $B_R$,    we have 
\begin{align*}
I(v,\vp)(x)& =  \int_{|y|\geq R} (v(x)-v(y))(1-\vp_R(y)) K(x,y) \, dy\\
&=v(x) \int_{|y|\geq R}   (1-\vp_R(y)) K(x,y) \, dy-  \int_{|y|\geq R} v(y)  (1-\vp_R(y)) K(x,y) \, dy.
\end{align*}
 We note that 
$$
 - \vp_{R/2}(x)\cL_{K} \vp_R(x)+  \vp_{R/2}(x) \int_{|y|\geq R}   (1-\vp_R(y)) K(x,y) \, dy=0.
$$
Therefore letting
$$
  G_{v,R}(x):= \vp_{R/2}(x)  \int_{|y|\geq R} v(y)  (1-\vp_R(y)) K(x,y) \, dy,
$$
it follows that  
$$
\cL_{K}(\vp_R v)= f + G_{v,R} \qquad\textrm{in  $  B_{R/2}\cap\O$}.
$$
 The proof is thus finished.
\end{proof}
 We close this section with the following result.
 \begin{lemma}\label{lem:ds-in-Hs-loc}
Let $\O$ be $C^1$ open set with $0\in \de\O$ and such that the interior normal of $\de\O$ at $0$ coincides with $e_N$. Then $d^s\in H^s(B_r)$, for some $r>0$.
\end{lemma}
\begin{proof}
Let $d_+(x):=\max(x_N,0)$. Then since $\Ds_1 d_+^s(x)=0$ for every $x\in \R^N_+$, by  Lemma \ref{lem:cat-off-sol}, we have 
$$
\Ds_1 v^+_R(x)=G_R(x) \qquad\textrm{ for every $x\in B_R^+$},
$$
with  $v^+_R=\vp_{2R} d_+^s $ and  $G_R\in L^\infty(\R^N)$. Now multiplying the above equation by $  v^+_R \vp_{R/2}^2$ (supported in $B_R^+$), integrating on $B_R^+$ and using the symmetry of $\mu_1$, we see that 
$$
\frac{1}{2}\int_{\R^{2N}}(v^+_R(x)-v^+_R(x))(v^+_R \vp_{R/2}^2(x)-v^+_R \vp_{R/2}^2(y) )\mu_1(x,y)\,dxdy=\int_{\R^N}v^+_R(x) \vp_{R/2}^2(x) G_R(x)\,dx.
$$
We may now apply  Lemma  \ref{lem:caciopp} (or eventually following its proof),   to  deduce that 
$$
\int_{\R^{2N}}(v^+_R (x)- v^+_R(y))^2\vp^2_{R/2}(y) \mu_1(x,y)\,dy dx<\infty.
$$
This implies that $d^s_+\in H^s_{loc}(\R^N)$. To conclude, we  use the  parameterization $\Phi_\rho$ given by Lemma \ref{lem:glob-diff} and make changes of variables, to get 
\begin{align*}
\int_{\Phi_\rho( B_{\rho})} \int_{\Phi_\rho( B_{\rho})}  \frac{(d^s(x)-d^s(y))^2}{|x-y|^{N+2s}}\, dxdy& =\int_{ B_{\rho}}\int_{ B_{\rho}}    \frac{(d_+^s(x)-d_+^s(y))^2}{|\Phi_\rho( x)-\Phi_\rho( y)|^{N+2s}}\, dxdy\\
& \leq C \int_{ B_{\rho}}\int_{ B_{\rho}}    \frac{(d_+^s(x)-d_+^s(y))^2}{|x-y|^{N+2s}}\, dxdy<\infty ,
\end{align*}
provided $\rho$ is small enough.
The proof is thus finished.
\end{proof}

\end{document}